%% file: symplectice_hecke_insertion_draft4.tex
\documentclass[11pt]{article}

\input{preamble.tex}
\usepackage{fullpage}

\begin{document}
\title{A symplectic refinement of shifted Hecke insertion}
\author{
Eric Marberg
\\ Department of Mathematics \\  Hong Kong University of Science and Technology \\ {\tt eric.marberg@gmail.com}
}

\date{}

\maketitle

\begin{abstract}
Buch, Kresch, Shimozono, Tamvakis, and Yong defined Hecke insertion to formulate a combinatorial rule for the expansion of the stable Grothendieck polynomials $G_\pi$ indexed by permutations in the basis of stable Grothendieck polynomials $G_\lambda$ indexed by partitions. Patrias and Pylyavskyy introduced a shifted analogue of Hecke insertion whose natural domain is the set of maximal chains in a weak order on orbit closures of the orthogonal group acting on the complete flag variety. We construct a generalization of shifted Hecke insertion for maximal chains in an analogous weak order on orbit closures of the symplectic group. As an application, we identify a combinatorial rule for the expansion of ``orthogonal'' and ``symplectic'' shifted analogues of $G_\pi$ in Ikeda and Naruse's basis of $K$-theoretic Schur $P$-functions.
\end{abstract}


\setcounter{tocdepth}{2}

\section{Introduction}

\subsection{Hecke words}

Let $\SS$ denote the group of permutations of the positive integers $\PP := \{1,2,3,\dots\}$ with finite support.
Define $s_i = (i,i+1) \in \SS$ for $i \in \PP$ so that $\SS = \langle s_1,s_2,s_3,\dots\rangle.$
In examples, we write elements of $\SS$ in one-line notation
and identify the word $\pi_1\pi_2\cdots \pi_n$ with the permutation $\pi \in \SS$
that has $\pi(i) = \pi_i$ for $i \leq n$ and $\pi(i) =i$ for $i>n$.

Let $\Nil$ denote the free $\ZZ$-module with a basis given by the symbols 
$U_\pi$ for $\pi \in \SS$. Set $U_i  :=U_{s_i}$ for $i \in \PP$.
The abelian group $\Nil$
 has a unique ring structure
 with 
 \[ 
 U_\pi U_i = \begin{cases} 
 U_{\pi s_i} &\text{if }\pi(i) < \pi(i+1) \\ 
 U_{\pi} &\text{if }\pi(i) > \pi(i+1)
 \end{cases}
 \qquad\text{for }\pi \in \SS \text{ and }i \in \PP.
 \]
This is
the usual one-parameter Iwahori-Hecke algebra $\cH_q = \ZZ[q]\spanning\{T_\pi : \pi \in \SS\}$
of $\SS$
but with $q=0$ and $U_\pi = -T_\pi$; see \cite[Theorem 7.1]{Humphreys}.
We retain the following terminology from \cite{BKSTY}:

\begin{definition}
A \emph{Hecke word} for $\pi \in \SS$ is any word $i_1i_2\cdots i_l$
such that $U_\pi = U_{i_1}U_{i_2}\cdots U_{i_l}$.
\end{definition}

Let $\cH(\pi)$ be the set of Hecke words for $\pi \in \SS$.
The \emph{reduced words} for $\pi$ are its Hecke words of minimal length.
Let $\cR(\pi)$ denote the set of reduced words for $\pi \in \SS$.
If $\pi = 321 \in \SS$ then
\[\cR(\pi) = \{ 121,212\} \subset \{ 121,212,1121,1221,1211,1212, 2212,2112,2122,2121,\dots\} = \cH(\pi).\]
Let $\ell(\pi)$ denote the length of a permutation $\pi \in \SS$, given by the common length of every reduced word in $\cR(\pi)$ or,
equivalently, by the number of pairs $(i,j) \in \PP\times\PP$ with $i<j$ and $\pi(i) > \pi(j)$.
We also write $\ell(w)$ to denote the length of a word $w$,
where we use the term \emph{word} to mean a finite sequence of positive integers.

Buch, Kresch, Shimozono, Tamvakis, and Yong proved the following in \cite{BKSTY}:

\begin{theorem}[See \cite{BKSTY}]
\label{thm1}
There is a bijection 
from Hecke words for $\pi \in \SS$
to pairs $(P,Q)$
where $P$ is an increasing tableau whose row reading word is a Hecke word for $\pi$
and $Q$ is a standard set-valued tableau with the same shape as $P$.
\end{theorem}

Here, an \emph{increasing tableau} means an assignment of 
positive integers to the boxes of the Young diagram of a partition such that rows and columns
are strictly increasing. The \emph{row reading word}  is formed by concatenating 
the rows of such a tableau, starting with the last row.
A \emph{standard set-valued tableau} is defined in the same way as an increasing tableau,
except that the values assigned to each box are the disjoint blocks of a partition of $\{1,2,\dots,n\}$ for some $n$,
and a list of nonempty finite sets $S_1, S_2,\dots$ is strictly increasing
if $\max(S_i) < \min(S_{i+1})$ for all $i$.

The authors of \cite{BKSTY} constructed a bijection 
called \emph{Hecke insertion} to realize the preceding theorem.
Restricted to reduced words for permutations, Hecke insertion coincides with the
\emph{Edelman-Greene insertion} algorithm,
which is itself a generalization of the well-known \emph{Robinson-Schensted-Knuth (RSK) insertion} algorithm.

There are two natural ``shifted'' analogues of Hecke insertion,
whose properties are the central topic of this article.
We provide an overview of these maps in the next two sections.

\subsection{Orthogonal Hecke words}\label{intro2}

Let $\II = \{ z \in \SS : z^2=1\}$ denote the set of involutions in the symmetric group $\SS$.
The $\ZZ$-module $\cM:= \ZZ\spanning\{ U_z \in \Nil : z \in \II\}$
 is a right $\Nil$-module under the linear action $\* : \cM \times \Nil \to \cM$ with
$U_z \* U_\pi := U_{\pi^{-1}} U_z U_\pi$
 for $z \in \II$ and $\pi \in \SS$.


\begin{definition}
An \emph{orthogonal Hecke word} for $z \in \II$ is any word $i_1i_2\cdots i_l$
such that 
\[U_z = U_{i_l}\cdots U_{i_2}U_{i_1} U_{i_2} \cdots U_{i_l} = (\cdots( (U_{i_1}\*U_{i_2})\* U_{i_3}) \* \cdots )\*U_{i_l}.\]
\end{definition}

Let $G = \GL_n(\CC)$ be the general linear group of $n\times n$ invertible matrices and write $B\subset G$ 
for the subgroup of upper triangular matrices.
The orthogonal group $K= \O_n(\CC)$ acts with finitely many orbits on the flag variety $G/B$.
The closures of these $K$-orbits are in bijection
with the set of involutions $\mathfrak{I}_n := \{ z \in \II : z(i) =i \text{ for }i > n\}$.
Orthogonal Hecke words (at least, those of minimal length) for elements $z \in \mathfrak{I}_n$
correspond to maximal chains in the \emph{weak order} on $K$-orbit closures defined in \cite[\S1.2]{WyserYong} (see also \cite{CJW, RichSpring}).

Let $\iH(z)$ be the set of orthogonal Hecke words for $z \in \II$.
To match the notation 
in \cite{HMP1,HMP2,HMP4}
we 
write $\iR(z)$
for the set of words of minimal length in $\iH(z)$
and refer to the elements of $\iR(z)$ as \emph{involution words} for $z$.
The same sequences, 
read in reverse order, are called
\emph{reduced $\underline S$-expressions} in \cite{HH,H2}
and \emph{reduced $I_*$-expressions} in \cite{HuZhang1,Marberg2014}.
For example, if $z = 321\in \II$
then \[\iR(z) = \{12,21\}\subset \{12,21,112,122,121,212,\dots\} = \iH(z).\]
The following 
is an analogue of Theorem~\ref{thm1} for orthogonal Hecke words.
This statement is essentially
\cite[Theorem 5.18]{PatPyl2014},
but to deduce our phrasing
one also requires
 \cite[Corollary 2.18]{HKPWZZ}.

\begin{theorem}[See \cite{HKPWZZ,PatPyl2014}]
\label{thm2}
There is a bijection 
from orthogonal Hecke words for $z \in \II$
to pairs $(P,Q)$
where $P$ is an increasing shifted tableau whose row reading word is an orthogonal Hecke word for $z$
and $Q$ is a standard shifted set-valued tableau with the same shape as $P$.
\end{theorem}

For relevant preliminaries on shifted tableaux, see Section~\ref{tab-sect}.
Patrias and Pylyavskyy proved this theorem by constructing another bijection, 
called \emph{shifted Hecke insertion} in
\cite{HKPWZZ,HMP4,PatPyl2014}, between words and pairs of shifted tableaux.
We refer to this correspondence as \emph{orthogonal Hecke insertion}
to distinguish it from our second shifted map.

\subsection{Symplectic Hecke words}\label{intro3-sect}

Our main results concern a new ``symplectic'' variant of orthogonal Hecke insertion.
The domain of this correspondence is the set of \emph{symplectic Hecke words} defined as follows.

Let $\Theta : \PP \to \PP$ be the permutation with $\Theta(i) = i - (-1)^i$ for $i \in \PP$,
so that $\Theta$ is the infinite product of cycles $(1,2)(3,4)(5,6)\cdots$.
Define 
$\FF = \{ \pi^{-1} \Theta \pi : \pi \in \SS\}.$
The elements of $\FF$ are the fixed-point-free involutions of $\PP$ 
that agree with $\Theta$ outside a finite set of inputs. Each 
permutation in $\FF$ therefore has infinite support, so $\II$ and 
$\FF$ are disjoint.

For each $z \in \FF$, there exists an even integer $n \in 2\PP$ 
such that $z(i) = \Theta(i)$ for all $i > n$.  Thus, one way to 
represent an element $z \in \FF$ with a finite amount of data is to 
just list the values $z_1z_2\cdots z_n$ where $z_i = z(i)$ and $n$ 
is the even integer just mentioned. We identify this finite sequence 
with the element of $\FF$ that maps $i \mapsto z_i$ for $i \in [n]$ 
and $i \mapsto i-(-1)^i$ for $i>n$. An arbitrary word $z_1z_2\cdots z_n$ 
with even length $n$ represents an element of $\FF$ in this way if and 
only if $\{z_1,z_2,\dots,z_n\} = \{1,2,\dots,n\}$ and whenever 
$j= z_{i}$ we have $i=z_{j} \neq j$.

Let $\cN$ be the free $\ZZ$-module with basis $\{  N_z: z \in \FF\}$.
Results of Rains and Vazirani \cite{RainsVazirani}
imply that $\cN$ has a unique structure as a right $\Nil$-module with multiplication
satisfying
\[
N_z U_i = \begin{cases} N_{s_i z s_i} &\text{if }z(i) < z(i+1) \\
N_z &\text{if }i+1 \neq z(i) > z(i+1) \neq i\\
0 &\text{if }i+1=z(i) > z(i+1)=i 
\end{cases}
\quad\text{for $z \in \FF$ and $i \in \PP$.}
\]
This follows specifically from \cite[Theorems 4.6 and 7.1]{RainsVazirani} with $q=0$; 
see also \cite[\S4.2]{Marberg2016}.

\begin{definition}
A \emph{symplectic Hecke word} for $z \in \FF$ is any word $i_1i_2\cdots i_l$
such that 
\[N_z = N_\Theta U_{i_1}U_{i_2}\cdots U_{i_l}.\]
\end{definition}

To explain this terminology, again let $G = \GL_n(\CC)$ and write $B\subset G$ 
for the subgroup of upper triangular matrices.
When $n$ is even, the set of orbits of the 
symplectic group $K= \Sp_n(\CC)$ 
acting on $G/B$
is naturally in bijection
with the set of fixed-point-free involutions $\mathfrak{F}_n := \{ z \in \FF : z(i) =\Theta(i) \text{ for }i > n\}$.
Symplectic Hecke words for elements of $\mathfrak{F}_n$
correspond to maximal chains in the weak order on these $K$-orbit closures studied in \cite{CJW,RichSpring,WyserYong}.

Let $\cHfpf(z)$ be the set of symplectic Hecke words for $z \in \FF$.
The shortest elements of $\cHfpf(z)$
are the words $i_1i_2\dots i_l$ of minimal length with
$z = s_{i_l}\cdots s_{i_2}s_{i_1}\Theta s_{i_1}s_{i_2}\cdots s_{i_l}.$
Following the convention of \cite{HMP1,HMP2,HMP5},
we 
write $\cRfpf(z)$ 
for the set of such words, which we refer to as \emph{FPF-involution words} for $z$.
These elements are a special case of 
Rains and Vazirani's notion of \emph{reduced expressions} for quasi-parabolic sets \cite[Definition 2.11]{RainsVazirani}.
If $z = 4321 \in \FF$ then
\[\cRfpf(z) = \{ 21,23\} \subset \{21,23,221,211,213, 223,233,231,\dots\} =\cHfpf(z).\]
Since $N_\Theta U_i = 0$ if $i$ odd, every symplectic Hecke word begins with an even letter.
The following analogue of Theorem~\ref{thm1}
reappears in a more explicit form as Theorem~\ref{fpf-bij-thm}.
 
\begin{theorem}\label{thm3}
There is a bijection  
from symplectic Hecke words for $z \in \FF$
to pairs $(P,Q)$
where $P$ is an increasing shifted tableau whose row reading word is a symplectic Hecke word for $z$
and $Q$ is a standard shifted set-valued tableau with the same shape as $P$.
\end{theorem}

To prove this theorem, we will construct another shifted analogue of Hecke insertion,
which we call \emph{symplectic Hecke insertion}. 
Besides being a bijection,
symplectic Hecke insertion is a length- and descent-preserving map in an appropriate
sense; see Theorem~\ref{fpf-des-thm}.

Although not all words are symplectic Hecke words,
one can define orthogonal Hecke insertion as a special case of symplectic Hecke insertion.
Thus, Theorem~\ref{thm3} is a generalization of Theorem~\ref{thm2},
and our analysis of symplectic Hecke insertion
lets us recover many known properties of orthogonal Hecke insertion, along with some new ones,
in Section~\ref{ortho-sect}.

\subsection{Stable Grothendieck polynomials}\label{stab-intro-sect}

A primary application of Hecke insertion in
\cite{BKSTY} was to describe a rule for the expansion of the stable Grothendieck polynomials
$G_\pi$ indexed by permutations $\pi \in \SS$
in the basis of stable Grothendieck polynomials $G_\lambda$ indexed by partitions $\lambda$.
We briefly recall this rule.

A pair of words of the same length $(w,i)$ is a \emph{compatible sequence}
if $i = (i_1\leq i_2 \leq \dots \leq i_l)$ is a weakly increasing of positive integers with $i_j < i_{j+1}$ whenever $w_j \leq w_{j+1}$.
Let $\beta$, $x_1$, $x_2$, $x_3$, \dots, be commuting indeterminates.
The \emph{stable Grothendieck polynomial} of $\pi \in \SS$ is
the power series
\be\label{gpi-eq}
G_\pi =  \sum_{(w,i)} \beta^{\ell(w)-\ell(\pi)}  x^i \in \ZZ[\beta][[x_1,x_2,\dots]]
\ee
where the sum is over compatible sequences $(w,i)$
with $w \in \cH(\pi)$, and $x^i:= x_{i_1}x_{i_2}\cdots x_{i_l}$.
For the definition of $G_\lambda$ when $\lambda$ is a partition, see Theorem~\ref{lam-thm}.

\begin{theorem}[{See \cite[Theorem 1]{BKSTY}}]
\label{gpi-thm}
If $\pi \in \SS$ then $G_\pi = \sum_\lambda a_{\pi\lambda}  \beta^{|\lambda|-\ell(\pi)} G_\lambda$
where the sum is over all partitions $\lambda$,
and $a_{\pi\lambda}$ is the finite number of increasing tableaux $T$ of shape $\lambda$
whose row reading words are Hecke words for $\pi^{-1}$.
\end{theorem}

\begin{remark}
The power series denoted $G_\pi$ in
\cite{Buch2002,BKSTY} and \cite{BuchSamuel} are the special cases of \eqref{gpi-eq} with $\beta=-1$ and $\beta=1$, respectively.
This poses no loss of generality, since as long as $\beta\neq 0$ one can recover the generic form of $G_\pi$ from any specialization by a simple change of variables.
\end{remark}

The elements of $\II$ and $\FF$ index two natural families of ``shifted'' stable Grothendieck polynomials.
 For $y \in \II$ and $z \in \FF$,
write $\ellhat(y)$ and $\ellfpf(z)$ for the common lengths of all words in $\iR(y)$ and $\cRfpf(z)$, respectively;
see \eqref{ellhat-eq} for explicit formulas for these numbers.
We define the \emph{shifted stable Grothendieck polynomial} of $y \in \II$ and 
$z \in \FF$
to be the power series
\be\label{igyz-eq}
\iG_y =  \sum_{(w,i)} \beta^{\ell(w)-\ellhat(y)}  x^i
\qquand
\Gfpf_z= \sum_{(w,i)} \beta^{\ell(w)-\ellfpf(z)}  x^i
\ee
where the sums are over compatible sequences with $w \in \iH(y)$ and $w \in \cHfpf(z)$, respectively.

Ikeda and Naruse  \cite{IkedaNaruse} have defined 
a family of \emph{$K$-theoretic Schur $P$-functions} $\GP_\lambda$
indexed by strict partitions $\lambda$.
These
functions represent Schubert classes in the 
$K$-theory of torus equivariant coherent sheaves on the maximal orthogonal Grassmannian \cite[Corollary 8.1]{IkedaNaruse}.
As an application of our results on symplectic Hecke insertion,
we prove the following
in Section~\ref{stab-sect}:

\begin{theorem}\label{lastmain-thm}
Let $y \in \II$ and $z \in \FF$.
Then
$\iG_y = \sum_\lambda b_{y\lambda}  \beta^{|\lambda|-\ellhat(y)}  \GP_\lambda
$
and
$
\Gfpf_z = \sum_\lambda c_{z\lambda} \beta^{|\lambda|-\ellfpf(z)}  \GP_\lambda
$
where the sums are over all strict partitions $\lambda$, and 
$b_{y\lambda }$ and $c_{z\lambda}$
are the finite numbers of increasing shifted tableaux $T$ of shape $\lambda$
whose row reading words are orthogonal Hecke words for $y$ 
and symplectic Hecke words for $z$, respectively.
\end{theorem}

The power series $G_\pi$ are of interest as the stable limits
of the \emph{Grothendieck polynomials} $\fk G_\pi$ defined in \cite{LS1982} to represent
the classes of the structure sheaves of Schubert varieties in the $K$-theory of the complete flag variety.
The precise relationship is that $G_\pi = \lim_{n\to \infty} \fk G_{1^m \times \pi}$ where $1^m\times \pi$ denotes the permutation
with $i \mapsto i$ for $i \leq m$ and $i \mapsto m + \pi(i-m)$ for $i>m$.
Wyser and Yong \cite{WyserYong}
have introduced analogous $K$-theory representatives 
for 
 the orbits of the symplectic group acting on the complete flag variety.
It is shown in \cite{MarPaw1,MarPaw2}
that the stable limits of Wyser and Yong's polynomials 
coincide (up to a minor change of variables)
with the symmetric functions $\{\Gfpf_z\}_{z \in \FF}$;
moreover, for each strict partition $\lambda$,
there exists a corresponding ``Grassmannian'' involution $z_\lambda \in \FF$
such that $\Gfpf_{z_\lambda} = \GP_{\lambda}$.

Contrary to what one might expect,
the symmetric functions $\{\iG_y\}_{y \in \II}$ do not arise in the same way by taking the stable limits 
of $K$-theory representatives for the orbits of the orthogonal group acting on the complete flag variety.
It is an open problem to find general formulas for these $K$-theory representatives and their stable limits.
At the same time, it also remains to find a geometric interpretation of $\iG_y$ for $y \in \II$.


Here is a short outline of the rest of this article.
Section~\ref{words-sect},
includes some preliminary results on Hecke words and tableaux.
Section~\ref{big-sect} constructs the symplectic Hecke insertion algorithm and its inverse.
In Section~\ref{var-sect} we discuss three related maps.
Section~\ref{des-sect}
formulates a semistandard version of our insertion algorithm.
In Section~\ref{ortho-sect}, we explain how  orthogonal Hecke insertion arises as a special case of symplectic Hecke insertion.
Section~\ref{ick-sect} provides a discussion of the simplified forms of
orthogonal and symplectic Hecke insertion obtained by restricting  the domain of each map
to (FPF-)involution words.
Section~\ref{stab-sect}, finally, contains the proof of Theorem~\ref{lastmain-thm}.

\subsection*{Acknowledgements}

This work was partially supported by Hong Kong RGC Grant ECS 26305218.
I am grateful to
Sami Assaf, Zachary Hamaker, Joel Lewis, and Brendan Pawlowski for many useful conversations. 
I also thank Tim Strunkheide for pointing out several mistakes in the published version of this article (most significantly in
Definition~\ref{weak-def}, which was missing a condition previously included in Definition~\ref{adm-def}), 
which have been corrected in this draft.

\section{Preliminaries}\label{words-sect}

We fix the following notation:
let $\PP = \{1,2,3,\dots\}$, $\NN = \{0,1,2,\dots\}$, and $[n] = \{i \in \PP : i \leq n\}$ for $n \in \NN$.
A \emph{word} is finite sequence of positive integers.
We write $\ell(w)$ for the length of a word $w$, $vw$ for the concatenation of words $v$ and $w$,
and $\emptyset$ for the unique empty word.

\subsection{Hecke words}\label{braid-sect}

%

A \emph{congruence} is an equivalence relation $\sim$ on words with the property that $v \sim w$
implies $avb\sim awb$ for all words $a$ and $b$.
Define $\simA$ to be the congruence generated by the usual Coxeter braid relations for $\SS$;
 i.e., let $\simA$ denote the strongest congruence with
$ ij \simA ji $ and $ i(i+1)i\simA (i+1)i(i+1)$
for all $i,j \in \PP$ with $|i-j| > 1$.
Write $\equivA$ for the strongest congruence 
with 
$ij \equivA ji$ and $jkj\equivA kjk$ and $i\equivA ii$
for all $i,j,k \in \PP$ with $|i-j| > 1$.
The following is well-known.

\begin{theorem}
\label{gp-prop}
If $\pi \in \SS$ then $\cR(\pi)$ is an equivalence class under $\simA$
while $\cH(\pi)$ is an equivalence class $\equivA$.
A word is reduced
if and only if its equivalence class under $\simA$ contains no words with equal adjacent letters.
\end{theorem}

There are versions of this theorem for orthogonal and symplectic Hecke words.
Define $\isim$  (respectively, $\iequiv$) to be the transitive closure
of $\simA$ (respectively, $\equivA$) and the relation with $w_1w_2w_3w_4\cdots w_n\sim w_2w_1w_3w_4\cdots w_n$
for all words with at least two letters.
The following combines \cite[Theorems 6.4 and 6.10]{HMP2}; the first claim is also equivalent to \cite[Theorem 3.1]{HuZhang1}:

\begin{theorem}[See \cite{HMP2}]\label{inv-hecke-char-thm}
If $z \in \II$ then $\iR(z)$ is an equivalence class under $\isim$
while $\iH(z)$ is an equivalence class under $\iequiv$.
A word is an involution word for some $z \in \II$
if and only if its equivalence class under $\isim$ contains no words with equal adjacent letters.
\end{theorem}


Theorems~\ref{gp-prop} and \ref{inv-hecke-char-thm}
imply that there are finite sets $\cA(z)\subset \cB(z)\subset \SS$ with 
\be\label{atoms2-eq}
\iR(z) = \bigsqcup_{\pi \in \cA(z)} \cR(\pi)
\qquand
\iH(z) = \bigsqcup_{\pi \in \cB(z)} \cH(\pi).
\ee
For example, $
\cA(321) = \{231, 312\}\subset \{231, 312, 321\} = \cB(321).
$
Following \cite{HMP2}, we refer to the elements of $\cA(z)$ as \emph{atoms} for $z$ and to the elements of $\cB(z)$ as \emph{Hecke atoms}.

Fix $z \in \II$
and suppose $a_1<a_2<a_3<\dots$ are the integers $a \in \PP$ such that $a \leq z(a)$.
Define $b_i = z(a_i)$ for each $i \in \PP$ and let 
$\alpha_{\min}(z) = (a_1b_1 a_2b_2 a_3b_3\cdots)^{-1} \in \SS$ where in the word $a_1b_1 a_2b_2 a_3b_3$ we omit $b_i$ whenever $a_i=b_i$.
Write $\pi_i =\pi(i)$ for $\pi \in \SS$ and $i \in \PP$.
Let $\sim_\cB$ be the strongest equivalence relation on $\SS$ 
with $\pi^{-1} \sim_\cB \sigma^{-1}$ whenever there are
integers $a<b<c$ and 
an index $i \in \PP$ such that 
$ \pi_{i}\pi_{i+1}\pi_{i+2}$ and $\sigma_{i}\sigma_{i+1}\sigma_{i+2}$
both belong to
$\{ cba, bca, cab\}$
while
$\pi_j = \sigma_j$ for all $j \notin \{i,i+1,i+2\}$.
The following is another consequence of \cite[Theorems 6.4 and 6.10]{HMP2}.

\begin{theorem}[{See \cite{HMP2}}]
If $z \in \II$ then $\alpha_{\min}(z) \in \cA(z)$ and $\cB(z) = \left\{ w \in \SS : \alpha_{\min}(z) \sim_\cB w\right\}$.
\end{theorem}

Define $\fsim$  (respectively, $\fequiv$) to be the transitive closure
of $\simA$ (respectively, $\equivA$) and the relation with $w_1w_2w_3\cdots w_n\sim w_1(w_2+2)w_3\cdots w_n$
whenever $w_1 = w_2 +1$.
Recall that $\cRfpf(z)$ is the set of minimal length words in $\cHfpf(z)$.
A word is a \emph{symplectic Hecke word} (respectively, an \emph{FPF-involution word})
if it is an element of $\cHfpf(z)$ (respectively, $\cRfpf(z)$) for some $z \in \FF$.

\begin{theorem}\label{fpf-hecke-char-thm}
If $z \in \FF$ then $\cRfpf(z)$ and
 $\cHfpf(z)$ are equivalence classes under $\fsim$ and $\fequiv$, respectively.
A word is a symplectic Hecke word 
if and only if its equivalence class under $\fequiv$ contains no words that begin with an odd letter.
A symplectic Hecke word is an FPF-involution word if and only if its
 equivalence class under $\fsim$ contains no words with equal adjacent letters.
\end{theorem}

\begin{proof}
The claim that $\cRfpf(z)$ is an equivalence class under $\fsim$ for each $z \in \FF$ follows from \cite[Theorem 6.22]{HMP2}.
Since $\cN$ is a $\Nil$-module and $N_\Theta U_iU_{i-1} = N_\Theta U_i U_{i+1}$ whenever $i \in 2\PP$ is even,
 Theorem~\ref{gp-prop} implies that each set $\cHfpf(z)$ for $z \in \FF$ is preserved by $\fequiv$.

The following argument is similar to the proof of  \cite[Theorem 6.18]{HMP2}.
Define $\sim_{\cBfpf}$ to be the strongest equivalence relation on $\SS$ 
with $\pi^{-1} \sim_{\cBfpf} \sigma^{-1}$ whenever there
are integers $a<b<c<d$ and  
an even index $i \in 2\NN$ such that 
$ \pi_{i+1}\pi_{i+2}\pi_{i+3} \pi_{i+4}$ and $\sigma_{i+1}\sigma_{i+2}\sigma_{i+3}\sigma_{i+4}$
both belong to
$\{ adbc, bcad, bdac\}$
while
$\pi_j = \sigma_j$ for all $j \notin \{i+1,i+2,i+3,i+4\}$.
In this case there is a word $w$ such that $\cH(\pi)$ and $\cH(\sigma)$
each contain one of $w(2i+1)(2i) $ or $w(2i-1)(2i)$ or $w(2i-1)(2i+1)(2i)$,
so $\cH(\pi^{-1}) \cup \cH(\sigma^{-1})$ is a subset of an equivalence class under $\fequiv$.

No symplectic Hecke word can begin with an odd letter 
or with 
$(2i)(2i+1)(2i-1)(2i)$ for any $i \in \PP$, since 
$N_\Theta U_{2i} U_{2i+1}U_{2i-1} = N_z$ for $z\in \FF$ with $2i+1 = z(2i) > z(2i+1)=2i$.
Suppose $w$ is a symplectic Hecke word for $z \in \FF$. Then $w\in \cH(\pi^{-1})$ for some $\pi \in \SS$,
and the preceding observations imply that $\pi_{2i-1} < \pi_{2i}$ for all $i \in \PP$ and that we never have
$\pi_{2i-1}\pi_{2i}\pi_{2i+1}\pi_{2i+2} = cdab$
for numbers $a<b<c<d$.
Using these facts, 
it is an exercise to show that $\pi^{-1}\sim_{\cBfpf}\sigma^{-1}$ 
for some $\sigma\in \SS$ with $\sigma_{2i-1} < \sigma_{2i}$ 
and $\sigma_{2i-1} < \sigma_{2i+1}$
for all $i \in \PP$.
Now define $y \in \FF$ to be the fixed-point-free involution with $y(\sigma_{2i-1}) = \sigma_{2i}$ for all $i \in \PP$.
\cite[Theorem 6.22]{HMP2}  asserts that $\cH(\sigma^{-1}) \subset \cHfpf(y)$;
since $v\fequiv w$ for all $v \in \cH(\sigma^{-1})$
and since $\cHfpf(z)$ is preserved by $\fequiv$, we must have $y = z$.
We conclude that if $a_1<a_2<\dots$ are the elements of $\{a \in \PP : a<z(a)\}$ listed in order and $b_i = z(a_i)$,
then every symplectic Hecke word for $z$ 
is equivalent under $\fequiv$ to every Hecke word for the permutation $\beta_{\min}(z) := (a_1b_1a_2b_2\cdots)^{-1}$.

Next, consider an equivalence class under $\fequiv$ that is not equal to $\cHfpf(z)$ for any $z \in \FF$.
Suppose $w$ is a word of minimal length in this class, so that $w$ is not a symplectic Hecke word.
Let $n$ be minimal such that the initial subword $w_1w_2\cdots w_n$ is not a symplectic Hecke word.
Since $\emptyset\in \cRfpf(\Theta)$,
we have $n>0$. Our minimality assumptions imply that 
if $\pi := s_{w_1}s_{w_2}\cdots s_{w_{n-1}} \in \SS$ then
$w_1w_2\cdots w_{n-1}$ is an FPF-involution word for 
$z :=\pi^{-1} \Theta \pi \in \FF$ and $z(w_n) = w_n+1$.
Therefore $\{\pi(w_n),\pi(w_{n+1})\} = \{2i-1,2i\}$ for some $i \in \PP$,
and we have $w_1w_2\cdots w_n \equivA (2i-1)w_1w_2\cdots w_{n-1}$.
We conclude that a word is not a symplectic Hecke word
if and only if it is equivalent under $\fequiv$
to a word that begins with an odd letter.

A similar argument shows that a symplectic Hecke word is an involution word
if and only if its equivalence class under $\fsim$ contains no words with equal adjacent letters.
We omit the details.
%
\end{proof}

Theorems~\ref{gp-prop} and \ref{fpf-hecke-char-thm}
imply
that
there are finite subsets $\cAfpf(z)\subset \cBfpf(z)\subset \SS$ with 
\be\label{atoms1-eq}
\cRfpf(z) = \bigsqcup_{\pi \in \cAfpf(z)} \cR(\pi)
\qquand
\cHfpf(z) = \bigsqcup_{\pi \in \cBfpf(z)} \cH(\pi)
\ee
for each $z \in \FF$.
For example,
$ 
\cAfpf(4321) =\{3124, 1342\} \subset \{3124, 1342, 3142\} = \cBfpf(4321).
$
We again refer to elements of $\cAfpf(z)$ as \emph{atoms} for $z$ and to elements of $\cBfpf(z)$ as \emph{Hecke atoms}.
The notation ``$\cBfpf(z)$'' is used in \cite[\S6.2]{HMP2} to denote a slightly larger set.
If $\sim_{\cBfpf}$ and $\beta_{\min}(z)$ are defined as above, then the
proof of Theorem~\ref{fpf-hecke-char-thm} reduces to the following statement:

\begin{theorem}
If $z \in \FF$ then $\beta_{\min}(z)  \in \cAfpf(z)$ and $\cBfpf(z) = \left\{ w \in \SS : \beta_{\min}(z) \sim_{\cBfpf} w\right\}$.
\end{theorem}

\subsection{Tableaux}\label{tab-sect}

Recall that $\PP$ is the set of positive integers.
Throughout, we use the term
 \emph{tableau} to mean any map from a finite subset of $\PP\times \PP$
to $\PP$.
We refer to the domain of a tableau as its \emph{shape},
and write $\emptyset$ for the unique tableau whose shape is the empty set.

A tableau has $m$ rows (respectively, $n$ columns) if its shape is 
 contained in $[m] \times \PP$ but not $[m-1] \times \PP$ (respectively,
 $\PP \times [n]$ but not $\PP\times [n-1]$).
The \emph{$i$th row} and \emph{$j$th column} of a tableau $T$ refer to the tableaux
 formed by restricting $T$ to the  subset of its domain
 in $\{i\} \times \PP$ and $\PP\times \{j\}$.
 
Let $T$ be a tableau with shape $\cD$.
We write $(i,j) \in T$ to mean that $(i,j) \in \cD$
and define $T_{ij}:=T(i,j)$ for $(i,j) \in T$.
A tableau $T$ is \emph{increasing} if $T_{ab} < T_{xy}$ whenever $(a,b),(x,y) \in T$ 
are distinct positions with $a\leq x$ and $b\leq y$.
If $(i,j) \in T$ then the tableau formed by ``removing box $(i,j)$ from $T$''
is the restriction of $T$ to $\cD-\{(i,j)\}$.
If $(i,j) \in T$ then the tableau formed by ``replacing the value of box  $(i,j)$ in $T$ by $c$''
is the map with domain $\cD$ that has $(i,j) \mapsto c$ and agrees with $T$ on the subdomain $\cD -\{(i,j)\}$.
If $(i,j)\notin T$ then the tableau formed by ``adding $c$ to box $(i,j)$ in $T$''
is the map with domain $\cD\sqcup \{(i,j)\}$ that extends $T$ and has $(i,j) \mapsto c$.

\begin{example}\label{tab-ex2}
We draw tableaux  in French notation, so that each row is placed on top of the previous one.
For example, the tableau  $T=\{ (1,1) \mapsto 1, (1,2) \mapsto 2, (2,3)\mapsto 3, (3,2)\mapsto 4\}$ is
\be\label{tab-ex-eq1}
\ytableausetup{boxsize = .4cm,aligntableaux=center}
 \begin{ytableau}
  \none[\cdot]  &4 &  \none[\cdot] \\
 \none[\cdot] & \none[\cdot] &  3   \\ 
 1 & 2 &  \none[\cdot] 
 \end{ytableau}
 \ee
The following tableaux are increasing with shape $\{(1,1),(1,2),(1,3),(2,2)\}$:
\be\label{tab-ex-eq2}
\ytableausetup{boxsize = .4cm,aligntableaux=center}
 \begin{ytableau}
 \none & 4 & \none  \\ 
 2 & 3 & 4
 \end{ytableau}
 \qquad
  \begin{ytableau}
 \none & 5 & \none   \\ 
 2 & 3 & 4
 \end{ytableau}
 \qquad
  \begin{ytableau}
 \none & 4  & \none \\ 
 2 & 3 & 5
 \end{ytableau}
 \qquad
  \begin{ytableau}
 \none & 8  & \none \\ 
 1 & 4 & 9
 \end{ytableau} 
 \ee
\end{example}


Let $T$ be a tableau.
The \emph{row reading word} (respectively, \emph{column reading word}) of $T$ 
is the finite sequence $\row(T)$ (respectively, $\col(T)$) whose entries are the values $T_{ij}$ 
as $(i,j)$ ranges over the domain of $T$ such that
$(-i,j)$ (respectively, $(j,-i)$) increases lexicographically.
For example,
the row reading word of the tableau in \eqref{tab-ex-eq1} is $4312$,
while the column reading word of that tableau is $1423$.
The tableaux in \eqref{tab-ex-eq2} have row reading words
4234, 5234, 4235, and 8149, and column reading words 2434, 2534, 2435, and 1849, respectively.

Let $\prec$ be the strict partial order on $\PP\times \PP$ that has $(a,b) \prec(x,y)$
if and only if $a\leq x$ and $ b\leq y$ and $(a,b) \neq (x,y)$.
A tableau $T$ is \emph{row-column-closed} if whenever $(a,b),(x,y) \in T$ and $(a,b) \prec (x,y)$,
it holds that
$(a,y) \in T$. The following picture illustrates this condition. If the two boxes in the first diagram are in the domain of $T$ then $T$ must also contain the third box:
\[
 \ytableausetup{boxsize = .3cm,aligntableaux=center}
 \begin{ytableau}
 \none[\cdot] & \none[\cdot] &  \\ 
 \none[\cdot] & \none[\cdot] & \none[\cdot]  \\
   & \none[\cdot] &  \none[\cdot]
 \end{ytableau}
 \qquad\Rightarrow\qquad
  \begin{ytableau}
 \none[\cdot] & \none[\cdot] &  \\ 
 \none[\cdot] & \none[\cdot] & \none[\cdot]  \\
   & \none[\cdot] &  
 \end{ytableau}
 \]
 Finally, define \emph{Coxeter-Knuth equivalence} to be the strongest congruence $\simCK$ with 
$ cab \simCK acb$
and
$ bca \simCK bac$
and
$ aba \simCK bab$
for all positive integers $a,b,c \in \PP$ with $a<b<c$.

\begin{lemma}\label{rcc-lem} If $T$ is an increasing, row-column-closed tableau,
then $\row(T) \simCK \col(T)$.
\end{lemma}

\begin{proof}
Form $w$ by reading the last column of $T$ in reverse order.
Form $U$ from $T$ by removing the last column. Then $U$ is also row-column-closed, $\col(T) = \col(U)w$,
and by induction $\row(U) \simCK \col(U)$. It is a manageable exercise to check that $\row(T) \simCK \row(U)w$,
so  $\row(T) \simCK \col(T)$.
\end{proof}

The \emph{northeast} (respectively, \emph{southwest}) \emph{diagonal reading word} of a tableau $T$
is the finite sequence 
 $\nediag(T)$
 (respectively, $\swdiag(T)$) whose entries are the values $T_{ij}$ as $(i,j)$
ranges over the domain of $T$ such that $(j-i,i)$ (respectively, $(j-i,-i)$)  increases lexicographically.
Equivalently, the northeast (respectively, southwest) diagonal reading word is formed by reading the entries of each diagonal from left to right
(respectively, top to bottom), starting with the first diagonal. 
For example,
the tableau $T$ in \eqref{tab-ex-eq1} has $\nediag(T) = 4123$
and $\swdiag(T) = 4132$.
The tableaux in \eqref{tab-ex-eq2} have northeast diagonal reading words
2434, 2534, 2435, and 1849, and southwest diagonal reading words 4234, 5234, 4235, and 8149, respectively.

A tableau $T$ is \emph{row-diagonal-closed} if whenever $(a,b),(x,y)\in T$ and $(a,b) \prec (x,y)$ and $\Delta := (y - x) - (b-a) \geq 0$, it holds that
$(x, y-\Delta) \in T$. The following picture illustrates this condition:
\[
 \ytableausetup{boxsize = .3cm,aligntableaux=center}
 \begin{ytableau}
 \none[\cdot] & \none[\cdot] & \none[\cdot] & \none[\cdot] &  \\ 
 \none[\cdot] & \none[\cdot] & \none[\cdot]& \none[\cdot]& \none[\cdot]  \\
   & \none[\cdot] & \none[\cdot]& \none[\cdot]& \none[\cdot]
 \end{ytableau}
  \qquad\Rightarrow\qquad
   \begin{ytableau}
 \none[\cdot] & \none[\cdot] &  & \none[\cdot] &  \\ 
 \none[\cdot] & \none[\cdot] & \none[\cdot]& \none[\cdot]& \none[\cdot]  \\
   & \none[\cdot] & \none[\cdot]& \none[\cdot]& \none[\cdot]
 \end{ytableau}
 \]
A tableau $T$ is \emph{column-diagonal-closed} if whenever $(a,b),(x,y)\in T$ and $(a,b) \prec (x,y)$ and $\Delta := (b - a) - (y - x) \geq 0$, it holds that
$(a+\Delta,b) \in T$. The following picture illustrates this condition:
\[
 \ytableausetup{boxsize = .3cm,aligntableaux=center}
 \begin{ytableau}
  \none[\cdot] &  \none[\cdot] &    \\ 
  \none[\cdot] &    \none[\cdot] &  \none[\cdot]   \\ 
 \none[\cdot] & \none[\cdot]& \none[\cdot]   \\
  \none[\cdot] & \none[\cdot] & \none[\cdot]  \\
   & \none[\cdot] & \none[\cdot]
 \end{ytableau}
 \qquad\Rightarrow\qquad
  \begin{ytableau}
  \none[\cdot] &  \none[\cdot] &    \\ 
  \none[\cdot] &    \none[\cdot] &  \none[\cdot]   \\ 
 & \none[\cdot]& \none[\cdot]   \\
  \none[\cdot] & \none[\cdot] & \none[\cdot]  \\
   & \none[\cdot] & \none[\cdot]
 \end{ytableau}
 \]

\begin{lemma}\label{row-diagonal-closed-lem} If $T$ is an increasing, row-diagonal-closed tableau,
then $\row(T) \simCK \swdiag(T)$.
\end{lemma}

\begin{proof}
Form $w$ by reading the first diagonal 
(i.e., the diagonal containing all $(i,j) \in T$ for which $j-i$ is minimal) of $T$ in reverse order.
Form $U$ from $T$ by removing the first diagonal.
Then $U$ is row-diagonal-closed, $\swdiag(T) = w \swdiag(U)$,
and by induction $\row(U) \simCK \swdiag(U)$. It is an easy exercise to check that $\row(T) \simCK w \row(U)$,
so $\row(T) \simCK\swdiag(T)$.
\end{proof}

\begin{lemma}\label{column-diagonal-closed-lem} If $T$ is an increasing, column-diagonal-closed tableau,
then $\col(T) \simCK \nediag(T)$.
\end{lemma}

\begin{proof}
Suppose $T$ is an increasing, column-diagonal closed tableau, and write
$T^\dag$ for its transpose. 
By the previous lemma $\row(T^\dag) \simCK \swdiag(T^\dag)$.
Write $w^\r$ for the word given by reversing $w$.
Then $u \simCK v$ if and only if $u^\r \simCK v^\r$.
The lemma follows since $\col(T)^\r = \row(T^\dag)$ and $\nediag(T)^\r = \swdiag(T^\dag)$.
\end{proof}

A tableau $T$ is \emph{shifted} if
for some strict partition $\lambda = (\lambda_1 >\lambda_2 >\dots >\lambda_l>0)$,
the domain of $T$
is the \emph{shifted Young diagram}
$
\SY_\lambda := \{ (i,i+j-1) \in [l]\times \PP :  1\leq j\leq \lambda_i\}
$.

\begin{corollary}\label{t-cor}
If $T$ is an increasing shifted tableau then 
$\row(T) \simCK \col(T) \simCK \nediag(T) \simCK\swdiag(T)$.
\end{corollary}

\begin{proof}
Such a tableau is  row-column-closed, row-diagonal-closed, and column-diagonal-closed.
\end{proof}


A \emph{set-valued tableau} is a map
from a finite subset of $\PP\times \PP$ to the set of nonempty, finite subsets of 
the \emph{marked alphabet} $\MM = \{1' < 1 < 2' < 2 < 3' <3 < \dots\}$.
Most of our conventions 
for referring 
to tableaux
 extend to set-valued tableaux without any complication.
However, 
with set-valued tableaux,
it is possible to add multiple entries to a given box.

A set-valued tableau $T$
is \emph{increasing} if $\max(T_{ab}) < \min(T_{xy})$ whenever $(a,b),(x,y) \in T$ 
are distinct positions with $a\leq x$ and $b\leq y$.
The \emph{length} (or \emph{degree}) of a set-valued tableau $T$ is the sum of the sizes of its entries;
we denote this quantity by 
 $|T| := \sum_{(i,j) \in T} |T_{ij}|.$

 A shifted set-valued tableau $T$ is \emph{standard} if 
 $T$ is increasing, no primed number belongs to any box of $T$ on the main diagonal,
 and exactly one of $i$ or $i'$ appears in some box of $T$ for each $i \in \{1,2,\dots,|T|\}$.
The entries of a standard set-valued tableaux $T$ must be pairwise disjoint
and cannot contain $i$ or $i'$ for any integer $i\notin \{1,2,\dots,|T|\}$.
The following are standard shifted set-valued tableaux
with length 6 and with shape $\SY_\lambda$ for $\lambda=(2,1)$:
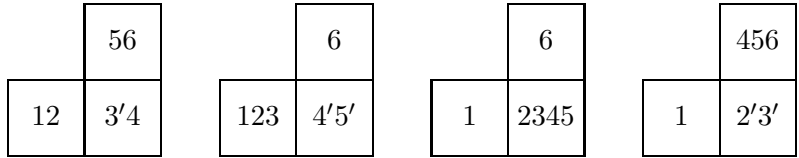
\[
\ytableausetup{boxsize = 1.0cm,aligntableaux=center}
 \begin{ytableau}
 \none & 5 6   \\ 
 12 & 3'4 
 \end{ytableau}
 \qquad
  \begin{ytableau}
 \none & 6    \\ 
 123 & 4'5' 
 \end{ytableau}
 \qquad
  \begin{ytableau}
 \none & 6 \\ 
 1 & 2345 
 \end{ytableau}
 \qquad
  \begin{ytableau}
 \none & 456   \\ 
 1 & 2'3' 
 \end{ytableau} 
 \]

A standard shifted set-valued tableau corresponds to a directed path 
starting at the empty partition in what Patrias and Pylyavskyy call 
the \emph{M\"obius deformation} of the shifted Young lattice \cite[\S5.2]{PatPyl2014}.
The length of this walk is the length of the set-valued tableau.

\section{Symplectic Hecke insertion}\label{big-sect}


\subsection{Forward transitions}\label{forward-sect}

Consider the following class of ``almost shifted'' tableaux:
 
\begin{definition}
A \emph{shifted insertion state} is a tableau that is either
(a) increasing, shifted, and nonempty or (b) formed by adding 
to an increasing shifted tableau with $m-2$ rows and $n-2$ columns
an extra box $(i,j)$  contained in either $\{m\} \times [n-1]$ or $[m-1]\times \{n\}$.
\end{definition}

A shifted insertion state is \emph{terminal} in case (a).
The position $(i,j)$ in case (b) is the state's \emph{outer box}.
A non-terminal insertion state is \emph{initial} if its outer box is in the first row.

\begin{example}\label{states-ex}
The following are all shifted insertion states:
\[
\ytableausetup{boxsize = .4cm,aligntableaux=center}
        \begin{ytableau}
\none[\cdot] & \none[\cdot]  & \none[\cdot]& \none[\cdot] \\
\none[\cdot] & \none[\cdot]  & \none[\cdot]& \none[\cdot] \\
\none[\cdot] & \none[\cdot]  & \none[\cdot]& \none[\cdot]\\
\none[\cdot] & \none[\cdot]  &4& \none[\cdot]
 \end{ytableau}
      \qquad
  \begin{ytableau}
\none[\cdot] & \none[\cdot]  & \none[\cdot]& \none[\cdot] \\
\none[\cdot] & \none[\cdot]  & \none[\cdot]& \none[\cdot] \\
\none[\cdot] & 4  &  \none[\cdot]& \none[\cdot]\\
2 & 3  &  \none[\cdot]& \none[\cdot]
 \end{ytableau}
      \qquad
  \begin{ytableau}
\none[\cdot] & \none[\cdot]  & \none[\cdot]& \none[\cdot] \\
\none[\cdot] & \none[\cdot]  & \none[\cdot]& \none[\cdot] \\
\none[\cdot] &4  &  \none[\cdot]& \none[\cdot]\\
2 & 3  &  \none[\cdot]& 2
 \end{ytableau}
      \qquad
  \begin{ytableau}
\none[\cdot] & \none[\cdot]  & \none[\cdot]& \none[\cdot] \\
\none[\cdot] & \none[\cdot]  & \none[\cdot]& \none[\cdot] \\
\none[\cdot] & 4  &  \none[\cdot]& 3 \\
2 & 3  &  \none[\cdot]& \none[\cdot]
 \end{ytableau}
        \qquad
  \begin{ytableau}
\none[\cdot] & 3 & \none[\cdot]& \none[\cdot] \\
\none[\cdot] & \none[\cdot]  & \none[\cdot]& \none[\cdot] \\
\none[\cdot] &4  &  \none[\cdot]& \none[\cdot] \\
2 & 3  &  \none[\cdot]& \none[\cdot]
 \end{ytableau}
       \qquad
  \begin{ytableau}
\none[\cdot] & \none[\cdot]  & 3& \none[\cdot] \\
\none[\cdot] & \none[\cdot]  & \none[\cdot]& \none[\cdot] \\
\none[\cdot] & 4  &  \none[\cdot]& \none[\cdot] \\
2 & 3  &  \none[\cdot]& \none[\cdot]
 \end{ytableau}
 \]
 The first and third states are initial, while the second is terminal.
\end{example}

To define symplectic Hecke insertion, we will
give the set of shifted insertion states the structure of a weighted directed graph whose edges are labelled by pairs of positive integers.
We call this the \emph{forward transition graph}.
Terminal insertion states are local sinks in this graph,  
while every other state has a unique outgoing edge.
Edges between shifted insertion states belong to three families, which we 
now describe.

Let $U$ be a non-terminal shifted insertion state
that has $m-2$ rows and $n-2$ columns when its outer box is removed.
Assume the outer box of $U$ is $(i,n)$.
Suppose $U_{in}$ is maximal in its row
and $j \in \PP$ is minimal with $i \leq j$ and $(i,j) \notin U$.
The unique outgoing edge from $U$ is then as follows:

\ben
\item[(R1)] 
If moving the outer box of $U$ to position $(i,j)$ yields an increasing shifted tableau $V$,
then 
there is an edge $U \xrightarrow{(i,j)} V$.
 
  \item[(R2)] 
  If moving the outer box of $U$ to position $(i,j)$ does not yield an increasing shifted tableau,
then there is an edge $U \xrightarrow{(i,j)} V$ where
 $V$ is formed from $U$ by removing the outer box $(i,n)$.

\een
Next, suppose there exists a minimal index $x \in \PP$ with $(i,x) \in U$ and $U_{in} < U_{ix}$.
Define $T$ to be the tableau formed from $U$ by replacing the value 
in box $(i,x)$ by $U_{in}$ and then removing the outer box $(i,n)$.
For the moment, assume $i<x$.
\ben  
\item[(R3)] 
Suppose 
the tableau $T$ is not increasing.
If $i+1 < x$ or row $i+1$ of $U$ is nonempty,
 then there is an edge $U \xrightarrow{(i,x)} V$ where
 $V$  is formed from $U$ 
 by moving box $(i,n)$ to $(i+1,n)$ and changing its value to $U_{ix}$,
 as in this picture:
 \[
 \ytableausetup{boxsize = .4cm,aligntableaux=center}
  U = \left\{\ 
 \begin{ytableau}
    \none[\cdot] &  \none[\cdot] &  \none[\cdot] & \none[\cdot] & \none[\cdot] & \none[\cdot]\\
  \none[\cdot] &  \none[\cdot] & \qm & \qm & \none[\cdot] & \none[\cdot]\\
   \none[\cdot] & 2 & \circled{4} & 5 & \none[\cdot] &3\\
\qm & \qm & 3 & \qm & \none[\cdot] & \none[\cdot]
\end{ytableau}
\ \right\}
\ \xrightarrow{(i,x)}\
\left\{\
 \begin{ytableau}
    \none[\cdot] &  \none[\cdot] &  \none[\cdot] & \none[\cdot] & \none[\cdot] & \none[\cdot]\\
  \none[\cdot] &  \none[\cdot] & \qm & \qm & \none[\cdot] &4\\
   \none[\cdot] & 2 & 4 & 5 & \none[\cdot] & \none[\cdot] \\
\qm & \qm & 3 & \qm & \none[\cdot] & \none[\cdot]
\end{ytableau}
\ \right\} = V
\quad\text{as}\quad
  T = \left\{\ 
 \begin{ytableau}
   \none[\cdot] &  \none[\cdot] &  \none[\cdot] & \none[\cdot]  &  \none[\cdot]\\
  \none[\cdot] &  \none[\cdot] & \qm & \qm &  \none[\cdot] \\
   \none[\cdot] & 2 & 3& 5&  \none[\cdot] \\
\qm & \qm & 3 & \qm &  \none[\cdot]
\end{ytableau}
\right\}.
  \]
   Here and in the next two cases, the circled entry indicates the location of box $(i, x)$.
   
  \item[(R4)] 
If 
the tableau $T$ is increasing,
then 
there is an edge $U \xrightarrow{(i,x)} V$ where
$V$ is formed from $T$ by adding an outer box in row $i+1$ with value $U_{ix}$,
as in the following picture:
 \[
 \ytableausetup{boxsize = .4cm,aligntableaux=center}
  U = \left\{\ 
 \begin{ytableau}
    \none[\cdot] &  \none[\cdot] &  \none[\cdot] & \none[\cdot] & \none[\cdot] & \none[\cdot]\\
  \none[\cdot] &  \none[\cdot] &  \qm & \qm & \none[\cdot] & \none[\cdot]\\
   \none[\cdot] & 2 &  \circled{4} & 5 & \none[\cdot] &3\\
\qm & \qm & 2 & \qm & \none[\cdot] & \none[\cdot]
\end{ytableau}
\ \right\}
\ \xrightarrow{(i,x)}\
\left\{\
 \begin{ytableau}
    \none[\cdot] &  \none[\cdot] &  \none[\cdot] & \none[\cdot] & \none[\cdot] & \none[\cdot]\\
  \none[\cdot] &  \none[\cdot] &  \qm & \qm & \none[\cdot] &4\\
   \none[\cdot] & 2 &3 & 5 & \none[\cdot] & \none[\cdot] \\
\qm & \qm & 2 & \qm & \none[\cdot] & \none[\cdot]
\end{ytableau}
\ \right\} = V
\quad\text{as}\quad
  T = \left\{\ 
 \begin{ytableau}
    \none[\cdot] &  \none[\cdot] &  \none[\cdot] & \none[\cdot] & \none[\cdot] \\
  \none[\cdot] &  \none[\cdot] & \qm & \qm  & \none[\cdot]\\
   \none[\cdot] & 2 & 3& 5 & \none[\cdot] \\
\qm & \qm & 2 & \qm  & \none[\cdot]
\end{ytableau}
\right\}.
  \]

\item[(D1)] 
If $U_{ii} \leq U_{in} < U_{i,i+1}$ (so that $x=i+1$) and row $i+1$ of $U$ is empty,
but 
the tableau $T$ is not increasing,
 then there is an edge $U \xrightarrow{(i,i+1)} V$ where
 $V$  is formed from $U$ 
 by moving box $(i,n)$ to $(m,i+1)$ and changing its value to $U_{i,i+1}$,
 as in the following picture:
 \[
 \ytableausetup{boxsize = .4cm,aligntableaux=center}
   U = \left\{\ 
 \begin{ytableau}
    \none[\cdot] &  \none[\cdot] &  \none[\cdot] & \none[\cdot] & \none[\cdot] & \none[\cdot]\\
  \none[\cdot] &  \none[\cdot] & \none[\cdot] & \none[\cdot] & \none[\cdot] & \none[\cdot]\\
   \none[\cdot] & 2 & \circled{4} & 5 & \none[\cdot] &3\\
\qm & \qm & 3 & \qm & \none[\cdot] & \none[\cdot]
\end{ytableau}
\ \right\}
\ \xrightarrow{(i,i+1)}\
\left\{\
 \begin{ytableau}
    \none[\cdot] &  \none[\cdot] &  4 & \none[\cdot] & \none[\cdot] 
    \\
  \none[\cdot] &  \none[\cdot] & \none[\cdot] & \none[\cdot] & \none[\cdot] 
  \\
   \none[\cdot] & 2 & 4 & 5 & \none[\cdot] 
   \\
\qm & \qm & 3 & \qm & \none[\cdot] 
\end{ytableau}
\right\} = V
\quad\text{as}\quad
  T = \left\{\ 
 \begin{ytableau}
    \none[\cdot] &  \none[\cdot] &  \none[\cdot] & \none[\cdot] & \none[\cdot] \\
  \none[\cdot] &  \none[\cdot] & \none[\cdot] & \none[\cdot]& \none[\cdot] \\
   \none[\cdot] & 2 & 3& 5 & \none[\cdot] \\
\qm & \qm & 3 & \qm & \none[\cdot]
\end{ytableau}
\right\}.
  \]
 \een
In the next three cases, assume $x=i$ so that $(i,i) \in U$ and $U_{in}<U_{ii}$.
  \ben
  \item[(D2)] If the entries $U_{in}$ and $U_{ii}$ have the same parity but
the tableau $T$ is not increasing,
  then  
  there is an edge $U \xrightarrow{(i,i)} V$ where
  $V$
 is formed from $U$
 by moving box $(i,n)$ 
to $(m, i + 1)$ and changing its value to $U_{ii}$, as in the following picture:
  \[
 \ytableausetup{boxsize = .4cm,aligntableaux=center}
   U = \left\{\ 
 \begin{ytableau}
  \none[\cdot] &  \none[\cdot] &   \none[\cdot] & \none[\cdot] & \none[\cdot]\\
    \none[\cdot] &  \none[\cdot] &    \none[\cdot] & \none[\cdot] & \none[\cdot]\\
   \none[\cdot] & 6 & 7   & \none[\cdot] &4\\
\qm & 4 & \qm &  \none[\cdot] & \none[\cdot]
\end{ytableau}
\ \right\}
\ \xrightarrow{(i,i)}\
\left\{\
 \begin{ytableau}
  \none[\cdot] &  \none[\cdot] &  6   & \none[\cdot] 
  \\
    \none[\cdot] &  \none[\cdot] &   \none[\cdot]  & \none[\cdot] 
    \\
   \none[\cdot] & 6 & 7   & \none[\cdot] 
   \\
\qm & 4 & \qm  &  \none[\cdot] 
\end{ytableau}
\right\} = V
\quad\text{as}\quad
  T = \left\{\ 
 \begin{ytableau}
    \none[\cdot] &  \none[\cdot] &  \none[\cdot] & \none[\cdot]\\
  \none[\cdot] &  \none[\cdot] & \none[\cdot] & \none[\cdot] \\
   \none[\cdot] & 4 & 7 & \none[\cdot]\\
\qm & 4 & \qm& \none[\cdot]
\end{ytableau}
\right\}.
  \]

 \item[(D3)] If the entries $U_{in}$ and $U_{ii}$ have the same parity and 
 the tableau $T$ is increasing,
   then 
    there is an edge $U \xrightarrow{(i,i)} V$ where
    $V$
 is  formed from $T$ by adding an outer box in column $i+1$ with value $U_{ii}$,
as in the following picture:
 \[
 \ytableausetup{boxsize = .4cm,aligntableaux=center}
   U = \left\{\ 
 \begin{ytableau}
  \none[\cdot] &  \none[\cdot] & \none[\cdot] & \none[\cdot] & \none[\cdot]\\
  \none[\cdot] &  \none[\cdot] & \none[\cdot] & \none[\cdot] & \none[\cdot]\\
  \none[\cdot] & 6 & 7 & \none[\cdot] &4 \\
 \qm & 3 & \qm & \none[\cdot]  & \none[\cdot]
\end{ytableau}
\ \right\}
\ \xrightarrow{(i,i)}\
\left\{\
 \begin{ytableau}
  \none[\cdot] &    \none[\cdot] & 6 & \none[\cdot] \\
 \none[\cdot] &  \none[\cdot] & \none[\cdot] & \none[\cdot] \\
  \none[\cdot] &  4 & 7 & \none[\cdot] \\
\qm & 3 & \qm & \none[\cdot] 
\end{ytableau}
\right\}=V
\quad\text{as}\quad
  T = \left\{\ 
 \begin{ytableau}
    \none[\cdot] &  \none[\cdot] &  \none[\cdot] & \none[\cdot]\\
  \none[\cdot] &  \none[\cdot] & \none[\cdot]& \none[\cdot]  \\
   \none[\cdot] & 4 & 7 & \none[\cdot]\\
\qm & 3 & \qm& \none[\cdot]
\end{ytableau}
\right\}.
  \]

\item[(D4)] If the entries $U_{in}$ and $U_{ii}$ have different parities,
then
    there is an edge $U \xrightarrow{(i,i)} V$ where
 $V$ is the tableau formed from $U$ 
  by moving box $(i,n)$ 
to $(m, i + 1)$
and changing its value to $U_{ii}+1$, as in the following picture:
 \[
 \ytableausetup{boxsize = .4cm,aligntableaux=center}
   U = \left\{\ 
 \begin{ytableau}
  \none[\cdot] &  \none[\cdot] & \none[\cdot] & \none[\cdot] & \none[\cdot]\\
  \none[\cdot] &  \none[\cdot] & \none[\cdot] & \none[\cdot] & \none[\cdot]\\
 \none[\cdot] & 4 & 6 & \none[\cdot] & 3 \\
 \qm & \qm & \qm & \none[\cdot] &\none[\cdot]
\end{ytableau}
\ \right\}
\ \xrightarrow{(i,i)}\
\left\{\
 \begin{ytableau}
  \none[\cdot] &  \none[\cdot] & 5 & \none[\cdot] \\
  \none[\cdot] &  \none[\cdot] & \none[\cdot] & \none[\cdot] \\
  \none[\cdot] & 4 & 6 & \none[\cdot] \\
 \qm & \qm & \qm & \none[\cdot]
\end{ytableau}
\right\} = V.
  \]
\een

For the last family of edges,
continue to suppose $U$ is a non-terminal shifted insertion state
that has $m-2$ rows and $n-2$ columns when its outer box is removed, but 
now assume that this outer box is $(m,j)$.
If $U_{mj}$ is maximal in its column
and $i \in \PP$ is minimal with $(i,j) \notin U$,
then the unique outgoing edge from $U$ is as follows:

\ben
\item[(C1)] If moving the outer box of $U$ to position 
$(i,j)$ yields an increasing shifted tableau $V$,
then there is an edge $U \xrightarrow{(i,j)} V$. 
 
   \item[(C2)] If moving the outer box of $U$ to position $(i,j)$ does not yield an increasing shifted tableau,
then 
    there is an edge $U \xrightarrow{(i,j)} V$ where
 $V$ is formed from $U$ by removing the outer box $(m,j)$.

  \een
Finally, suppose there exists a minimal index $x \in \PP$ with $(x,j) \in U$ and $U_{mj} < U_{xj}$.
For this case, define $T$ to be the tableau formed from $U$ by replacing the value 
in box $(x,j)$ by $U_{mj}$ and then removing the outer box $(m,j)$.
\ben
\item[(C3)] If 
the tableau $T$ is not increasing,
 then there is an edge $U \xrightarrow{(x,j)} V$ where
 $V$  is formed from $U$ 
 by moving box $(m,j)$ to $(m,j+1)$ and changing its value to $U_{xj}$,
 as in the following picture:
 \[
 \ytableausetup{boxsize = .4cm,aligntableaux=center}
   U = \left\{\ 
 \begin{ytableau}
  \none[\cdot] &  \none[\cdot] &  4 & \none[\cdot] &  \none[\cdot]\\
    \none[\cdot] &  \none[\cdot] &  \none[\cdot] & \none[\cdot] &  \none[\cdot]\\
  \none[\cdot] &  4 & \circled{6} & \none[\cdot] & \none[\cdot]  \\
   \qm & \qm & 3 & \qm & \none[\cdot]  
\end{ytableau}
 \right\}
\ \xrightarrow{(x,j)}\
\left\{\
 \begin{ytableau}
  \none[\cdot] &  \none[\cdot] &  \none[\cdot] & 6 & \none[\cdot]\\
    \none[\cdot] &  \none[\cdot] &  \none[\cdot] & \none[\cdot] & \none[\cdot]\\
  \none[\cdot] &  4 & 6 & \none[\cdot] & \none[\cdot]\\
   \qm & \qm & 3 & \qm & \none[\cdot]  
\end{ytableau}
 \right\} = V
 \quad\text{as}\quad
  T = \left\{\ 
 \begin{ytableau}
  \none[\cdot] &  \none[\cdot] &  \none[\cdot] & \none[\cdot] & \none[\cdot] \\
    \none[\cdot] &  \none[\cdot] &  \none[\cdot] & \none[\cdot]& \none[\cdot] \\
  \none[\cdot] &  4 & 4 & \none[\cdot]  & \none[\cdot] \\
   \qm & \qm & 3 & \qm  & \none[\cdot] 
\end{ytableau}
\right\}.
  \]
Here and in the next case, the circled entry indicates the location of box $(x, j)$.

\item[(C4)] 
If 
the tableau $T$ is increasing,
then there is an edge $U \xrightarrow{(x,j)} V$ where
$V$ is formed from $T$ by adding an outer box in column $j+1$ with value $U_{xj}$,
 as in the following picture:
 \[
 \ytableausetup{boxsize = .4cm,aligntableaux=center}
   U = \left\{\ 
 \begin{ytableau}
  \none[\cdot] &  \none[\cdot] &  5 & \none[\cdot] &  \none[\cdot]\\
    \none[\cdot] &  \none[\cdot] &  \none[\cdot] & \none[\cdot] &  \none[\cdot]\\
   \none[\cdot] & 4 & \circled{6} & \qm & \none[\cdot]   \\
\qm & \qm &4 & \qm & \none[\cdot]  
\end{ytableau}
 \right\}
\ \xrightarrow{(x,j)}\
\left\{\
 \begin{ytableau}
  \none[\cdot] &  \none[\cdot] &  \none[\cdot] & 6 & \none[\cdot]\\
    \none[\cdot] &  \none[\cdot] &  \none[\cdot] & \none[\cdot] & \none[\cdot]\\
   \none[\cdot] & 4& 5 & \qm & \none[\cdot]   \\
\qm & \qm & 4 & \qm & \none[\cdot]  
\end{ytableau}
 \right\}
=V
 \quad\text{as}\quad
  T = \left\{\ 
 \begin{ytableau}
  \none[\cdot] &  \none[\cdot] &  \none[\cdot] & \none[\cdot] & \none[\cdot] \\
    \none[\cdot] &  \none[\cdot] &  \none[\cdot] & \none[\cdot] & \none[\cdot]\\
  \none[\cdot] &  4 & 5 & \qm & \none[\cdot] \\
   \qm & \qm & 4 & \qm   & \none[\cdot]
\end{ytableau}
\right\}.
  \]
 \een
 This completes our definition of the forward transition graph.
 
 We refer to edges of types (R1)-(R4), (D1)-(D4), and (C1)-(C4), respectively, as \emph{row transitions},
 \emph{diagonal transitions}, and \emph{column transitions} between shifted insertion states.
When the position labeling an edge is unimportant,
we simply 
write that $U\to V$ is \emph{forward transition}.

A unique path leads from any shifted insertion state
to a terminal state in the forward transition graph.
If a shifted insertion state with its outer box removed has $m-2$ rows and $n-2$ columns,
then this path consists of at most $\max\{m,n\}-1$ edges,
so the following is well-defined:

\begin{definition}\label{farrow-def}
Suppose $T$ is an increasing shifted tableau and $a \in \PP$.
Write $T\oplus a$ for the (initial) shifted insertion state formed by adding $a$
to the second unoccupied box in the first row of $T$.
If the maximal directed path from $T\oplus a$ to a terminal state
in the forward transition graph
 is 
\be
\label{Toplusa}
T\oplus a = U_0 \xrightarrow{(i_1,j_1)} U_1  \xrightarrow{(i_2,j_2)} U_2  \xrightarrow{(i_3,j_3)} \cdots  \xrightarrow{(i_l,j_l)} U_l 
\ee
then we define $T\farrow a $ to be the increasing shifted tableau $U_l$ and
call the sequence of positions
$(i_1,j_1), (i_2,j_2), \dots, (i_l,j_l)$ the \emph{bumping path} of inserting $a$ into $T$. 
\end{definition}

We refer to the operation transforming $(T,a) $ to $ T\farrow a $ as \emph{symplectic Hecke insertion}.
With slight abuse of notation, we sometimes refer to 
$(i_1,j_1), (i_2,j_2), \dots, (i_l,j_l)$ as the ``bumping path of $T \farrow a$''
and to the sequence of tableaux \eqref{Toplusa} as the ``insertion path of $T\farrow a$.''

\begin{example}
We have 
\[
\ytableausetup{boxsize = .4cm,aligntableaux=center}
 \begin{ytableau}
\none &  6  \\ 
4 & 5 
 \end{ytableau}\ \farrow 2
\ =\
  \begin{ytableau}
\none &  6 &\none  \\ 
2 &  4 &5 
 \end{ytableau}
\qquand
  \begin{ytableau}
\none &  4  \\ 
2 &  3 &  5
 \end{ytableau}\ \farrow 2
\ =\
  \begin{ytableau}
\none &  4 &\none  \\ 
2 &  3 &5
 \end{ytableau}
\]
 since the corresponding insertion paths in the forward transition graph are
\[
\ytableausetup{boxsize = .4cm,aligntableaux=center}
\ba
\left\{\ \begin{ytableau}
\none[\cdot] &  \none[\cdot] & \none[\cdot] & \none[\cdot] & \none[\cdot] \\ 
\none[\cdot] &  \none[\cdot] & \none[\cdot] & \none[\cdot] & \none[\cdot] \\ 
\none[\cdot] &  6 & \none[\cdot] & \none[\cdot] & \none[\cdot] \\ 
4 &  5 & \none[\cdot] & 2 & \none[\cdot] 
 \end{ytableau}\right\}
&\underset{\text{D3}}{\xrightarrow{(1,1)}}
\left\{\
  \begin{ytableau}
\none[\cdot] &  4 & \none[\cdot] & \none[\cdot] & \none[\cdot] \\ 
\none[\cdot] &  \none[\cdot] & \none[\cdot] & \none[\cdot] & \none[\cdot] \\ 
\none[\cdot] &  6 & \none[\cdot] & \none[\cdot] & \none[\cdot] \\ 
2 &  5 & \none[\cdot] & \none[\cdot] & \none[\cdot] 
 \end{ytableau}\right\}
\underset{\text{C4}}{\xrightarrow{(1,2)}}
\left\{\
  \begin{ytableau}
\none[\cdot] &  \none[\cdot] & 5 & \none[\cdot] & \none[\cdot] \\ 
\none[\cdot] &  \none[\cdot] & \none[\cdot] & \none[\cdot] & \none[\cdot] \\ 
\none[\cdot] &  6 & \none[\cdot] & \none[\cdot] & \none[\cdot] \\ 
2 &  4 & \none[\cdot] & \none[\cdot] & \none[\cdot] 
 \end{ytableau}\right\}
\underset{\text{C1}}{\xrightarrow{(1,3)}}
\left\{\
  \begin{ytableau}
\none[\cdot] &  \none[\cdot] & \none[\cdot] & \none[\cdot] & \none[\cdot] \\ 
\none[\cdot] &  \none[\cdot] & \none[\cdot] & \none[\cdot] & \none[\cdot] \\ 
\none[\cdot] &  6 & \none[\cdot] & \none[\cdot] & \none[\cdot] \\ 
2 &  4 &5 & \none[\cdot] & \none[\cdot] 
 \end{ytableau}\right\},
 \\
 \left\{\ \begin{ytableau}
\none[\cdot] &  \none[\cdot] & \none[\cdot] & \none[\cdot] & \none[\cdot] \\ 
\none[\cdot] &  \none[\cdot] & \none[\cdot] & \none[\cdot] & \none[\cdot] \\ 
\none[\cdot] &  4 & \none[\cdot] & \none[\cdot] & \none[\cdot] \\ 
2 &  3 &5 & \none[\cdot] &2 
 \end{ytableau}\ \right\}
&\underset{\text{R3}}{\xrightarrow{(1,2)}}
  \left\{\ \begin{ytableau}
\none[\cdot] &  \none[\cdot] & \none[\cdot] & \none[\cdot] & \none[\cdot] \\ 
\none[\cdot] &  \none[\cdot] & \none[\cdot] & \none[\cdot] & \none[\cdot] \\ 
\none[\cdot] &  4 & \none[\cdot] &  \none[\cdot] & 3\\ 
2 &  3 &5 & \none[\cdot] & \none[\cdot] 
 \end{ytableau}\ \right\}
\underset{\text{D4}}{\xrightarrow{(2,2)}}
  \left\{\ \begin{ytableau}
\none[\cdot] &  \none[\cdot] & 5 & \none[\cdot] & \none[\cdot] \\ 
\none[\cdot] &  \none[\cdot] & \none[\cdot] & \none[\cdot] & \none[\cdot] \\ 
\none[\cdot] &  4 & \none[\cdot] & \none[\cdot] & \none[\cdot] \\ 
2 &  3 &5 & \none[\cdot] & \none[\cdot] 
 \end{ytableau} \right\}
\underset{\text{C2}}{\xrightarrow{(2,3)}}
  \left\{\ \begin{ytableau}
\none[\cdot] &  \none[\cdot] & \none[\cdot] & \none[\cdot] & \none[\cdot] \\ 
\none[\cdot] &  \none[\cdot] & \none[\cdot] & \none[\cdot] & \none[\cdot] \\ 
\none[\cdot] &  4 & \none[\cdot] & \none[\cdot] & \none[\cdot] \\ 
2 &  3 &5 & \none[\cdot] & \none[\cdot] 
 \end{ytableau} \right\}.
 \ea
 \]
\end{example}

\subsection{Symplectic $K$-Knuth equivalence}\label{insert-sect}

Recall 
the definition of
$\simCK$ from before Lemma~\ref{rcc-lem}.
Define
\emph{$K$-Knuth equivalence} to be the strongest congruence $\simKK$ with 
$ cab \simKK acb$
and
$bca \simKK bac$
and
$aba \simKK bab$
and
$a \simKK aa$
for all positive integers $a,b,c \in \PP$ with $a<b<c$.

We have symplectic analogues of these relations.
Say that 
two  words  are connected by a \emph{symplectic Coxeter-Knuth move} 
if one word is obtained from the other in one of these ways:
\begin{itemize}
\item By interchanging the first two letters when these have the same parity.
\item If the first two letters are $a(a-1)$ for some $a\geq 2$, by changing these letters to $a(a+1)$.
\end{itemize}
Write $\simFCK$ (respectively, $\simFKK$) for the strongest equivalence relation 
that has $v \simFCK w$ (respectively, $v\simFKK w$) whenever $v$ and $w$ are  words
that are
connected by a symplectic Coxeter-Knuth move, or that satisfy $v \simCK w$ (respectively, $v \simKK w$).
We call these relations \emph{symplectic Coxeter-Knuth equivalence} and 
\emph{symplectic $K$-Knuth equivalence}.
For example, $
21 \simKK 211 \simFCK 231 \simCK 213.
$

The object of this section is to prove that if $T$ is an increasing shifted tableau
and $a \in \PP$ is a positive integer such that $\row(T) a$ is a symplectic Hecke word,
then $\row(T)a \simFKK \row(T\farrow a)$.
This will require several lemmas involving the following technical condition:

\begin{definition}\label{weak-def}
Let $T$ be a shifted insertion state with outer box $(i,j)$. 
Assume $T$ with its outer box  removed has $m-2$ rows and $n-2$ columns,
and set $T_{xy}:=\infty$ for all positions 
 $(x,y) \notin T$.
When $j=n$,  
we say that $T$ is \emph{weakly admissible} if the following conditions hold:
 \begin{itemize}
 \item Either $i =1$ or there exists a column $x\geq i$ with $T_{i-1,x} \leq T_{in} < T_{ix}$.
 \item If $T_{i-1,i} = T_{in}$ then $(i,i) \in T$.
\end{itemize}
 When $i=m$,  
 we say that $T$ is \emph{weakly admissible} if the following condition holds:
  \begin{itemize}
 \item Either $T_{mj}=T_{j-1,j}$ or there exists a row $x<j$ with $T_{x,j-1} \leq T_{mj} < T_{xj}$.
 \end{itemize}
Finally, we also say that any
terminal shifted insertion state is \emph{weakly admissible}.
\end{definition}

\begin{example}
All of the shifted insertion states in Example~\ref{states-ex} are weakly admissible.
The following states are \emph{not} weakly admissible:
\[
\ytableausetup{boxsize = .4cm,aligntableaux=center}
  \begin{ytableau}
\none[\cdot] & \none[\cdot]  & \none[\cdot]& \none[\cdot]& \none[\cdot] \\
\none[\cdot] & \none[\cdot]  & \none[\cdot]& \none[\cdot]& \none[\cdot] \\
\none[\cdot] & 4  &  \none[\cdot]& \none[\cdot]& 5 \\
2 & 3  &  6& \none[\cdot]& \none[\cdot]
 \end{ytableau}
 \qquad
   \begin{ytableau}
\none[\cdot] & \none[\cdot]  & \none[\cdot]& \none[\cdot] \\
\none[\cdot] & \none[\cdot]  & \none[\cdot]& \none[\cdot] \\
\none[\cdot] & 5  &  \none[\cdot]& 3 \\
2 & 4  &  \none[\cdot]& \none[\cdot]
 \end{ytableau}
        \qquad
  \begin{ytableau}
\none[\cdot] & \none[\cdot]  & \none[\cdot]& \none[\cdot] \\
\none[\cdot] & 4  & \none[\cdot]& \none[\cdot] \\
\none[\cdot] & \none[\cdot]  &  \none[\cdot]& \none[\cdot] \\
2 & 3  &  \none[\cdot]& \none[\cdot]
 \end{ytableau}
  \qquad
          \begin{ytableau}
2 & \none[\cdot]  & \none[\cdot]& \none[\cdot] \\
\none[\cdot] & \none[\cdot]  & \none[\cdot]& \none[\cdot] \\
\none[\cdot] &5  &  \none[\cdot]& \none[\cdot] \\
3 & 4  &  \none[\cdot]& \none[\cdot]
 \end{ytableau}
       \qquad
         \begin{ytableau}
\none[\cdot] & 4 & \none[\cdot]& \none[\cdot] \\
\none[\cdot] & \none[\cdot]  & \none[\cdot]& \none[\cdot] \\
\none[\cdot] &5  &  \none[\cdot]& \none[\cdot] \\
2 & 3  &  \none[\cdot]& \none[\cdot]
 \end{ytableau}
 \qquad
  \begin{ytableau}
\none[\cdot] & \none[\cdot]  & 2& \none[\cdot] \\
\none[\cdot] & \none[\cdot]  & \none[\cdot]& \none[\cdot] \\
\none[\cdot] & 4  &  \none[\cdot]& \none[\cdot] \\
2 & 3  &  \none[\cdot]& \none[\cdot]
 \end{ytableau}
 \]
\end{example}

Any initial insertion state is weakly admissible.
A weakly admissible insertion state 
cannot have its outer box in the first column.
This property naturally lends itself to inductive arguments.


\begin{proposition}\label{path-lem}
If $U \to V$ is an edge in the forward transition graph then $V$ is weakly admissible.
\end{proposition}

\begin{proof}
This is easy  to check directly from the definition of 
the forward transition graph.
\end{proof}

Our first two lemmas relate symplectic $K$-Knuth equivalence to row and column transitions.

\begin{lemma}\label{ck-lem1} Suppose $U \to V$ is a row transition
between weakly admissible shifted insertion states.
Then $\row(U) \simKK \row(V)$.
If $\row(U)$ is reduced then $\row(U) \simCK \row(V)$.
 \end{lemma}

\begin{proof}
If $U \to V$ is of type (R1) then $\row(U) = \row(V)$, and if
$U \to V$ is of type (R4) then it is easy to check that 
$\row(U) \simCK \row(V)$. 

Suppose the outer box of $U$ occurs in the first row and this row has the form
\[
\ytableausetup{boxsize = .6cm,aligntableaux=center}
 \begin{ytableau}
 c_1 &   c_2 &   \cdots &  c_p & \none[\cdot] & b 
 \end{ytableau}
 \]
where
 $c_1<c_2<\dots<c_p$ and $p \geq 0$.
If $U \to V$ is of type (R2), then we must have $c_p=b$,
so $\row(U)$ is not reduced and
 $\row(U) \simKK \row(V)$.
If $U \to V$ is of type (R3), then $c_i = b < c_{i+1}$ for a unique index $i \in [p-1]$, in which case $\row(U) \simCK \row(V)$.

 Next suppose the outer box of $U$ is in row $k>1$
 and rows $k-1$ and $k$ of $U$ have the form
 \[
 \ytableausetup{boxsize = .6cm,aligntableaux=center}
 \begin{ytableau}
\none[\cdot] &  c_1 &   c_2 &   \cdots & c_i  & \cdots &  c_p & \none[\cdot] & \none[\cdot] & \none[\cdot] & b  \\
  a_0 &   a_1 &  a_2 &   \cdots & a_i & \cdots &  a_p & \cdots & a_q & \none[\cdot] & \none[\cdot] 
 \end{ytableau}
 \]
 where $0 \leq p <q$ and  $a_0<a_1<\dots<a_q$ and $c_1<c_2<\dots < c_p$ and $a_i < c_i$ for all $i \in [p]$.
 If $p=0$ then $q>0$ and $a_1 \neq b$ since $U$ is weakly admissible,
 in which case the edge $U \to V$ is necessarily of type (R1).
 
Suppose $U \to V$ is of type (R2) so that $p>0$ and $c_p \leq b$. 
Since $U$ is weakly admissible, we must have   $a_{p+1} \leq b$.
If $c_p = b$ then $\row(U)$ is not reduced and 
we again have $\row(U) \simKK \row(V)$.
Assume $c_p < b$. Then $a_{p+1} = b$ since otherwise
moving $b$ to the column adjacent to $c_p$ would produce an increasing tableau.
In this case $\row(U)$ is not reduced since it contains the consecutive subword
$ba_0a_1\cdots a_p b$ where $a_0<a_1<\cdots<a_p<c_p<b$.
We conclude that $U \to V$ cannot be of type (R2) if $\row(U)$ is a reduced word.
To show that $\row(U) \simKK \row(V)$ in this case, it suffices to check that we have
$c_p b a_0a_1\cdots a_p b \simKK c_p a_0a_1\cdots a_p b$, or equivalently
that 
\be\label{e-eq} 
(p+1)(p+2) 123\cdots p (p+2)\simKK (p+1)123\cdots p (p+2)
\ee for any integer $p\geq 2$. Proving this is an instructive exercise; in brief,
one should
move $p+2$ all the way to the right, then apply a braid relation, then move $p+1$ all the way to right,
then apply another braid relation,
then use $p+1$ as a witness to commute $p$ and $p+2$, 
then combine the two final letters (which are both $p+2$),
and then finally move $p+1$ back to the start of the word. 
We conclude that $\row(U) \simKK \row(V)$.

Finally suppose $U \to V$ is of type (R3). Then $p\geq 2$ since there 
exists a minimal index $i \in [p-1]$
with $c_i \leq b < c_{i+1}$.
If $c_i = b$ then we have $\row(U) \simCK \row(V)$ as before,
so assume $c_i < b < c_{i+1}$.
Since replacing $c_{i+1}$ by $b$ does not produce an increasing tableau,
we must have $b \leq a_{i+1}$. Since $U$ is weakly admissible,  $b=a_{i+1}$ so $\row(U)$ is again not reduced
as it contains the consecutive subword $b a_0a_1\cdots a_i b$ where 
$a_0<a_1<\dots<a_i < c_i < b$. We conclude that if $\row(U)$ is reduced then
$\row(U) \simCK \row(V)$.
To show that $\row(U) \simKK \row(V)$,
we must check that 
\be
\label{ee-eq}
c_1c_2\cdots c_p b a_0a_1\cdots a_i b
\simKK 
c_{i+1} c_1c_2\cdots  c_p a_0a_1\cdots a_i b.
\ee
To prove 
this, let $w=c_{i+2}c_{i+3}\cdots c_p$.
We first observe that 
$
c_ic_{i+1}w b \simCK c_ic_{i+1}b w
\simKK c_ic_ic_{i+1}b w
\simCK c_ic_{i+1}c_ib w
\simCK c_{i+1}c_i c_{i+1}b w
\simCK c_{i+1}c_i c_{i+1}w b
\simCK c_i c_{i+1}c_i w b
\simCK c_i c_{i+1} w c_ib.
$
The identity \eqref{e-eq} implies that
$c_ib a_0a_1\cdots a_i b\simKK c_i a_0a_1\cdots a_i b$,
and it is easy to check that
$c_i c_{i+1}w c_i \simCK c_i c_{i+1}c_iw
\simCK c_{i+1}c_ic_{i+1} w = c_{i+1}c_ic_{i+1}\cdots c_p$
and
$
c_1c_2\cdots c_{i-1} c_{i+1} c_i \simCK c_{i+1} c_1c_1\cdots c_i.
$
Combining  these equivalences gives \eqref{ee-eq},
so $\row(U)\simKK \row(V)$ as desired.
This completes the proof of the lemma.
\end{proof}

\begin{lemma}\label{ck-lem2} Suppose $U \to V$ is a column transition between weakly admissible shifted insertion states.
Then $\col(U) \simKK \col(V)$.
If $\col(U)$ is reduced then $\col(U) \simCK \col(V)$.
 \end{lemma}

 \begin{proof}
 Suppose that $U$ with its outer box removed has $m-2$ rows and $n-2$
 columns, and that the outer box is $(m,j)$. 
 If there exists a row $x$ with $U_{x,j-1} \leq U_{mj} < U_{xj}$
 then the result follows by transposing the proof of
 Lemma~\ref{ck-lem1}; we omit the details.
 Assume instead that $U_{mj}=U_{j-1,j}$. If $(j,j) \notin U$
 then  $\col(U)$ is not reduced and 
 $U \to V$ is of type (C2), so $\col(U) \simKK \col(V)$. 
 If $(j,j) \in U$ then 
 $U\to V$ is of type (C3), 
 in which case $\col(U) \simCK \col(V)$.
 \end{proof}

We need a more intricate lemma to handle diagonal transitions.

\begin{lemma}\label{ck-lem3} 
Suppose $U \to V$ is a diagonal transition 
between weakly admissible shifted insertion states.
Assume $(i,n)$ is the outer box of $U$, so that $(i,i) \in U$.
\ben
\item[(a)]
If $U \to V$ is of type (D1)
and $U_{ii}\equiv U_{i,i+1} \modu2)$,
then
$\row(U) \simFKK \col(V)$.
 
 \item[(b)] If $U \to V$ is of type (D2), (D3), or (D4),
all entries on the main diagonal of $U$ have the same parity,
and either $U_{in}\equiv U_{ii}\modu 2)$ or  $U_{in} = U_{ii}-1$,
then $\row(U) \simFCK \col(V)$.

\een
In particular, if $\row(U)$ is a symplectic Hecke word
then $\row(U) \simFKK \col(V)$,
and if $\row(U)$ is a symplectic Hecke word that is also a reduced word then $\row(U)\simFCK \col(V)$.
 \end{lemma}

 \begin{proof}
First
 assume that $U \to V$ is of type (D1)
and $U_{ii}\equiv U_{i,i+1} \modu2)$.
 Let $c_1 = U_{ii}$, $c_2=U_{i,i+1}$, and $b = U_{in}$.
 Then $c_1\leq b < c_2$ and row $i+1$ of $U$ is empty, so 
  $V$ is formed from $U$ by removing box $(i,n)$
 and adding $c_2$ to an outer box in column $i+1$, as in the following picture:
 \[
  \ytableausetup{boxsize = .6cm,aligntableaux=center}
 U = \left\{\ 
 \begin{ytableau}
    \none[\cdot] &  \none[\cdot] &  \none[\cdot] & \none[\cdot] & \none[\cdot] & \none[\cdot]\\
  \none[\cdot] &  \none[\cdot] & \none[\cdot] & \none[\cdot] & \none[\cdot] & \none[\cdot]\\
   \none[\cdot] & c_1 & c_2 & \cdots & \none[\cdot] &b\\
\qm & \qm & \qm &\cdots & \none[\cdot] & \none[\cdot]
\end{ytableau}
\ \right\}
 \xrightarrow{(i,i+1)} 
 \left\{\ 
 \begin{ytableau}
    \none[\cdot] &  \none[\cdot] &  c_2 & \none[\cdot] & \none[\cdot] & \none[\cdot]\\
  \none[\cdot] &  \none[\cdot] & \none[\cdot] & \none[\cdot] & \none[\cdot] & \none[\cdot]\\
   \none[\cdot] & c_1 & c_2 & \cdots & \none[\cdot] &\none[\cdot]\\
\qm & \qm & \qm &\cdots & \none[\cdot] & \none[\cdot]
\end{ytableau}
 \right\} = V.
\]
 If $i=1$ then we must have $c_1=b$ and it follows that $\row(U) \simCK \row(V)$. If $i>1$ then
 the last paragraph of the proof of Lemma~\ref{ck-lem1} implies
 that $\row(U) \simKK \row(V)$, and that if $\row(U)$ is reduced then $\row(U)\simCK \row(V)$.

The word $\row(V)$ begins with $c_2c_1c_2$.
Suppose $c_1$ and $c_2$ have the same parity.
This must hold
if
$\row(U)$ is a symplectic Hecke word,
since then $c_1$ and $c_2$ must both be even
by Theorem~\ref{fpf-hecke-char-thm}.
In any case,  $\row(V)$ is then unreduced so $\row(U)$ is also unreduced.
Let $T$ be the increasing shifted tableau formed by removing the outer box of $V$.
As $\row(V) \simFKK \row(T)$ and $\col(V) \simKK \col(T)$ and $\row(T) \simCK \col(T)$  by Corollary~\ref{t-cor}, we have
 $\row(U) \simKK \row(V) \simFKK \row(T) \simCK \col(T) \simKK \col(V)$.
 This proves part (a).
 
For part (b),  assume that 
 $U \to V$ is of type (D2), (D3), or (D4)
 and that $U$  has $m-2$ rows.
Let $a=U_{in}$ and $b_j = U_{jj}$ for $ j \in [m-2]$, so that $a < b_i$
 and $b_1<b_2<\dots <b_{m-2}$.
Suppose $b_1,b_2,\dots,b_{m-2}$ all have the same parity. 
Define $\tilde U$ to be the tableau formed from $U$ by doubling the row and column indices of all boxes
and then moving the outer box of $U$ to position $(2i-1,2i-1)$.
For example, writing $b:=b_i$, we might have
\be\label{pic-eq}
\ytableausetup{boxsize = .5cm,aligntableaux=center}
 U = \left\{\
  \begin{ytableau}
 \none[\cdot] &  \none[\cdot]  & \qm &  \none[\cdot]& \none[\cdot] \\
  \none[\cdot] & b & \qm &   \none[\cdot]& a \\
\qm & \qm & \qm &  \none[\cdot]& \none[\cdot]
\end{ytableau}
\ \right\}
\qquad\text{and}
\qquad
 \tilde U = \left\{
 \begin{ytableau}
  \none[\cdot] &    \none[\cdot] & \none[\cdot] &  \none[\cdot]  &  \none[\cdot]& \qm    \\
   \none[\cdot] &       \none[\cdot] &   \none[\cdot] &  \none[\cdot] &  \none[\cdot]&  \none[\cdot]  \\
    \none[\cdot] &   \none[\cdot] &   \none[\cdot] & b &  \none[\cdot]& \qm   \\
   \none[\cdot] &   \none[\cdot] &    a & \none[\cdot] & \none[\cdot] & \none[\cdot]      \\
  \none[\cdot] & \qm &    \none[\cdot] & \qm &  \none[\cdot] & \qm  \\
  \none[\cdot] &  \none[\cdot] &  \none[\cdot] &  \none[\cdot] &  \none[\cdot] &  \none[\cdot]  
  \end{ytableau}
\  \right\}.
\ee
Let $T$ be the increasing shifted tableau formed 
from $U$ by omitting the outer box and the main diagonal and then translating all boxes left one column.
Clearly $\row(U) = \row(\tilde U)$,  and we have $ \nediag(T) \simCK\swdiag(T)$ by Corollary~\ref{t-cor}.
 There are two cases to consider.

First suppose $a\equiv b_i \modu 2)$ so that $U \to V$ is of type (D2) or (D3).
 If $(i-1,i) \in U$ then $U_{i-1,i} \leq a$, 
 so the tableau $\tilde U$ is increasing, row-diagonal-closed, and column-diagonal-closed.
Therefore $\row(U) = \row(\tilde U) \simCK \swdiag(\tilde U)$ 
 and $\col(\tilde U) \simCK \nediag(\tilde U)$
 by Lemmas~\ref{row-diagonal-closed-lem} and \ref{column-diagonal-closed-lem},
and it is easy to see that $\col(V) \simCK \col(\tilde U)$.
To show that $\row(U) \simFCK \col(V)$,
 it suffices to check that $\nediag(\tilde U) \simFCK \swdiag(\tilde U)$.
 Let $\delta = b_1\cdots b_{i-1} a b_i\cdots b_{m-2}$ be the word formed by reading the main diagonal of $\tilde U$ and let $\delta'$ be its reverse,
so that $\nediag(\tilde U) =\delta \cdot \nediag(T)$
 and $\swdiag(\tilde U) = \delta' \cdot \swdiag(T)$.
Since $ \nediag(T) \simCK\swdiag(T)$, 
it is enough to show that $\delta \simFCK \delta'$.
This is straightforward since $\delta$ is strictly increasing with all letters of the same parity.

Next suppose $a=b_i-1$ so that $U \to V$ is of type (D4). 
 Define $\tilde V$ to be the tableau formed from $\tilde U$
 by moving box $(2i-1,2i-1)$ to $(2i+1,2i+1)$ and adding 2 to its value.
 For example, if $U$ is as in our earlier picture \eqref{pic-eq} 
 where $b=b_i$ and $\tilde a = a+2=b+1$, then we would have
 \[
  \tilde V = \left\{ \begin{ytableau}
  \none[\cdot] &    \none[\cdot] & \none[\cdot] &  \none[\cdot]  &  \none[\cdot]& \qm    \\
   \none[\cdot] &       \none[\cdot] &   \none[\cdot] &  \none[\cdot] &   \tilde a &  \none[\cdot]  \\
    \none[\cdot] &   \none[\cdot] &   \none[\cdot] & b &  \none[\cdot]& \qm   \\
   \none[\cdot] &   \none[\cdot] &    \none[\cdot] & \none[\cdot] & \none[\cdot] & \none[\cdot]      \\
  \none[\cdot] & \qm &    \none[\cdot] & \qm &  \none[\cdot] & \qm  \\
  \none[\cdot] &  \none[\cdot] &  \none[\cdot] &  \none[\cdot] &  \none[\cdot] &  \none[\cdot]  
  \end{ytableau}\ \right\}.
  \]
  Observe that $\row(U) = \row(\tilde U)$ and $\col(V) = \col(\tilde V)$.
  Both $\tilde U$ and $\tilde V$ are increasing, row-diagonal-closed, and column-diagonal-closed,
  so
 Lemmas~\ref{row-diagonal-closed-lem} and \ref{column-diagonal-closed-lem}
imply  that $\row(\tilde U) \simCK \swdiag(\tilde U)$ 
 and $\col(\tilde V) \simCK \nediag(\tilde V)$.
 To show that $\row(U) \simFCK \col(V)$,
 it suffices to check that $\nediag(\tilde V) \simFCK \swdiag(\tilde U)$.
 Let $\delta = b_1\cdots b_i(b_i+1)b_{i+1}\cdots b_{m-2}$ and $\delta' = b_{m-2}\cdots b_i(b_i-1)b_{i-1} \cdots b_1$, so that $\nediag(\tilde V) =\delta \cdot \nediag(T)$
 and $\swdiag(\tilde U) = \delta' \cdot \swdiag(T)$.
It is enough to show that $\delta \simFCK \delta'$
since $ \nediag(T) \simCK\swdiag(T)$,
and this is again straightforward.
    In either case  $\row(U) \simFCK \col(V)$,
    which proves part (b).

 To prove the last assertion,
assume that $\row(U)$ is a symplectic Hecke word.
 We have already seen that if $U \to V$
 is of type (D1), then $\row(U)$ cannot be reduced
 but $\row(U) \simFKK \col(V)$.
Assume that $U \to V$
 is of type (D2), (D3), or (D4).
In view of part (b),
it is enough 
to show that the entries on main diagonal of $U$ are all even 
and that either $U_{in}$ is even or $U_{in}=U_{ii}-1$.
Define $a=U_{in}$ and $b_j = U_{jj}$ for $ j \in [m-2]$, so that $a < b_i$
 and $b_1<b_2<\dots <b_{m-2}$.
Since every letter preceding $b_j$ in $\row(U)$ is at least $b_j+2$, 
and since every letter preceding $a$ is at least $b_i$,
 Theorem~\ref{fpf-hecke-char-thm} implies that 
 $b_i,b_{i+1},\dots,b_{m-2}$ are all even and that
if $a$ is odd then $a=b_i-1$.
If $i>1$ then $b_{i-1} < U_{i-1,i} \leq a$ since $U$ is weakly admissible, so it 
follows similarly that
$b_1,b_2,\dots,b_{i-2}$ are also even.
The last thing to check is that $b_{i-1}$ is even when $i>1$.
If $b_{i-1} < a-1$ then this follows as before,
and if $b_{i-1}=a-1$ then we must have $U_{i-1,i} = a$, in which case 
$\row(U)$ has the form $v a(a-1)a w$ for some words $v$ and $w$,
where the smallest letter of $v$ is at least $a+1$.
Such a word has $ \row(U) \simA (a-1) v a (a-1) w$ so
Theorem~\ref{fpf-hecke-char-thm}  implies that $b_{i-1}=a-1$ is also even.
\end{proof}

We arrive at the main theorem of this section.

\begin{theorem}\label{farrow-thm}
Suppose $T$ is an increasing shifted tableau and $a \in \PP$ is 
such that $\row(T) a$ is a symplectic Hecke word.
The following properties then hold:
\ben
\item[(a)]  The tableau $T\farrow a$ is   increasing and shifted with $\row(T\farrow a) \simFKK \row(T) a$.
\item[(b)] If $\row(T) a$ is an FPF-involution word then $\row(T\farrow a) \simFCK \row(T) a$.
\een
\end{theorem}

\begin{proof}
Let $T\oplus a = U_0 \to U_1 \to \dots \to U_l = T\farrow a$ be the 
insertion path of $T\farrow a$.
The initial state $T\oplus a$ is weakly admissible, so each
$U_i$ is weakly admissible by Proposition~\ref{path-lem}.
The terminal state $T\farrow a$ is increasing and shifted by construction.
Lemmas~\ref{ck-lem1}, \ref{ck-lem2}, and \ref{ck-lem3}
imply that $\row(T) a = \row(T\oplus a) \simFKK \col(T \farrow a)$
and that $\row(T) a \simFCK \col(T\farrow a)$ if $\row(T) a$ is an FPF-involution word.
Since $\col(T\farrow a) \simCK \row(T\farrow a)$ by Corollary~\ref{t-cor}, the theorem follows.
\end{proof}

\subsection{Inverse transitions}\label{inverse-sect}

From any word  $w=w_1\cdots w_n$,
one may form a tableau 
$(\cdots( (\emptyset \farrow w_1) \farrow w_2)\farrow \cdots) \farrow w_n$.
If $w$ is a symplectic Hecke word, then only certain states arise when performing this sequence of insertions.
The results in this section will show that the following technical conditions precisely characterize such insertion states.

\begin{definition}\label{adm-def}
Let $T$ be a shifted insertion state with outer box $(i,j)$. 
Assume $T$ with box $(i,j)$ removed has $m-2$ rows and $n-2$ columns.
When $j=n$,  
we say that
 $T$ is \emph{admissible} if:
 \begin{itemize}
 \item $T$ is weakly admissible. 
 \item  The row reading word of $T$ is a symplectic Hecke word.
\end{itemize}
 When $i=m$, 
 we say that $T$ is \emph{admissible} if:
  \begin{itemize}
 \item $T$ is weakly admissible. 
  \item  The column reading word of $T$ is a symplectic Hecke word.
  \item If $T_{mj} = T_{j-1,j}$ then $(j,j) \notin T$ or $T_{mj}$ is odd, and if $T_{mj} = T_{x,j-1}$ then $x>1$.
 \end{itemize}
In addition, we say that a
terminal shifted insertion state is \emph{admissible}
if its row (equivalently, column) reading word is a symplectic Hecke word.
\end{definition}

The following propositions identify two important consequences of this definition.

\begin{proposition}\label{even-prop}
Suppose $T$ is an admissible shifted insertion state.
Assume that $T$ has $r$ rows with its outer box removed
(if one exists).
The diagonal entries $T_{ii}$ for $i\in[r]$ are then all even. 
\end{proposition}

\begin{proof}
If $T$ has no outer box, then $T$ is an increasing shifted tableau
and $\row(T)$ is a symplectic Hecke word.
In this case, it is easy to see that 
$\row(T)$ is equivalent under $\simA$
to a word beginning with $T_{ii}$ for each $i \in [r]$. All of 
these entries must be even by Theorem~\ref{fpf-hecke-char-thm}.

Assume $T$ has an outer box. If this box is in the last column,
then the result follows by
the argument
in the last paragraph of the proof of Lemma~\ref{ck-lem3}.
Let $m=r+2$
and suppose instead that
the outer box is $(m,j)$ for some column $j$.
Since $T$ is weakly admissible and since removing the outer box
leaves an increasing tableau,
it follows that
$\col(T)$ is equivalent under $\simA$
to a word beginning with $T_{ii}$
for each $i \in [r] - \{j\}$. 
Each of these numbers must be even by Theorem~\ref{fpf-hecke-char-thm}.
Assume $(j,j) \in T$ so that $T_{mj}  < T_{jj}$.
If $T_{mj} < T_{jj} - 1$ then the argument above shows that $T_{jj}$ is even.
Assume $a := T_{mj} = T_{jj} - 1$.
Since $T$ is weakly admissible,
this holds only if
$T_{mj}=T_{j-1,j}$,
but then 
$\col(T)$ has the form $v a(a+1)a w$
where every letter in $v$ is at most $a-1$,
so $\col(T) \simA v 
 (a+1)va(a+1)w$
and it follows by Theorem~\ref{fpf-hecke-char-thm} 
that $T_{jj} = a+1$ is again even.
\end{proof}

Suppose $T$ is a shifted insertion state that occupies $m-2$ rows
with its outer box removed. 
Set $\word(T):=\col(T)$ if the outer box of $T$ is in column $m$,
and set $\word(T):=\row(T)$ otherwise.

\begin{proposition}\label{adm-prop}
Suppose $U \to V$ is a forward transition between shifted insertion states.
Assume $U$ is admissible. Then $\word(U) \simFKK \word(V)$
and $V$ is admissible.
\end{proposition}

\begin{proof}
Proposition~\ref{path-lem} implies that
 $V$ is weakly admissible.
 In view of Proposition~\ref{even-prop}, it follows from
Lemmas~\ref{ck-lem1}, \ref{ck-lem2}, and \ref{ck-lem3}
 that $\word(U) \simFKK \word(V)$.
 This is enough to conclude that if $V$ is terminal then $V$ is admissible.
Assume $V$ is not terminal
and that $U$ and $V$ with their outer boxes removed have $m-2$ rows and $n-2$ columns.
It remains to check the 
 minor technical conditions in Definition~\ref{adm-def}.

Suppose
$U \to V$ is a row transition and the outer box of $V$ is $(i,n)$.
The only way we can have $V_{in} = V_{i-1,i}$ is if $U\to V$ is of type (R3),
in which case $(i,i) \in V$,
so $V$ is admissible.

Suppose next that $U \to V$ is a diagonal transition and the outer box of $V$ is $(m,j)$,
so that the outer box of $U$ is $(j-1,n)$.
Since a transition of type (D2) would require us to have $j-1\geq 2$,
we must have $V_{mj} \neq V_{1,j-1}$.
The only way we can have
$V_{mj} = V_{j-1,j}$
is if $U\to V$ is of type (D1) or (D4),
and in the first case $(j,j) \notin V$,
while in the second $V_{mj} =V_{j-1,j-1}+1$ must be odd.
We conclude that $V$ is admissible.

Finally, if $U \to V$ is a column transition and the outer box of $V$ is $(m,j)$,
then there is no way we can have $V_{mj} = V_{1,j-1}$
or
$V_{mj} = V_{j-1,j}$,
so $V$ is again admissible.
\end{proof}

Suppose $T$ is a shifted tableau.
A position $(i,j) \in \PP\times \PP$ is an \emph{outer corner} of $T$ if 
$(i,j) \notin T$,  either $i=j$ or $(i,j-1) \in T$, and  either $i=1$ or $(i-1,j) \in T$.
A position $(i,j) \in \PP\times \PP$ is an \emph{inner corner} of $T$
if $(i,j) \in T$ but $(i,j+1) \notin T$ and $(i+1,j) \notin T$.
The inner (outer) corners are exactly the positions that can be removed from (added to) $T$ while retaining 
a shifted tableau.

\begin{lemma}\label{corner-lem}
Suppose $U \xrightarrow{(i, j)} V$ is a forward transition between shifted insertion states 
where $U$ is admissible and $V$ is terminal. 
Then $U\to V$ is a row or column transition and $(i,j)$ is an inner or outer corner of $V$.
In addition, the following properties hold:
\ben
\item[(a)] If $U \to V$ is a row transition and $(i,j)$ is an outer corner of $V$,
then $i<j$.
\item[(b)] If $U \to V$ is a column transition and $(i,j)$ is an inner corner of $V$,
then $i<j$.
\item[(c)] If $U \to V$ is a column transition and $(i,j)$ is an outer corner of $V$,
then $i>1$.
 \een
\end{lemma}

\begin{proof}
The edge $U \to V$ cannot be a diagonal transition when $V$ is terminal.
Suppose $U \to V$ is a row transition. 
Since $V$ is terminal, this transition is either of type (R1) or (R2).
In the first case, 
$(i,j)$ is an inner corner of $V$ by definition,
while in the second case, 
$(i,j)$ must be an outer corner of $V$ since $U$ is weakly admissible.
The only way it can happen that $U \to V$ is of type (R2) and $i=j$
is if the value of the outer box of $U$ is equal to $U_{i-1,i}$ and $(i,i) \notin U$,
but then $U$ would not be admissible.

Suppose $U \to V$ is a column transition, necessarily of type (C1) or (C2).
It follows as in the previous paragraph that 
$(i,j)$ is an inner corner if $U \to V$ is of type (C1)
and 
an outer corner of $V$ if $U \to V$ is of type (C2).
The only way it can happen that $U \to V$ is of type (C2) and $i=1$
is if the value in the outer box of $U$ is $U_{1,j-1}$,
but then $U$ would not be admissible.
Similarly, the only way it can happen that $U \to V$ is of type (C1) and $i=j$
is if the outer box of $U$ is the largest value in its column and all preceding columns,
but then $U$ would not even be weakly admissible.
\end{proof}

By Proposition~\ref{adm-prop},
the family of admissible shifted insertion states spans a subgraph 
of the forward transition graph.
We introduce a second directed graph on these states, which we call the
\emph{inverse transition graph}.
We indicate that an edge goes from a state $V$ to $U$ in this new graph 
by writing $V \leadsto U$, and refer to such edges as \emph{inverse transitions}.
It will turn out that
the inverse transition graph is exactly the graph obtained by reversing 
all edges between admissible
states in the forward transition graph.
This will not be obvious from the definitions, however.

For the duration of this section,
let $V$ be an
 admissible shifted insertion state.
If $V$ is initial then it has no outgoing edges
in the inverse transition graph.
When $V$ is not initial, we define the possible edges $V \leadsto U$ in the inverse transition graph
by a series of cases corresponding to  
the row, diagonal, and column transitions in the forward transition graph.

First suppose $V$ is a terminal state, i.e., an increasing shifted
tableau such that $\row(V)$ is a symplectic Hecke word.
In the inverse transition graph, $V$ has no incoming edges
but multiple outgoing edges, of the following types:
 \begin{itemize}
 
 \item[(iR1)] For each inner corner $(i,j)$ of $V$,
there is an edge $V  \leadsto U$
 where $U$ is formed from $V$ by moving box $(i,j)$ to an outer position in row $i$.
It is clear that $U$ is also admissible and that $U \xrightarrow{(i,j)} V$
is a row transition of type (R1).
  
   \item[(iR2)] For each outer corner $(i,j)$ of $V$ with $i<j$,
 there is an edge $V \leadsto U$
 where $U$ is formed from $V$
by adding an outer box in row $i$
 whose value is whichever of $V_{i-1,j}$ or $V_{i,j-1}$ is defined and larger, as in the following picture where box $(i,j)$ is circled:
 \[
  \ytableausetup{boxsize = .4cm,aligntableaux=center}
V = \left\{\ 
\begin{ytableau}
\none[\cdot]  & \none[\cdot]   &  \qm & \qm & \none[\cdot] & \none[\cdot] & \none[\cdot]  \\
\none[\cdot]   & \qm &  \qm  & 2& \none[\circled{$\cdot$}] & \none[\cdot] & \none[\cdot]  \\
 \qm & \qm & \qm  & \qm  & 6& \none[\cdot] & \none[\cdot] \\ 
\end{ytableau}\ \right\}
\ \leadsto\ 
\left\{\
\begin{ytableau}
\none[\cdot]  & \none[\cdot]   &  \qm & \qm  & \none[\cdot] & \none[\cdot] & \none[\cdot]  \\
\none[\cdot]   & \qm &  \qm  & 2 & \none[\cdot]& \none[\cdot] & 6  \\
 \qm & \qm & \qm  & \qm  & 6 & \none[\cdot] & \none[\cdot]\\ 
\end{ytableau}
\ \right\} = U.
\]
In this case
$U \xrightarrow{(i,j)} V$ is a row transition of type (R2), so
we have $\row(U) \simKK \row(V)$ by Lemma~\ref{ck-lem1}.
It follows that $U$ is also admissible.

  \item[(iC1)] For each inner corner $(i,j)$ of $V$  with $i<j$,
  there is an edge $V \leadsto U$
   where $U$ is formed from $V$ by moving box $(i,j)$ to an outer position in column $j$.
 It is clear that $U$ is   admissible and that $U \xrightarrow{(i,j)} V$
is a column transition of type (C1).

 \item[(iC2)] For each outer corner $(i,j)$ of $V$ with $i>1$,
 there is an edge $V \leadsto U$
 where $U$ is formed from $V$
by adding an outer box in column $j$
 whose value is whichever of $V_{i-1,j}$ or $V_{i,j-1}$ is defined and larger, as in the following picture: 
 \[
  \ytableausetup{boxsize = .4cm,aligntableaux=center}
V = \left\{\ 
\begin{ytableau}
\none[\cdot]  & \none[\cdot]   &  \none[\cdot]  & \none[\cdot]  & \none[\cdot]  & \none[\cdot] \\
\none[\cdot]  & \none[\cdot]   &  \none[\cdot]  & \none[\cdot]  & \none[\cdot]  & \none[\cdot] \\
\none[\cdot]   & \qm &  5  & \none[\circled{$\cdot$}]  & \none[\cdot] & \none[\cdot] \\
 \qm & \qm & \qm &3  & \qm  & \qm 
\end{ytableau}\ \right\}
\ \leadsto\ 
\left\{\
\begin{ytableau}
\none[\cdot]  & \none[\cdot]  & \none[\cdot] &  5   & \none[\cdot]  & \none[\cdot] \\
\none[\cdot]  & \none[\cdot]   &  \none[\cdot]  & \none[\cdot]  & \none[\cdot]  & \none[\cdot] \\
\none[\cdot]   & \qm & 5 &  \none[\cdot]  & \none[\cdot]  & \none[\cdot]   \\
 \qm & \qm & \qm & 3  & \qm  & \qm
\end{ytableau}
\ \right\} = U.
\]
In this case 
$U \xrightarrow{(i,j)} V$ is a column transition of type (C2), so
we have $\col(U) \simKK \col(V)$ by Lemma~\ref{ck-lem2}. Since $i>1$, it follows that
$U$ is also admissible.

 \end{itemize}
To distinguish between these edges, we
write $V\rowbt U$ and $ V\colbt U$
to indicate the inverse transitions of type (iR1)-(iR2)
and (iC1)-(iC2), respectively,
corresponding to an inner or outer corner $(i,j)$ of the terminal state $V$.
 
From this point on, we assume that the admissible state $V$ is neither terminal nor initial.
All such states will have a unique outgoing edge in the inverse transition graph.
Suppose $V$ with its outer box removed has $m-2$ rows and $n-2$ columns.

First assume that the  outer box of $V$ is $(i,n)$ where $i>1$.
Since $V$ is weakly admissible, there exists a maximal $x \geq i$ with $V_{i-1,x} \leq V_{in}$,
and it must hold that $V_{in} < V_{ix}$ and $V_{in} < V_{i-1,x+1}$. 
The unique inverse transition starting at $V$ then has one of the following types:
\ben
\item[(iR3)] If $V_{i-1,x} = V_{in}$, then there is an edge $V\leadsto U$ where 
$U$ is formed from $V$ by moving box $(i,n)$ to $(i-1,n)$ and changing its value to 
be whichever of $V_{i-1,x-1}$ or $V_{i-2,x}$ is defined and larger, as in the following picture
where box $(i-1,x)$ is circled:
 \[
 \ytableausetup{boxsize = .4cm,aligntableaux=center}
V =\left\{\ \begin{ytableau}
    \none[\cdot] &  \none[\cdot] &  \none[\cdot] & \none[\cdot] & \none[\cdot] & \none[\cdot]\\
  \none[\cdot] &  \none[\cdot] & \qm & \qm & \none[\cdot] &4\\
   \none[\cdot] &2 & \circled{$4$} & 5 & \none[\cdot] & \none[\cdot] \\
\qm & \qm & 3 & \qm & \none[\cdot] & \none[\cdot]
\end{ytableau}\ \right\}
\ \leadsto\
\left\{\ \begin{ytableau}
    \none[\cdot] &  \none[\cdot] &  \none[\cdot] & \none[\cdot] & \none[\cdot] & \none[\cdot]\\
  \none[\cdot] &  \none[\cdot] & \qm & \qm & \none[\cdot] & \none[\cdot]\\
   \none[\cdot] & 2 & 4 & 5 & \none[\cdot] &3\\
\qm & \qm & 3 & \qm & \none[\cdot] & \none[\cdot]
\end{ytableau}\
\right\} = U.
  \]
Here and in the next case, the circled entry  indicates the location of box $(i-1,x)$.
Since $V$ is admissible, we must have $(i,i) \in V$,
so $U \xrightarrow{(i-1,x)} V$ is a row transition of type (R3).
It follows from Lemma~\ref{ck-lem1} that $U$ is also admissible.

\item[(iR4)] If $V_{i-1,x} < V_{in}$,  
then there is an edge $V\leadsto U$ where 
$U$ is formed from $V$ by moving box $(i-1,x)$ to $(i-1,n)$ and then box $(i,n)$ to $(i-1,x)$, 
as in the following picture:
\[
V = \left\{\
 \begin{ytableau}
    \none[\cdot] &  \none[\cdot] &  \none[\cdot] & \none[\cdot] & \none[\cdot] & \none[\cdot]\\
  \none[\cdot] &  \none[\cdot] &  \qm & \qm & \none[\cdot] &4\\
   \none[\cdot] & 2 & \circled{$3$} & 5 & \none[\cdot] & \none[\cdot] \\
\qm & \qm & 2 & \qm & \none[\cdot] & \none[\cdot]
\end{ytableau}
\ \right\}
\ \leadsto\
\left\{\
 \ytableausetup{boxsize = .4cm,aligntableaux=center}
 \begin{ytableau}
    \none[\cdot] &  \none[\cdot] &  \none[\cdot] & \none[\cdot] & \none[\cdot] & \none[\cdot]\\
  \none[\cdot] &  \none[\cdot] &  \qm & \qm & \none[\cdot] & \none[\cdot]\\
   \none[\cdot] & 2 & 4 & 5 & \none[\cdot] &3\\
\qm & \qm & 2 & \qm & \none[\cdot] & \none[\cdot]
\end{ytableau}
\ \right\}
=U.\]
In this case $U\xrightarrow{(i-1,x)} V$ is a row transition of type (R4), so it follows by Lemma~\ref{ck-lem1} that $U$ is also admissible.

\een
Next,
assume the outer box of $V$ is $(m,j)$
and $V_{j-1,j-1} \leq V_{mj}$. Since $V$ is weakly admissible,
we must have $ V_{mj} \leq V_{j-1,j}$. The unique edge $V\leadsto U$ is then of one of the following types:
\ben
\item[(iD1)] Suppose $V_{j-1,j-1} < V_{mj} = V_{j-1,j}$ and $V_{mj}$ is even,
so that $(j,j) \notin V$. 
There is then an edge $V \leadsto U$ where $U$ is formed from $V$
by moving box $(m,j)$ to $(j-1,n)$ and changing its value to whichever of 
$V_{j-1,j-1}$ or $V_{j-2,j}$ is defined and larger,
as in the following picture: 
\[
V = \left\{\
 \begin{ytableau}
  \none[\cdot] &  \none[\cdot] &  4   & \none[\cdot] & \none[\cdot]\\
    \none[\cdot] &  \none[\cdot] &   \none[\cdot] & \none[\cdot] & \none[\cdot]\\
   \none[\cdot] & 2 & 4  & \none[\cdot] & \none[\cdot] \\
\qm & \qm & 3 &  \none[\cdot] & \none[\cdot]
\end{ytableau}
\ \right\}
\ \leadsto\
\left\{\
 \ytableausetup{boxsize = .4cm,aligntableaux=center}
  \begin{ytableau}
  \none[\cdot] &  \none[\cdot] &   \none[\cdot] & \none[\cdot] & \none[\cdot]\\
    \none[\cdot] &  \none[\cdot] &    \none[\cdot] & \none[\cdot] & \none[\cdot]\\
   \none[\cdot] & 2 & 4   & \none[\cdot] &3\\
\qm & \qm & 3 &  \none[\cdot] & \none[\cdot]
\end{ytableau}
\ \right\}
=U.\]
Since
$U \xrightarrow{(j-1,j)} V$ is a diagonal transition of type (D1)
and $U_{j-1,j-1}=V_{j-1,j-1}$ and $U_{j-1,j}=V_{j-1,j}$ are both even
(by Proposition~\ref{even-prop}),
Lemma~\ref{ck-lem3}(a) implies
that $U$ is also admissible.

\item[(iD2)] Suppose $V_{j-1,j-1} = V_{mj}$, so that $j>2$ since $V$ is admissible.
If $V_{j-2,j-1}$ is even,
then there is an edge $V \leadsto U$ where $U$ is formed from $V$
by moving box $(m,j)$ to $(j-1,n)$ and changing its value to $V_{j-2,j-1}$,
as in the following picture:
\[
V = \left\{\
 \begin{ytableau}
  \none[\cdot] &  \none[\cdot] &  4   & \none[\cdot] & \none[\cdot]\\
    \none[\cdot] &  \none[\cdot] &   \none[\cdot] & \none[\cdot] & \none[\cdot]\\
   \none[\cdot] & 4 & 5  & \none[\cdot] & \none[\cdot] \\
\qm & 2 & \qm &  \none[\cdot] & \none[\cdot]
\end{ytableau}
\ \right\}
\ \leadsto\
\left\{\
 \ytableausetup{boxsize = .4cm,aligntableaux=center}
  \begin{ytableau}
  \none[\cdot] &  \none[\cdot] &   \none[\cdot] & \none[\cdot] & \none[\cdot]\\
    \none[\cdot] &  \none[\cdot] &    \none[\cdot] & \none[\cdot] & \none[\cdot]\\
   \none[\cdot] & 4 & 5   & \none[\cdot] &2\\
\qm & 2 & \qm &  \none[\cdot] & \none[\cdot]
\end{ytableau}
\ \right\}
=U.\]
In this case,
since $V_{j-1,j-1}$ is even by Proposition~\ref{even-prop}, $U \xrightarrow{(j-1,j-1)} V$ is a diagonal transition of type (D2), 
so it follows from Lemma~\ref{ck-lem3}(b) and Proposition~\ref{even-prop}
that $U$ is also admissible.

\item[(iD3)] If $V_{j-1,j-1} < V_{mj}< V_{j-1,j}$ and $V_{mj}$ is even,
then there is an edge $V \leadsto U$ where $U$ is formed from $V$ by moving box $(j-1,j-1)$ to $(j-1,n)$
and then box $(m,j)$ to $(j-1,j-1)$,
e.g.:
\[
V = \left\{\
\begin{ytableau}
  \none[\cdot] &    \none[\cdot] & 4 & \none[\cdot] & \none[\cdot]\\
 \none[\cdot] &  \none[\cdot] & \none[\cdot] & \none[\cdot] & \none[\cdot]\\
  \none[\cdot] &  2 & 5 & \none[\cdot] &\none[\cdot]\\
\qm & \qm & \qm & \none[\cdot] & \none[\cdot] 
\end{ytableau}
\ \right\}
\ \leadsto\
\left\{\
 \begin{ytableau}
  \none[\cdot] &  \none[\cdot] & \none[\cdot] & \none[\cdot] & \none[\cdot]\\
  \none[\cdot] &  \none[\cdot] & \none[\cdot] & \none[\cdot] & \none[\cdot]\\
  \none[\cdot] & 4 & 5 & \none[\cdot] & 2 \\
 \qm & \qm & \qm & \none[\cdot]  & \none[\cdot]
\end{ytableau}
\ \right\}
=U.\]
In this case, since $V_{j-1,j-1}$ is even by Proposition~\ref{even-prop},  $U \xrightarrow{(j-1,j-1)} V$ is a diagonal transition of type (D3), 
so it follows from Lemma~\ref{ck-lem3}(b) and Proposition~\ref{even-prop}
that $U$ is also admissible.

\item[(iD4)] If $V_{j-1,j-1} < V_{mj}$
and $V_{mj}$ is odd,
then there is an edge $V \leadsto U$ where $U$ is formed from $V$ 
by moving box $(m,j)$ to $(j-1, n)$
and changing its value to $V_{j-1,j-1}-1$,
 as in this picture: 
\[
 \ytableausetup{boxsize = .4cm,aligntableaux=center}
V = \left\{\
 \begin{ytableau}
  \none[\cdot] &  \none[\cdot] & 3 & \none[\cdot] & \none[\cdot]\\
  \none[\cdot] &  \none[\cdot] & \none[\cdot] & \none[\cdot] & \none[\cdot]\\
  \none[\cdot] & 2 & 5 & \none[\cdot] &\none[\cdot]\\
 \qm & \qm & \qm & \none[\cdot] & \none[\cdot] 
\end{ytableau}
\ \right\}
\ \leadsto\
\left\{\
  \begin{ytableau}
  \none[\cdot] &  \none[\cdot] & \none[\cdot] & \none[\cdot] & \none[\cdot]\\
  \none[\cdot] &  \none[\cdot] & \none[\cdot] & \none[\cdot] & \none[\cdot]\\
 \none[\cdot] & 2 & 5 & \none[\cdot] & 1 \\
 \qm & \qm & \qm & \none[\cdot] &\none[\cdot]
\end{ytableau}
\ \right\}
=U.\]
In this case,
 $V_{j-1,j-1}$ is even by Proposition~\ref{even-prop}.
By Theorem~\ref{fpf-hecke-char-thm},
we must have $V_{mj} = V_{j-1,j-1}+1$ since $\col(V)$ is a symplectic Hecke word.
Therefore $U \xrightarrow{(j-1,j-1)} V$ is a diagonal transition of type (D4),
so Lemma~\ref{ck-lem3}(b) and Proposition~\ref{even-prop} imply that $U$ is admissible.

\item[(iC3a)] Suppose $V_{j-1,j-1} = V_{mj}$, so that $j>2$ since $V$ is admissible.
If  $V_{j-2,j-1}$ is odd,
then there is an edge $V \leadsto U$ where $U$ is formed from $V$ 
by moving box $(m,j)$ to $(m,j-1)$ and changing its value to $V_{j-2,j-1}$, as in the following picture: 
\[
 \ytableausetup{boxsize = .4cm,aligntableaux=center}
V =\left\{\ 
 \begin{ytableau}
  \none[\cdot] &  \none[\cdot] &  4   & \none[\cdot]  \\
    \none[\cdot] &  \none[\cdot] &   \none[\cdot] & \none[\cdot]  \\
   \none[\cdot] & 4 & 5  & \none[\cdot]  \\
\qm & 3 & \qm & \qm
\end{ytableau}
\ \right\}
\ \leadsto\
\left\{\
  \begin{ytableau}
  \none[\cdot] & 3 &   \none[\cdot] & \none[\cdot]  \\
    \none[\cdot] &  \none[\cdot] &    \none[\cdot] & \none[\cdot] \\
   \none[\cdot] & 4 & 5   & \none[\cdot]  \\
\qm & 3 & \qm & \qm  
\end{ytableau}
\ \right\}
 = U.
\]
In this case 
$U\xrightarrow{(j-1,j-1)} V$ is a column transition of type (C3), so
Lemma~\ref{ck-lem2} implies that $\col(U) \simCK \col(V)$.
Although $U_{m,j-1} = U_{j-2,j-1}$ and $(j-1,j-1) \in U$,
the number $U_{m,j-1}$ is odd, so $U$ is also admissible.

\een
Finally, assume the outer box of $V$ is $(m,j)$
and $V_{mj} < V_{j-1,j-1}$.
Since $V$ is weakly admissible, there exists a maximal row $x < j-1$ with 
$V_{x,j-1} \leq V_{mj}$, and it must hold that $V_{mj} < V_{xj}$ and $V_{mj} < V_{x+1,j-1}$. 
The unique 
inverse transition $V \leadsto U$ is then of one of the following types:
\ben
\item[(iC3b)] Suppose $V_{x,j-1} = V_{mj}$, so that $x>1$ since $V$ is admissible.
There is then an edge $V \leadsto U$ where $U$ is formed from $V$  by moving box $(m,j)$ to $(m,j-1)$ and changing its value to be
whichever of $V_{x-1,j-1}$ or $V_{x,j-2}$ is defined and larger, as in the following picture:
\[
 \ytableausetup{boxsize = .4cm,aligntableaux=center}
V =\left\{\ 
 \begin{ytableau}
  \none[\cdot] &  \none[\cdot] & \none[\cdot] & 3   & \none[\cdot]  \\
    \none[\cdot] &  \none[\cdot] & \none[\cdot]  &   \none[\cdot] & \none[\cdot]  \\
   \none[\cdot] & 2 & \circled{$3$} & 4  & \none[\cdot]  \\
\qm & \qm & 1 & \qm & \qm
\end{ytableau}
\ \right\}
\ \leadsto\
\left\{\
 \begin{ytableau}
  \none[\cdot] &  \none[\cdot] &  2   & \none[\cdot] &\none[\cdot]  \\
    \none[\cdot] &  \none[\cdot] & \none[\cdot]  &   \none[\cdot] & \none[\cdot]  \\
   \none[\cdot] & 2 & 3 & 4  & \none[\cdot]  \\
\qm & \qm & 1 & \qm & \qm
\end{ytableau}
\ \right\}
 = U.
\]
The circled entry indicates the location of box $(x,j-1)$.
In this case $U\xrightarrow{(x,j-1)} V$ is a column transition of type (C3),
so it follows from Lemma~\ref{ck-lem2} that $U$ is also admissible.

\item[(iC4)] If $V_{x,j-1} < V_{mj}$
then there is an edge $V \leadsto U$ where $U$ is formed from $V$  
by moving box $(x,j-1)$ to $(m,j-1)$ and then box $(m,j)$ to $(x,j-1)$, as in the following picture:
\[
 \ytableausetup{boxsize = .4cm,aligntableaux=center}
V =\left\{\ 
 \begin{ytableau}
  \none[\cdot] &  \none[\cdot] & \none[\cdot] & 4   & \none[\cdot]  \\
    \none[\cdot] &  \none[\cdot] & \none[\cdot]  &   \none[\cdot] & \none[\cdot]  \\
   \none[\cdot] & 2 & \circled{$3$} & 5  & \none[\cdot]  \\
\qm & \qm & 2 & \qm & \qm
\end{ytableau}
\ \right\}
\ \leadsto\
\left\{\
 \begin{ytableau}
  \none[\cdot] &  \none[\cdot] &  3   & \none[\cdot] &\none[\cdot]  \\
    \none[\cdot] &  \none[\cdot] & \none[\cdot]  &   \none[\cdot] & \none[\cdot]  \\
   \none[\cdot] & 2 & 4 & 5  & \none[\cdot]  \\
\qm & \qm & 2 & \qm & \qm
\end{ytableau}
\ \right\}
 = U.
\]
The circled entry indicates the location of box $(x,j-1)$.
In this case $U\xrightarrow{(x,j-1)} V$ is a column transition of type (C4),
so it follows from Lemma~\ref{ck-lem2} that $U$ is also admissible.
\een
This completes our definition of the inverse transition graph.

\begin{example} 
The four paths 
in the inverse transition graph starting at $\
\ytableausetup{boxsize = .4cm,aligntableaux=center}
 \begin{ytableau}
  \none & 4 & 5 \\
  2 & 3 & 4 
  \end{ytableau} 
\ $
are
{
\ben
\item[]
$\ytableausetup{boxsize = .4cm,aligntableaux=center}
 \begin{ytableau}
  \none & 4 & 5 \\
  2 & 3 & 4 
  \end{ytableau}\
\ 
\overset{(2,3)}{\underset{\row}{\leadsto}}
\
   \left\{\ \begin{ytableau}
   \none[\cdot] &  \none[\cdot] &  \none[\cdot] &  \none[\cdot] &  \none[\cdot]\\
   \none[\cdot] &  \none[\cdot] &  \none[\cdot] &  \none[\cdot] &  \none[\cdot]\\
  \none[\cdot] & 4 & \none[\cdot]  & \none[\cdot]  & 5 \\
  2 & 3 & 4 & \none[\cdot]  & \none[\cdot] 
  \end{ytableau}\ \right\}
\ 
\underset{\text{iR4}}\leadsto
\
   \left\{\ \begin{ytableau}
    \none[\cdot] &  \none[\cdot] &  \none[\cdot] &  \none[\cdot] &  \none[\cdot]\\
   \none[\cdot] &  \none[\cdot] &  \none[\cdot] &  \none[\cdot] &  \none[\cdot]\\
   \none[\cdot]  & 4  &  \none[\cdot]  &  \none[\cdot] &  \none[\cdot]  \\
  2 & 3 & 5 &  \none[\cdot]  & 4
  \end{ytableau}\ \right\},
$

\item[]
$
 \begin{ytableau}
  \none & 4 & 5 \\
  2 & 3 & 4 
  \end{ytableau}\
\ 
\overset{(1,4)}{\underset{\row}{\leadsto}}
\
   \left\{\ \begin{ytableau}
\none[\cdot] &  \none[\cdot] &  \none[\cdot] &  \none[\cdot] &  \none[\cdot] \\
\none[\cdot] &  \none[\cdot] &  \none[\cdot] &  \none[\cdot] &  \none[\cdot] \\
\none[\cdot] & 4 & 5  & \none[\cdot] & \none[\cdot]\\
  2 & 3 & 4 & \none[\cdot] & 4
  \end{ytableau}\ \right\},
$

\item[]
$
 \begin{ytableau}
  \none & 4 & 5 \\
  2 & 3 & 4 
  \end{ytableau}\
\ 
\overset{(2,3)}{\underset{\col}{\leadsto}}
\
   \left\{\ \begin{ytableau}
     \none[\cdot] &  \none[\cdot] & 5 &  \none[\cdot] & \none[\cdot] \\
     \none[\cdot] &  \none[\cdot] &  \none[\cdot] &  \none[\cdot]&  \none[\cdot] \\
   \none[\cdot] & 4 &  \none[\cdot] &  \none[\cdot] &  \none[\cdot] \\
  2 & 3 & 4 &  \none[\cdot] &  \none[\cdot]
  \end{ytableau} \right\}
\ 
\underset{\text{iD4}}\leadsto
\
   \left\{\ \begin{ytableau}
     \none[\cdot] &  \none[\cdot] &  \none[\cdot] &  \none[\cdot] &  \none[\cdot] \\
     \none[\cdot] &  \none[\cdot] &  \none[\cdot] &  \none[\cdot] &  \none[\cdot] \\
   \none[\cdot] & 4&  \none[\cdot] & \none[\cdot] &  3  \\
  2 & 3 & 4 &  \none[\cdot] &  \none[\cdot]
  \end{ytableau}\ \right\}
\ 
\underset{\text{iR3}}\leadsto
\
   \left\{\ \begin{ytableau}
     \none[\cdot] &  \none[\cdot] &  \none[\cdot] &  \none[\cdot] &  \none[\cdot] \\
     \none[\cdot] &  \none[\cdot] &  \none[\cdot] &  \none[\cdot] &  \none[\cdot] \\
  \none[\cdot] & 4 & \none[\cdot] & \none[\cdot] & \none[\cdot]  \\
  2 & 3 & 4 & \none[\cdot] & 2
  \end{ytableau}\ \right\},
$

\item[]
$
 \begin{ytableau}
  \none & 4 & 5 \\
  2 & 3 & 4 
  \end{ytableau}\
\ 
\overset{(3,3)}{\underset{\col}{\leadsto}}
\
   \left\{\ \begin{ytableau}
     \none[\cdot] &  \none[\cdot] &  5&  \none[\cdot] &  \none[\cdot] \\
     \none[\cdot] &  \none[\cdot] &  \none[\cdot] &  \none[\cdot] &  \none[\cdot] \\
  \none[\cdot] & 4  & 5 &  \none[\cdot] &  \none[\cdot]\\
  2 & 3 & 4 &  \none[\cdot] &  \none[\cdot]
  \end{ytableau} \right\}
\ 
\underset{\text{iD1}}\leadsto
\
   \left\{\ \begin{ytableau}
     \none[\cdot] &  \none[\cdot] &  \none[\cdot] &  \none[\cdot] &  \none[\cdot] \\
      \none[\cdot] &  \none[\cdot] &  \none[\cdot] &  \none[\cdot] &  \none[\cdot] \\
  \none[\cdot] & 4& 5 &\none[\cdot] &  4  \\
  2 & 3 & 4&\none[\cdot]&\none[\cdot]
  \end{ytableau}\ \right\}
\ 
\underset{\text{iR3}}\leadsto
\
   \left\{\ \begin{ytableau}
     \none[\cdot] &  \none[\cdot] &  \none[\cdot] &  \none[\cdot] &  \none[\cdot] \\
      \none[\cdot] &  \none[\cdot] &  \none[\cdot] &  \none[\cdot] &  \none[\cdot] \\
  \none[\cdot] & 4 & 5 &\none[\cdot] & \none[\cdot] \\
  2 & 3 & 4 & \none[\cdot] & 3
  \end{ytableau}\ \right\}.
$
\een
}
\end{example}



\begin{theorem}\label{main-tech-thm}
Let $U$ and $V$ be admissible insertion states.
Then $U \to V$ is a forward transition if and only if $V \leadsto U$ is an inverse transition.
If $V$ is terminal,
then
$U\xrightarrow{(i,j)} V$ is a row (respectively, column) transition
if and only if 
$V \rowbt U$
(respectively, $V\colbt U$)
 is an inverse transition.
 \end{theorem}
 
 \begin{proof}
 We have already seen that if $V\leadsto U$ is an inverse transition then
$U\to V$ is a forward transition.
Likewise, if $V$ is terminal and
$V \rowbt U$
or $V\colbt U$
 is an inverse transition, then
$U\xrightarrow{(i,j)} V$ is either a row or
column transition, respectively.
To show the converse of  these statements, suppose $U \to V$ is a forward transition.

First assume $V$ is not terminal.
The edge $V\leadsto U$ is then of type (iR3), (iR4), (iD1), (iD2), (iD3), (iD4), (iC3a), (iC3b), or (iC4).
In each of these cases the required analysis is straightforward.
In detail,
suppose $V$ is as in case (iR3). Adopt the notation from that definition.
Then 
$U$ must have outer box $(i-1,n)$ and all entries of $U$ except the outer box
must be the same as in $V$.
In this case, we must have $V_{i-1,x-1} \leq U_{i-1,n}$ and (when $i>2$) $V_{i-2,x} \leq U_{i-1,n}$
since $U$ is weakly admissible. 
These inequalities cannot both be strict, so  $U$ is the unique state 
with $V\leadsto U$.

A similar argument shows that $V\leadsto U$ is an inverse transition if $V$ is
as in case (iC3b), or if $V$ is as in case (iC3a) and $U \to V$ is a column transition.
If $V$ is as in case (iC3a) and
$U \to V$ is not a column transition then $U\to V$ must be a diagonal transition,
but this is impossible since (in our notation when defining the inverse transition graph)
$U_{j-2,j-1}=V_{j-2,j-1}$ is odd (by hypothesis) and $U_{j-1,j-1}=V_{j-1,j-1}$ is even
(by Proposition~\ref{even-prop})
and the value in the outer box of $U$ cannot be less than $U_{j-2,j-1}$
(since $U$ is weakly admissible).

If $V$ is as in case (iR4) or (iC4), then there
is only one insertion state $U$, admissible or not, such that $U \to V$ is a forward transition.
We are left to examine cases (iD1)-(iD4).
Cases (iD1) and (iD3) are dual to case (iD4).
In all three cases, $U \to V$ cannot be a row or column transition
since $U$ is weakly admissible,
and the parity constraints on the  main diagonal and outer box of $V$
imply that there exists a unique diagonal transition $U \to V$.
Therefore $U$ must be the admissible state for which $V\leadsto U$ is an inverse transition.
Finally, if $V$ is as in case (iD2),
then there are exactly two weakly admissible states $U$ such $U \to V$ is a forward transition.
One of these is the state described in case (iD2).
The other is
formed by moving the outer box of $V$ to position $(m,j-1)$ and changing its value to $V_{j-2,j-1}$.
Although $U \to V$ is a valid column transition in this case,
the state $U$ is not admissible since $U_{m,j-1} = U_{j-2,j-1}$ is even and $(j-1,j-1)\in U$.
Hence, once again, $U$ must be the unique state for which $V\leadsto U$ is an inverse transition.

Finally suppose $V$ is terminal and $U \xrightarrow{(i,j)} V$ is a row transition,
so that the outer box of $U$ has the form $(i,n)$.
If $(i,j)$ is an inner corner of $V$
then obviously $U \to V$ is of type (R1) and $U$ is the state described in case (iR1).
If $(i,j)$ is an outer corner of $V$, then $i<j$ by Lemma~\ref{corner-lem} and $U \to V$ is of type (R2).
In this case, we must have $U_{i,x-1} \leq U_{in}$ and (when $i>1$) $U_{i-1,x} \leq U_{in}$
since $U$ is weakly admissible, but
these inequalities cannot both be strict, so $U$ is the state described in case (iR2). 
We conclude that $V \rowbt U$ is an inverse transition.
The argument to show that $V \colbt U$ is an inverse transition when 
$U$ is admissible and $U\xrightarrow{(i,j)} V$ is a column transition is similar.
 \end{proof}

 \subsection{Insertion tableaux and recording tableaux}
\label{rec-sect}

We may now define the {insertion} and {recording tableaux} of symplectic Hecke insertion.

\begin{definition}
\label{p-def}
For a word $w=w_1w_2\cdots w_n$,
let 
\[ \PF(w) = (\cdots( (\emptyset \farrow w_1) \farrow w_2)\farrow \cdots) \farrow w_n.\]
We call $\PF(w)$ the \emph{insertion tableau} of $w$ under symplectic Hecke insertion.
\end{definition}

By construction, $\PF(w)$ is an increasing shifted tableau with at most $\ell(w)$ boxes.
The definition of $\PF(w)$ makes sense for any word but the intended domain is the set of symplectic Hecke words.

\begin{example}
We compute some examples of insertion tableaux $\PF(w)$:
\begin{center}
\begin{tabular}{rrr}
$\ytableausetup{boxsize = .4cm,aligntableaux=center}
\PF(62)  =  \begin{ytableau}
2 & 6
 \end{ytableau},\ \
$
&
$\PF(46)  =  \begin{ytableau}
4 & 6
 \end{ytableau},\ \
$
&
$\PF(67)  =  \begin{ytableau}
6 & 7
 \end{ytableau},
$
\\
\\
$\PF(6224)  =  \begin{ytableau}
\none & 6  \\ 
2 & 4 
 \end{ytableau},\ \
$
&
$\PF(4626)  =  \begin{ytableau}
2 & 4 & 6
 \end{ytableau},\ \
$
&
$\PF(6752)  =  \begin{ytableau}
2 & 6 & 7
 \end{ytableau},
$
\\
\\
$
\PF(622453)  =  \begin{ytableau}
\none & 4 & 6 \\ 
2 & 3 & 5
 \end{ytableau},\ \
$
&
$\PF(462634)  =  \begin{ytableau}
\none & 4 & 6 \\ 
2 & 3 & 4
 \end{ytableau},\ \
$
&
$\PF(675245)  =  \begin{ytableau}
\none & 6 & 7 \\ 
2 & 4 & 5
 \end{ytableau}.
$
\end{tabular}
\end{center}
\end{example}

As usual, the \emph{last box} in column $j$ (respectively, row $i$) of a set-valued tableau $T$ refers to the position
$(i,j)$ where $i \in \PP$ is maximal (respectively, $j \in \PP$ is maximal) with $(i,j) \in T$.

\begin{definition}
\label{q-def}
For a symplectic Hecke word $w=w_1w_2\cdots w_n$, we inductively define a set-valued tableau $\QF(w)$.
Let $\QF(\emptyset) =\emptyset$ and assume $n>0$.
Let $(i,j)$ be the label of the last transition in
the insertion path of $\PF(w_1\cdots w_{n-1}) \farrow w_n$.
Form $\QF(w)$ from $\QF(w_1\cdots w_{n-1})$ as follows:
\begin{itemize}
\item[(a)] If the last transition is of type (R1) then
add $n$ to box $(i,j)$. 

\item[(b)] If the last transition is of type (C1) then
add $n'$ to box $(i,j)$. 

\item[(c)] If the last transition is of type (R2) then
add $n$ to the last box in column $j-1$. 

\item[(d)] If the last transition is of type (C2) then
add $n'$ to the last box in row $i-1$.
\end{itemize}
We call $\QF(w)$ the \emph{recording tableau} of $w$ under symplectic Hecke insertion.
\end{definition}

Lemma~\ref{corner-lem} ensures that $\QF(w)$ is well-defined for any symplectic Hecke word $w$.
By construction, $\QF(w)$ is a standard shifted set-valued tableau of length $|\QF(w)| = \ell(w)$.

\begin{example}
The symplectic Hecke words $w$ of length 4 with
$\PF(w)=
\ytableausetup{boxsize = .4cm,aligntableaux=center}
\begin{ytableau}
\none & 4 \\
2 & 3    
\end{ytableau}$
are 2243, 2433, 2443, 2423, 4423, 4233, 4243, and 4223.
Their recording tableaux are
{\footnotesize
\begin{center}
\begin{tabular}{cccc}
$\QF(2243)
=
 \ytableausetup{boxsize = .6cm,aligntableaux=center}
\begin{ytableau}
\none & 4 \\
12 & 3    
\end{ytableau}
$,
&
$\QF(2433)
=
\begin{ytableau}
\none & 34 \\
1 & 2   
\end{ytableau}
$,
&
$\QF(2443)
=
\begin{ytableau}
\none & 4 \\
1 & 23    
\end{ytableau}
$,
&
$\QF(2423)
=
\begin{ytableau}
\none & 4 \\
1 & 23'    
\end{ytableau}
$,
\\ \\
$\QF(4423)
=
\begin{ytableau}
\none & 4 \\
12 & 3'    
\end{ytableau}
$,
&
$\QF(4233)
=
\begin{ytableau}
\none & 34 \\
1  & 2'    
\end{ytableau}
$,
&
$\QF(4243)
=
\begin{ytableau}
\none & 4 \\
1  & 2'3    
\end{ytableau}
$,
&
$\QF(4223)
=
\begin{ytableau}
\none & 4 \\
1 & 2'3'    
\end{ytableau}
$.
\end{tabular}
\end{center}
}
\end{example}

We also refer to the operation 
$w\mapsto (\PF(w),\QF(w))$ as \emph{symplectic Hecke insertion}.
Before analyzing this map, we note two obvious corollaries of Theorem~\ref{farrow-thm}:

\begin{corollary}\label{fkk-cor}
If $w$ is a symplectic Hecke word
then $w \simFKK \row( \PF(w))$.
\end{corollary}

\begin{corollary}\label{reduced-cor}
If $w$ is an FPF-involution word
then $w \simFCK \row( \PF(w))$.
\end{corollary}

It follows that if $v$ and $w$ are symplectic Hecke words
(respectively, FPF-involution words)
with $\PF(v) = \PF(w)$,
then $v \simFKK w$ (respectively, $v\simFCK w$).
Two symplectic Hecke words $v$ and $w$
can have $v \simFKK w$ and $\PF(v) \neq \PF(w)$, however.
For example,
we have $265213 \simFKK 265231$ 
but
\[
  \ytableausetup{boxsize = .4cm,aligntableaux=center}
\PF(265213)  = 
\begin{ytableau}
\none & 6 & 7 \\
2 & 3 & 5 & 7
\end{ytableau}
\qquand
\PF(265231) = 
\begin{ytableau}
\none & 6  \\
2 & 3 & 6 & 7
\end{ytableau}.
\]
This pathology does not seem to arise for the relation $\simFCK$ restricted to FPF-involution words:

\begin{conjecture}\label{conj1}
If $v$ and $w$ are FPF-involution words
with $v \simFCK w$ then $\PF(v) = \PF(w)$.
\end{conjecture}

\begin{remark*}[Note added in proof]
Hiroshima has given a proof of this result in \cite{Hiroshima}.
\end{remark*}

Fix $z \in \FF$. We describe how to invert 
the operation $w \mapsto (\PF(w),\QF(w))$ for $w \in \cHfpf(z)$.
Let $P$ be an increasing shifted tableau,
let $Q$ be a standard set-valued tableau
with the same shape as $P$,
and let $w$ be a word such that $\row(P)w \in \cHfpf(z)$.
Suppose $Q$ has length $n>0$.
Then $ Q$ contains exactly one of  
$n$ or $n'$,
and this number must appear in an inner corner $(i,j)$.
Define $V_1$ to be the unique admissible shifted insertion state such that:
\begin{itemize}
\item if $\{n\} = Q_{ij}$ then $P \rowbt V_1$ is an inverse transition;
\item if $\{n'\} = Q_{ij}$ (so that $i<j$) then $P \colbt V_1$ is an inverse transition;
\item if $\{n\} \subsetneq Q_{ij}$
then $P \overset{(r,j+1)}{\underset{\row}{\leadsto}} V_1$ is an inverse transition,
where $r$ is the row of the
unique outer corner of $Q$ in column $j+1$;
\item if $\{n'\} \subsetneq Q_{ij}$  (so that $i<j$)
then  $P \overset{(i+1,s)}{\underset{\col}{\leadsto}} V_1$ is an inverse transition,
where $s$ is the column of the unique outer corner of $Q$ in row $i+1$.
\end{itemize}
Now let
$P  \leadsto V_1 \leadsto V_2 \leadsto \cdots \leadsto V_l$
be the maximal directed path in the inverse transition graph 
containing $P \leadsto V_1$.
The last state $V_l$ is initial,
so has the form $\hat P \oplus a$ for a shifted tableau $\hat P$
and an integer $a \in \PP$.
Set $\hat w = aw$, form $\hat Q$ from $Q$
by removing whichever of
$n$ or $n'$ appears,
and define 
\be\label{unbump-eq}
\unbump(P,Q,w):= (\hat P,\hat Q,\hat w).
\ee
The set-valued tableau $\hat Q$ is standard with length $n-1$ and the same shape as $\hat P$.
Theorem~\ref{main-tech-thm} implies that $P = \hat P \farrow a $,
so $\row(P) \simFKK \row(\hat P)a$  by Theorem~\ref{farrow-thm}
and $\row(\hat P)\hat w \in \cHfpf(z)$.
Thus, 
$(\hat P,\hat Q,\hat w)$ has the same properties as $(P,Q,w)$, so
we can iterate the operation $\unbump$.

\begin{definition}
Let $z \in \FF$.
Given an increasing shifted tableau $P$ with $\row(P) \in \cHfpf(z)$
and 
a standard set-valued tableau $Q$
with the same shape,
define $\wSp(P,Q)$ to be the word such that 
\be\label{unbump-eq2}
\underbrace{\unbump\circ \unbump \circ \cdots \circ \unbump}_{|Q|\text{ times}}(P, Q, \emptyset) = (\emptyset, \emptyset, \wSp(P,Q)).
\ee
\end{definition}

\begin{example}
The word $\wSp(P,Q)$ has length $|Q|$, so $\wSp(P,Q)=\emptyset$ when $P=Q=\emptyset$.
If
\[\ytableausetup{boxsize = .7cm,aligntableaux=center}
P = \begin{ytableau}
\none & 4 \\
2 & 3    
\end{ytableau}
\qquand
Q=
\begin{ytableau}
\none & 4 \\
1 & 2'3'    
\end{ytableau}
\]
then $\wSp(P,Q) = 4223$ since applying $\unbump$ four times  has the effect of mapping 
\[(P,Q,
\emptyset)
\mapsto
\ytableausetup{boxsize = .7cm,aligntableaux=center}
\ba
\(
\begin{ytableau}
2 & 4   
\end{ytableau}, \
\begin{ytableau}
1 & 2'3'    
\end{ytableau}, \
3 \)
\ytableausetup{boxsize = .5cm,aligntableaux=center}
\mapsto\(
\begin{ytableau}
2 & 4   
\end{ytableau}, \
\begin{ytableau}
1 & 2'
\end{ytableau}, \
23 \)
\mapsto\(
\begin{ytableau}
4   
\end{ytableau}, \
\begin{ytableau}
1  
\end{ytableau}, \
223 \)
\mapsto
(
\emptyset, 
\emptyset, 
4223 ).
\ea
\]
\end{example}

 A \emph{marked tableau} is a set-valued tableau whose entries are singletons,
 or equivalently a map from a finite subset of $\PP\times \PP$ to the marked alphabet $\MM = \{1'<1<2'<2<\dots\}$.

\begin{theorem}\label{fpf-bij-thm}
Let $z \in \FF$.
Then 
$w \mapsto (\PF(w), \QF(w))$ and $(P,Q) \mapsto \wSp(P,Q)$
are inverse bijections  between the set 
of symplectic Hecke words (respectively, FPF-involution words)
for $z$ of length $n \in \NN$
and the set of pairs
$(P,Q)$
where $P$ is an increasing shifted tableau 
with $\row(P) \in \cHfpf(z)$ 
 (respectively, $\row(P) \in \cRfpf(z)$) and
$Q$ is a standard shifted set-valued (respectively, marked) tableau with length $|Q|=n$ and the same shape as $P$.
\end{theorem}

\begin{proof}
Let $P$ be an increasing shifted tableau,
let $Q$ be a standard set-valued tableau
with the same shape as $P$,
and let $w=w_1w_2\cdots w_m$ be a word with $\row(P)w \in \cHfpf(z)$.
Suppose $m>0$ and $|Q|=n-1\geq 0$.
We may assume by induction that $(P,Q) = (\PF(v),\QF(v))$ for some symplectic Hecke word $v$.
Define  $\check P := P\farrow w_1 = \PF(vw_1)$ and form $\check Q$ from $Q$
according to the rules in Definition~\ref{q-def} so that $\check Q = \QF(v w_1)$.
Then set $\check w := w_2w_3\cdots w_m$ and define
\be\label{bump-eq}
\bump(P,Q,w):= (\check P,\check Q,\check w).
\ee
The set-valued tableau $\check Q$ is standard with the same shape as $\check P$.
Since $\row(\check P) \check w \simFKK \row(P) w_1 \check w = \row(P)w$
by Theorem~\ref{farrow-thm},
it holds that $\row(\check P) \check w \in \cHfpf(z)$.
We can therefore iterate the operation $\bump$,
and it is easy to see that if $w  \in \cHfpf(z)$ then
\be\label{bump-eq2}
\underbrace{\bump\circ \bump \circ \cdots \circ \bump}_{\ell(w)\text{ times}}(\emptyset, \emptyset, w) = (\PF(w), \QF(w), \emptyset).
\ee
Let $\cT_n^m$ be the set of triples $(P,Q,w)$
where
$P$ is an increasing shifted tableau,
$Q$ is a standard set-valued tableau of length $n$
with the same shape as $P$,
and $w$ is a word of length $m$ such that $\row(P)w \in \cHfpf(z)$.
The formulas \eqref{unbump-eq} and \eqref{bump-eq}
give well-defined maps $\bump: \cT_n^{m+1} \to \cT_{n+1}^m$ and $\unbump: \cT_{n+1}^{m} \to \cT_{n}^{m+1}$ 
for all $m,n \in \NN$.
In view of \eqref{unbump-eq2} and \eqref{bump-eq2}, it suffices to show that these maps are inverse bijections.
The hard work needed to check this claim has already been done, however:
what needs to be shown 
is equivalent to 
Theorem~\ref{main-tech-thm}.

Finally, observe that if $w \in \cRfpf(z)$ then we must have $\ell(w) = |\PF(w)|$ since $w \simFKK \row(\PF(w))$,
so $\QF(w)$ is a marked tableau. Conversely, if $w \in \cHfpf(z)$ but $\row(\PF(w)) \in \cRfpf(z)$ and $\QF(w)$ is a marked tableau then 
$|\QF(w)| = |\PF(w)|$
so 
$w \in \cRfpf(z)$
since $\ell(w) = |\QF(w)|$.
\end{proof}

%
%

In the following corollary, 
we say that a shifted tableau has shape $\lambda$
if its domain is the shifted Young diagram $\SY_\lambda = \{ (i,i+j-1) \in \PP\times \PP :  1 \leq j \leq  \lambda_i\}$.

\begin{corollary}
Fix $n \in 2\PP$ and let $z_{\max}  = n\cdots 321 \in \FF$ be the fixed-point-free involution with
$z_{\max}(i) = n+1-i$ for $i \leq n$ and $z_{\max}(i) = i - (-1)^i$ for $i>n$.
The map $w \mapsto \QF(w)$ is then a length-preserving bijection from symplectic Hecke words for $z_{\max}$
to standard shifted set-valued tableaux 
of shape $\lambda = (n-2,n-4,\dots,6,4,2)$.
Consequently,
the size of $\cRfpf(z_{\max})$ is the number of standard shifted marked tableaux of this shape.
\end{corollary}

One can compute $|\cRfpf(z_{\max})|$
using well-known hook length formulas; see \cite[Theorem 1.4]{HMP1}.

\begin{proof}
Consider the shifted tableau $T$ whose first row is $234\cdots (n-1)$, 
whose second row is $456\cdots(n-1)$, 
whose third row is $678\cdots(n-1)$, and so forth,
and whose last row is $(n-2)(n-1)$.
It is easy to check that $\row(T) \in \cRfpf(z_{\max})$.
It follows from Theorem~\ref{fpf-hecke-char-thm}
that every symplectic Hecke word for $z_{\max}$
has at least $(n-2) + (n-4) + \dots +2$ letters, each of which is at most $n-1$.
Since $T$ is the only increasing shifted tableau
with $(n-2) + (n-4) + \dots   +2$
boxes, with entries in $\{1,2,\dots,n-1\}$, and with no odd entries on the main diagonal,
we conclude that $T$ is the insertion tableau of every symplectic Hecke word for $z_{\max}$, so
the result follows from Theorem~\ref{fpf-bij-thm}.
\end{proof}

 \begin{example}\label{zmax-ex}
If $n=8$ and $w=42 6 1 75 3 4 2 1 3 2$ then $w\in \cRfpf(z_{\max})$ and
 \[  \ytableausetup{boxsize = .6cm,aligntableaux=center}
 (\PF(w),\QF(w)) =
 \(\
\begin{ytableau}
\none & \none & 6 & 7 \\
\none & 4 &5 &6 & 7 \\
2 & 3 & 4 &5 &6 & 7 
\end{ytableau},\ \ 
\begin{ytableau}
\none & \none & 8 & 11' \\
\none & 6 &7' &9' & 12' \\
1 & 2' & 3 &4' &5 & 10' 
\end{ytableau}
\
\).
\]
 \end{example}

A result of Sagan \cite{Sagan1980} describes a fast algorithm for sampling standard shifted marked tableaux
of a given shape
uniformly at random.
Combining this with the preceding corollary
gives an algorithm for generating FPF-involution words for $n\cdots 321 \in \FF$ uniformly at random.

There is a fascinating literature on the properties of random reduced words 
for   $n\cdots 321\in \SS$, called \emph{random sorting networks}
by Angel, Holroyd, Romik, and Vir\'ag in \cite{AHRV}.
The bijections in this article would make it possible to conduct a similar study of random (FPF-)involution words.

 \section{Variations}\label{var-sect}
 
 \subsection{Semistandard insertion}\label{des-sect}

Suppose $T$ is an increasing shifted tableau and $a \in \PP$.
Let $(i_1,j_1)$, $(i_2,j_2)$, \dots, $(i_l,j_l)$ be the bumping path resulting from inserting $a$ into $T$ to form $T \farrow a$,
 as described in  Definition~\ref{farrow-def}.
The next result shows that this sequence contains at most two diagonal positions, which must be consecutive.
We refer to the positions up to and including the first diagonal position
as \emph{row-bumped positions},
and to any subsequent positions
as \emph{column-bumped positions}.
If $(i_t,j_t)$ is a row-bumped position then  $i_t =t $,
while if $(i_t,j_t)$ is a column-bumped position then $ j_t = t$.
If $t \in [l-1]$ is the index of the last row-bumped position then
$(i_t,j_t) = (t,t)$ and
$j_{t+1} = t + 1$.

\begin{proposition}\label{bump-lem1}
Maintain the setup of the previous paragraph.
Suppose $t \in [l]$ is the index of the bumping path's last row-bumped position.
The following properties then hold:
\ben
\item[(a)] One has $j_1  \geq j_2 \geq \dots \geq j_t \geq t$ and if $t<l$ then $t+1\geq  i_{t+1} \geq i_{t+2} \geq \dots \geq i_l $.

\item[(b)] If $\row(T)a$ is an FPF-involution word and $t<l$, then $i_l < t+1$.

\item[(c)] 
If $(i,j)$ is column-bumped and $(i',j')$ is row-bumped then we do not have $i \leq i'$ and $j\leq j'$.
In other words,
no column-bumped position is weakly southwest of any row-bumped position.

\item [(d)] All positions in the bumping path are distinct. 
\een
 
\end{proposition}

\begin{proof}
Suppose $U\xrightarrow{(i,j)} V \xrightarrow{(i',j')} W$ are successive edges in the
maximal directed path leading from $T\oplus a$ to a terminal shifted insertion state.
Then $(i,j)$ and $(i',j')$ are consecutive positions in the bumping path.

Suppose $(i,j)$ and $(i',j')$ are both row-bumped.
If $U \to V$ is a diagonal transition of type (D1), then $i+1 = j = i' = j' $.
Otherwise, 
$U \to V$ must be a row transition and $V \to W$ must be a 
row transition or a diagonal transition of type (D2), (D3), or (D4),
so $i<j$ and $i+1 = i' \leq j'$.
In this case the value in the outer box of $V$  is equal to $U_{ij}$, 
which is strictly less than $U_{i+1,j} = V_{i+1,j}$ if $(i+1,j) \in V$,
so $j'\leq j$.

Next, suppose  
$(i,j)$ is row-bumped and $(i',j')$ is column-bumped.
Then $U \to V$ is a diagonal transition 
of type (D2), (D3), or (D4) and $V \to W$ is a column transition,
so $i=j=t$ and $i' \leq j'=j+1=t+1$.

Finally, suppose $(i,j)$ and $(i',j')$ are both column-bumped. 
Then $U \to V$ and $V \to W$ are both column transitions,
so $j'=j+1$.
The value in the outer box of $V$ is then equal to $U_{ij}$, 
which is strictly less than $U_{i,j+1}=V_{i,j+1}$ if $(i,j+1) \in V$,
so $i' \leq i$. 
This completes the proof of part (a).

Suppose the insertion path \eqref{Toplusa} of $T\farrow a$ is 
$T\oplus a = U_0 \xrightarrow{(i_1,j_1)} U_1  \xrightarrow{(i_2,j_2)}  \cdots  \xrightarrow{(i_l,j_l)} U_l $.
To prove part (b), assume $\row(T)a$ is a symplectic Hecke word and $t<l$.
We will show that if $i_l = t+1$ then $\row(T) a$ is not an FPF-involution word.
Suppose
rows $t$ and $t+1$ of $U_{t}$ are 
\[
 \ytableausetup{boxsize = .7cm,aligntableaux=center}
\begin{ytableau}
\none[\cdot]  & w_1  & w_2 & \cdots  & w_m  & \none[\cdot]  & \none[\cdot]   \\
v_0 &  v_1 & v_2 & \cdots  & v_m&\cdots & v_n     
\end{ytableau}
\]
for some numbers $v_0<v_1<\dots<v_n$ and $w_1<w_2<\dots<w_m$
with $v_i < w_{i}$ for $i \in [m]$.
Write $w_0$ for the value in the outer box of $U_t$, which will be in column $t+1$
since $U_{t-1}\to U_t$ is a diagonal transition of type (D2), (D3), or (D4).
We must have $w_0 \leq v_1$,
and the only way we can have 
$t+1 = i_{t+1}=i_{t+2}=\dots =i_{l}$ is if
 $v_{i} = w_{i-1}$ for all $i \in [m+1]$.
But in this case the last edge $U_{l-1}\to U_l$ would be a column transition of type (C2),
so $T\farrow a$ would have the same number of boxes as $T$.
Since $\row(T)a \simFKK \row(T\farrow)$ by Theorem~\ref{farrow-thm},
it would follow that $\row(T)a$ is not an FPF-involution word.
This completes the proof of part (b).

We turn to part (c).
For each $r \in [l]$, let $a_r$ be the value in the outer box of $U_{r-1}$,
let $b_r$ be the value in box $(i_r,j_r)$ of $U_r$,
and let $c_r$ be the value in box $(i_r,j_r)$ of $U_{r-1}$.
For example, if $U_{r-1} \to U_r$ is a row transition and $r<l$, then
these numbers would correspond to the following picture:
\[
 \ytableausetup{boxsize = .8cm,aligntableaux=center}
U_{r-1}\ = \left\{\ \begin{ytableau}
\none[\cdot]  & \qm & \qm & \none[\cdot] & \none[\cdot]   \\
\qm & c_r & \qm & \none[\cdot] & a_r   
\end{ytableau}\ \right\}
\ \xrightarrow{(i_r,j_r)}\
\left\{\
\begin{ytableau}
\none[\cdot]  & \qm & \qm & \none[\cdot] & a_{r+1}   \\
\qm & b_r & \qm & \none[\cdot] & \none[\cdot]
\end{ytableau}\ \right\}= \ U_r.
\]
If $r \in [l-1]$ then either 
$a_r = b_r < c_r = a_{r+1}$ or 
$a_r < b_r = c_r = a_{r+1}$ or 
$a_r <  b_r = c_r = a_{r+1}-1$, 
with the last case occurring only if the 
forward transition
$U_{r-1} \to U_r$ is  of type (D4).
Therefore $a_r \leq b_r \leq c_r \leq a_{r+1}$ and at least one inequality is strict
for each $r \in [l-1]$.

We now argue by contradiction.
Let $r,s \in [l]$ be indices with $r\leq t < s$.
Suppose $s$ is minimal such that $i_s \leq i_r $ and $j_s \leq j_r$.
Since $i_t=j_t=t$ and $j_{t+1} = t+1$, we 
cannot have $r=t<t+1=s$, so
 $r+1 < s$
and $b_r < c_s$.
Part (a) and the minimality of $s$ imply that $c_s$ is the value in box $(i_s,j_s)$ of 
each of the states $U_r$, $U_{r+1}$, \dots, $U_{s-1}$.
This is impossible, however, since 
$U_r$ with its outer box removed is an increasing tableau.
We conclude from this contradiction that $i_r < i_s$ or $j_r <j_s$,
which is equivalent to part (c).

Since $i_r = r$ and $j_s = s$ for all $r,s \in [l]$ with $r\leq t < s$,
the only way that repeated positions can occur in the bumping path is if
some column-bumped position coincides with a row-bumped position.
This is impossible by part (c), so part (d) holds.
\end{proof}

\begin{proposition}\label{bump-lem2}
Suppose $T$ is an increasing shifted tableau and $a,b \in \PP$ are integers with $a\leq b$
such that $\row(T)ab$ is a symplectic Hecke word.
Let $U= T \farrow a$ and $V = U \farrow b$.
Refer to the bumping paths of $T\farrow a$ and $U\farrow b$ as the \emph{first}
and \emph{second bumping paths}, respectively.
Then:
\ben
\item[(a)] Suppose the $i$th element of the first path is row-bumped
and the second path has length at least $i$.
Then 
the $i$th elements of both paths are row-bumped and in row $i$, and
the $i$th element of the first path is weakly left of the $i$th element of the second path.


\item[(b)] If the last position $(r_1,c_1)$ in the first path is row-bumped, then
the last position $(r_2,c_2)$ in the second path is also row-bumped with $r_1 \geq r_2 $ and $c_1\leq c_2$.

\item[(c)] Suppose the $i$th element of the second path is column-bumped.
The first path then has length at least $i$,
the $i$th elements of both paths are column-bumped and in column $i$, and 
the $i$th element of the first path is weakly below the $i$th element of the second path.


\item[(d)] If the last position $(r_2,r_2)$ in the second path is column-bumped,
then the last position $(r_1,r_1)$ in the first path is also column-bumped with $r_1 \leq r_2$ and $c_1 \geq c_2$. 

\een
Moreover, if
$\row(T)ab$ is an FPF-involution word,
then parts (a) and (c) hold 
with ``weakly'' replaced by ``strictly''
and the inequalities in parts (b) and (d) are strict. 
\end{proposition}

\begin{proof}
Let 
$T\oplus a = U_0 \to  U_1 \to \cdots  \to  U_l = U$
and
$
U\oplus b = V_0 \to  V_1 \to \cdots  \to V_m = V$
be the insertion paths of $T\farrow a$ and $U \farrow b$,
so that the $i$th elements of the first and second bumping paths are the labels
of the edges $U_{i-1} \to U_i$ and $V_{i-1}\to V_i$.
By Proposition~\ref{adm-prop}, all of the states $U_i$ and $V_i$ are admissible.

Suppose $i \in [m]$ and
  the $i$th element of the first bumping path is row-bumped.
All preceding elements of the first bumping path are then also row-bumped.
Since $a\leq b$, it is straightforward to check that 
for each $j \in [i]$,
the value in the outer box of $U_{j-1}$ is at most the value in the outer box of $V_{j-1}$,
that the $j$th position in the second bumping path is row-bumped,
and that this position appears in row $j$ weakly to the right of the $j$th position in the first bumping path.
This proves part (a).

Suppose the last position in the first bumping path is row-bumped.
If $m \leq l$ then it follows from part (a) and Proposition~\ref{bump-lem1}
that the last position in the second bumping path is also row-bumped and occurs in a column weakly to the right of
the column containing the last position in the first bumping path.
If $l \leq m$ then $l=m$ since part (a) implies that the $l$th position in the second bumping path is in the same row as 
and
weakly to the right of the last position in the first bumping path, which is on the boundary of $U$.
In either case the rows of the last positions in the first and second paths are $l$ and $m$, respectively, and we have $m \leq l$.
This proves part (b).

Suppose next that the $i$th element of the second bumping path is column-bumped.
The last position in the second bumping path is then also column-bumped.
By part (b), the last position in the first bumping path must therefore be column-bumped as well.
Let $r \in [l]$ and $s \in [m]$ be the indices of the last row-bumped positions
in the first and second bumping paths. The $r$th position in the first path is then $(r,r)$
and the $s$th position in the second path is $(s,s)$.
Part (a) implies that $r<s$ and obviously $s<i$. The last position in the first bumping path is in column $l$ and
weakly below row $r$ by Proposition~\ref{bump-lem1}.
Since this position is on the boundary of $U$, we must have $s < l$.
From these considerations, it is straightforward to check
that the value in the outer box of $U_{j-1}$ is weakly less than the value in the outer box of $V_{j-1}$ for each $j \in [i] \cap [l]$
and that the $j$th position in the second bumping path is in the $j$the column and weakly above the 
$j$th position in the first bumping path for each $j \in [i] \cap [l] - [s]$. 
Since the last position of the first bumping is on the boundary of $U$, it follows that $i \leq l$,
so
this proves part (c).

Suppose finally that the last position in the second path is column-bumped.
It follows from part (c) 
that 
$m\leq l$
and that the $m$th position in the first bumping path is in a row weakly below the last position in the second bumping path.
By Proposition~\ref{bump-lem1}, the
 last 
position in the first bumping path is weakly below the $m$th position and also column-bumped.
Thus the columns of the last positions in the first and second paths are $l$ and $m$, respectively.
This proves part (d).

For the last assertion, 
note that if $\row(T)ab$ is an FPF-involution word, then $a<b$
and no forward transitions in the insertion paths of $T\farrow a$ or $U\farrow b$
are of type (R2), (D1), or (C2);
moreover, transitions of type (R3), (D2), and (C3)
only occur when a box adjacent to the bumped position is equal to the value in the outer box of the previous state. 
Given these observations, 
only minor changes to the preceding arguments are needed to deduce 
strict versions of parts (a), (b), (c), and (d). We omit these details.
\end{proof}

%
%
%


The \emph{descent set} of a word $w=w_1w_2\cdots w_n$ is $\Des(w) = \{ i \in [n-1]: w_i > w_{i+1}\}.$
The \emph{descent set} of a standard shifted set-valued tableau $T$ with length $|T|=n$ is 
\[
\Des(T) = \left\{
i \in [n-1] : 
\ba
&\text{$i$ appears in $T$ and $(i+1)'$ appears in $T$, or} 
\\
&\text{$i$ appears in $T$ and $i+1$ appears in $T$ in a row above $i$, or}
\\
&\text{$i'$ appears in $T$ and $(i+1)'$ appears in $T$ in a column right of $i'$}
\ea
\
\right\}.
\]
Observe that $i \in [n-1]$ is \emph{not} a descent 
of a standard set-valued tableau $T$
if and only if
$i'$ and $i+1$ both appear in $T$, 
or $i$ and $i+1$ both appear in $T$ with $i+1$ weakly southeast of $i$ (in French notation), or
$i'$ and $(i+1)'$ both appear in $T$ with $(i+1)'$ weakly northwest of $i'$.

\begin{example}\label{426-ex}
If  $w=42 6 1 75 3 4 2 1 3 2$ is the word from Example~\ref{zmax-ex} and $T=\QF(w)$
then \[\Des(w) = \Des(T) = \{ 1,3,5,6,8,9,11\}.\]
If $U$ is the ``doubled'' tableau
formed from $T$
by moving all primed entries in a given box $(x,y) \in T$ to the transposed position $(y,x)$,
then $i \in [n-1]$ is a descent if and only if the row of $U$ containing $i$ is strictly below the row of $U$ containing $i+1$. 
\end{example}

\begin{theorem}\label{fpf-des-thm}
If $w$ is a symplectic Hecke word then $\Des(w) = \Des(\QF(w))$.
\end{theorem}

\begin{proof}
Let $w$ be a symplectic Hecke word of length $n$.
Both descent sets are empty if $n\in \{0,1\}$ so assume $n\geq 2$.
Noting that the last position in 
any bumping path under symplectic Hecke insertion must be an inner or outer corner,
it is straightforward to deduce 
from parts (b) and (d) of
Proposition~\ref{bump-lem2} that if $i \in [n-1]$ is not a descent of $w$ then $i$ is not a descent of $\QF(w)$.
Therefore $\Des(\QF(w)) \subset \Des(w)$.
We will show that this containment is equality using a counting argument and induction.

Fix $m \in 2\PP$. Let $\cW_n$ be the set of symplectic Hecke words of length $n$
with all letters less than $m$.
Let $\cW_n^-$, $\cW_n^0$, and $\cW_n^+$ be the sets of words $w \in \cW_n$
with $w_{n-1} > w_n$, $w_{n-1} = w_n$, and $w_{n-1} < w_n$, respectively.
The maps
$
w_1w_2\cdots w_n \mapsto (m - w_1)(m-w_2)\cdots (m-w_n)
$ and $
w_1w_2\cdots w_n \mapsto w_1w_2\cdots w_{n-1}
$
 are bijections $\cW_n^- \to \cW_n^+$
and $\cW_n^0 \to \cW_{n-1}$,
so
$
 |\cW_n| = 2|\cW_n^-| + |\cW_{n-1}|.
$

Now let $\cX_n$ be the set of pairs $(P,Q)$ where $P$ is an increasing shifted tableau
whose row reading word is a symplectic Hecke word with all letters less than $m$
and $Q$ is a standard set-valued tableau of length $n$ with the same shape as $P$.
Let $\cX_n^-$ be the set of pairs $(P,Q) \in \cX_n$ with $n-1 \in \Des(Q)$.
Let $\cX_n^0$ be the set of pairs $(P,Q) \in \cX_n$ such that $Q$
contains either $n-1$ and $n$ in the same box
or $(n-1)'$ and $n'$ in the same box.
Finally define $\cX_n^+= \cX_n -\cX_n^- -\cX_n^0$.
Removing $n$ and $n'$ from $Q$ gives a bijection $\cX_n^0 \to \cX_{n-1}$,
and 
 altering $Q$ as follows gives a bijection $\cX_n^-\to\cX_n^+$:
 \begin{itemize}
\item If $n'$ is in the same row as $n-1$ or $(n-1)'$ but not the same box, remove the prime from $n$.
\item If $n-1$ is in the same column as $n$ or $n'$ but not the same box, add a prime to $n - 1$.
\item In all other cases when $n-1 \in \Des(Q)$, interchange $n-1$ with $n$ and $(n-1)'$ with $n'$.
\end{itemize}
We conclude that 
$
|\cX_n| = 2|\cX_n^-| + |\cX_{n-1}|.
$

Theorem~\ref{fpf-bij-thm}
implies that 
$w \mapsto (\PF(w),\QF(w))$ is a bijection $\cW_n \to \cX_n$, so
$|\cW_n| = |\cX_{n}|$ for all $n \in \NN$
and therefore $|\cW_n^-| = |\cX_n^-|$.
Since $\Des(\QF(w)) \subset \Des(w)$,
the map $w \mapsto (\PF(w),\QF(w))$ 
must restrict to a bijection $\cW_n^- \to \cX_n^-$,
so we have $n-1\in \Des(\QF(w))$ if and only if $n-1 \in \Des(w)$
for $w \in \cW_n$.
As
we may assume by induction that 
$\Des(\QF(w)) \cap [n-2] = \Des(w) \cap [n-2]$,
we conclude that $\Des(w) = \Des(\QF(w))$
for all symplectic Hecke words $w$.
\end{proof}

Theorem~\ref{fpf-des-thm} allows us to formulate a semistandard version
of symplectic Hecke insertion.
A \emph{weak set-valued tableau}
is a map from a finite subset of $\PP\times \PP$
to the set of finite, nonempty multi-subsets of 
the marked alphabet $\MM = \{1' < 1 < 2' < 2 < 3' <3 < \dots\}$.
All conventions for set-valued tableaux 
extend to weak set-valued tableaux without difficulty.

A weak set-valued tableau is \emph{shifted} if its domain is the shifted Young diagram of a strict partition.
A shifted weak set-valued tableau $T$ is 
\emph{semistandard} if the following conditions hold:
\begin{itemize}
\item If $(a,b),(x,y) \in T$ have $(a,b)\neq(x,y)$ and  $a \leq x$ and $b\leq y$, then $\max(T_{ab}) \leq \min(T_{xy})$.
\item No primed number belongs to any box of $T$ on the main diagonal.
\item Each unprimed number appears in at most one box in each column of $T$.
\item Each primed number appears in at most one box in each row of $T$.
\end{itemize}
A \emph{semistandard shifted set-valued tableau} is a semistandard shifted weak set-valued tableau
whose entries are sets. 
A \emph{semistandard shifted marked tableau} is a semistandard shifted set-valued tableau
whose entries are all singleton sets.
For example, the shifted weak set-valued tableaux
\be\label{uv-eq}
\ytableausetup{boxsize = 0.8cm,aligntableaux=center}
U =\
\begin{ytableau}
\none & 5 & 6' 6\\
1\ 2 & 2\ 2 & 3'  6'
\end{ytableau}
\qquand
V =\
\begin{ytableau}
\none & 5 & 6'\\
2 & 2 &   6'
\end{ytableau}
\ee
are both semistandard, and $V$ is a shifted marked tableau.
The \emph{weight} of a weak set-valued tableau $T$
is the map $\weight(T) : \PP \to \NN$ whose value at $i \in \PP$ is the number of times $i$ or $i'$ appears in $T$.
It is convenient to represent as $\weight(T)$ as a weak composition; 
for example, 
if $U$ and $V$ are as in \eqref{uv-eq} 
then
$\weight(U) = (1,3,1,0,1,3)$ and $\weight(V) = (0,2,0,0,1,2)$.

Let $w=w_1w_2\cdots w_m$ be a word of length $m$.
Define a \emph{weakly increasing factorization} of $w$
to be a weakly increasing sequence of positive integers $i=(i_1\leq i_2 \leq \dots \leq i_m)$
with $i_j <  i_{j+1}$ if $j \in \Des(w)$.
The \emph{weight} of such a factorization is the map $\mu : \PP \to \NN$
with $\weight(a) = |\{j\in [m] : i_j = a\}|$ for $a \in \PP$.
The data of a weakly increasing factorization of $w$ is 
equivalent to 
a decomposition of $w$ into a countable sequence of weakly increasing subwords
$w=w^1w^2w^3\cdots $. 

When $w$ is a symplectic Hecke word of length $m$ and $i=(i_1\leq i_2 \leq \dots \leq i_m)$ is a
weakly increasing factorization of $w$,
we define $\QF(w,i)$ to be the shifted weak set-valued tableau
formed from $\QF(w)$ by replacing $j$ by $i_j$ and $j'$ by $i_j'$ 
for each $j\in [m]$. E.g.,
if $w=42 6 1 75 3 4 2 1 3 2$ as in Example~\ref{zmax-ex} 
and $i =122334556889$, so that
$(w,i) \leftrightarrow (4)(26)(17)(5)(34)(2)()(13)(2)$, then
\[ 
\ytableausetup{boxsize = .6cm,aligntableaux=center}
\QF(w,i)=
\begin{ytableau}
\none & \none & 5 & 8' \\
\none & 4 &5' &6' & 9' \\
1 & 2' & 2 &3' &3 & 8' 
\end{ytableau}.
\]
We now have the following refinement of Theorem~\ref{fpf-bij-thm}.

\begin{theorem}\label{ssyt-thm}
Let $z \in \FF$.
The correspondence $(w,i) \mapsto (\PF(w),\QF(w,i))$
is a bijection
from weakly increasing factorizations of symplectic Hecke words
(respectively, FPF-involution words)
for $z$ to pairs
$(P,Q)$ where $P$ is an increasing shifted tableau with $\row(P) \in \cHfpf(z)$
(respectively, $\row(P) \in \cRfpf(z)$)
and $Q$ is a semistandard shifted weak set-valued (respectively, marked) tableau
with the same shape as $P$.
Moreover, $(w,i)\mapsto \QF(w,i)$ is a weight-preserving map.
\end{theorem}

\begin{proof}
Suppose $i=(i_1\leq i_2\leq \dots \leq i_m)$ is a weakly increasing factorization of
a symplectic Hecke word
$w=w_1w_2\cdots w_m \in \cHfpf(z)$.
The shifted weak set-valued tableau $\QF(w,i)$ 
has the same weight as $i$ by construction.
To check that $\QF(w,i)$ is semistandard,
fix $h \in \{i_1,i_2,\dots, i_m\}$
and suppose $j \in \NN$ and $b \in\PP$
are such that $h=i_t$ if and only if $t \in \{j+1,j+2,\dots,j+b\}$,
so that $\Des(\QF(w)) \cap \{j+1,j+2,\dots,j+b-1\} = \varnothing$.
Theorem~\ref{fpf-des-thm}
implies that there exists an integer $0\leq a \leq b$ such that the primed numbers
$(j+1)',(j+2)',\dots,(j+a)'$ all appear in $\QF(w)$
and the unprimed numbers $j+a+1,j+a+2,\dots,j+b$ all appear in $\QF(w)$;
moreover, none of the primed numbers 
can appear in different boxes in the same row of $\QF(w)$
and none of the unprimed numbers can appear in different boxes in the same column.
We conclude that $\QF(w,i)$ is weakly increasing in the required sense. Since this weak set-valued tableau
obviously contains no primed numbers on the main diagonal, $\QF(w,i)$ is semistandard.

Suppose $Q$ is a semistandard shifted weak set-valued tableau.
Following \cite[\S3.2]{HKPWZZ}, define the \emph{standardization} of $Q$ 
to be the standard shifted set-valued tableau $\st(Q)$ 
formed from $Q$ by the following procedure.
Start by replacing all 1s appearing in $Q$, read from left to right,
by  $1, 2, \dots, i$. (Note that no $1'$s appear in $Q$.)
Then replace all $2'$s appearing in $Q$, read bottom to top,
by the primed numbers $(i+1)'$, $(i+2)'$, \dots $(i+j)'$,
Then replace all $2$s appearing $Q$, read left to right, 
by $i+j+1$, $i+j+2$, \dots, $i+j+k$.
Then replace all $3'$s appearing in $Q$, read bottom to top, by
the primed numbers $(i+j+k+1)'$, $(i+j+k+2)'$, \dots, $(i+j+k+l)'$,
and so on,
continuing this substitution process for the numbers $3,4',4,\dots,n',n$.
If $\st(Q)$ has length $m$, 
then define $i^Q =  (i^Q_1 \leq i^Q_2 \leq \dots \leq i^Q_m)$
to be the weakly increasing sequence of positive integers with $i^Q_j = a$
if $a$ or $a'$ appears in $Q$ and changes to $j$ or $j'$ in $\st(Q)$.

Now suppose $(w,i)$ is a weakly increasing factorization of a symplectic Hecke word.
Using Theorem~\ref{fpf-des-thm}, it is easy to see that every semistandard shifted weak set-valued tableau whose standardization is $\QF(w)$
arises as $\QF(w,i)$ for some choice of factorization $i$.
It follows from Theorem~\ref{fpf-bij-thm} that the map described in the theorem is surjective.
Similarly, it is straightforward to deduce that
we recover $(w,i)$ from $(P,Q):=(\PF(w), \QF(w,i))$
as $w=\wSp(P,\st(Q))$ and $i = i^Q$.
We conclude that the given map is also injective.
The ``marked'' version of the theorem for FPF-involution words follows by the same argument.
\end{proof}

\subsection{Orthogonal Hecke insertion}\label{ortho-sect}

Given a word $w=w_1w_2\cdots w_m$, define $2[w]$ to be the word $(2w_1)(2w_2)\cdots (2w_m)$.
When $T$ is an increasing tableau, write $2[T]$ for the tableau formed by doubling every entry of $T$.
When $T$ has all even entries, define $\frac{1}{2}[T]$ by halving every entry
analogously. For example, 
\[
2[3412] = 6824 
\qquand
 \ytableausetup{boxsize = .4cm,aligntableaux=center}
 \tfrac{1}{2} \left[\
 \begin{ytableau}
\none & 6 & \none \\
2 & 4 & 8 
 \end{ytableau}
\ \right]
 =\
 \begin{ytableau}
\none & 3 & \none \\
1 & 2 & 4 
 \end{ytableau}.
 \]
 If $T$ has 
 all even entries and $a \in 2\PP$ then $T\farrow a$ has all even entries,
so the following is well-defined:

\begin{definition}
\label{iarrow-def}
Given an increasing shifted tableau $T$ and $a \in \PP$,
let
$T\iarrow a: = \tfrac{1}{2} \left[2[T] \farrow 2a\right].$
\end{definition}

We refer to the operation transforming $(T,a) $ to $ T\iarrow a $ as \emph{orthogonal Hecke insertion}.
We could also define $T\iarrow a$ 
exactly as we defined $T\farrow a$, without any doubling of letters,
by slightly modifying the forward transition graph from Section~\ref{forward-sect}.
All that is needed is to remove  the parity condition from
 transition (D3)
and omit transition (D4).

Any word with all even letters is a symplectic Hecke word,
so the following is also well-defined.

\begin{definition}
For any word $w$,
define $ \PI(w) =\tfrac{1}{2}\left[ \PF\(2[w]\)\right]$
and $ \QI(w) =\QF\(2[w]\)$.
\end{definition}

We call $\PI(w)$ the \emph{insertion tableau} 
and $\QI(w)$ the \emph{recording tableau}
of $w$ under orthogonal Hecke insertion.
If $w=w_1w_2\cdots w_n$
then $\PI(w) = (\cdots ((\emptyset \iarrow w_1) \iarrow w_2)\cdots) \iarrow w_n$.

\begin{example}\label{pqi-ex}
If $w=451132$ then 
 \[
\ytableausetup{boxsize = .7cm,aligntableaux=center}
 \PI(w) =\ 
\begin{ytableau}
\none &  3  \\ 
1 &  2 & 4 & 5 
 \end{ytableau}
 \qquand
  \QI(w) = 
\begin{ytableau}
\none &  5  \\ 
1 &  2 & 3'4' & 6' 
 \end{ytableau}.
 \]
\end{example}

\begin{proposition}\label{same-prop}
The correspondence $w \mapsto (\PI(w),\QI(w))$
is the \emph{shifted Hecke insertion} algorithm introduced by Patrias and Pylyavskyy in \cite[\S5.3]{PatPyl2014}.
\end{proposition}

\begin{proof}
This is clear from comparing the rules (S1)-(S4) defining shifted Hecke insertion in \cite[\S5.3]{PatPyl2014}
with the forward transitions (R1)-(R4), (D1)-(D4), and (C1)-(C4) described in Section~\ref{forward-sect}.
\end{proof}

The insertion and recording tableaux $\PI(w)$, $\QI(w)$
are denoted $P_S(w)$, $Q_S(w)$ in \cite[\S5.3]{PatPyl2014}, $P_{SK}(w)$, $Q_{SK}(w)$ in \cite[\S2]{HKPWZZ},
and $P_{SH}(w)$, $Q_{SH}(w)$ in \cite[\S5]{HMP4}.

Proposition~\ref{same-prop}
lets us recover several facts about shifted Hecke insertion from what we have
already shown
about symplectic Hecke insertion.
Write $\simICK$ (respectively, $\simIKK$) for the strongest equivalence relation that has $v \simICK w$
(respectively, $v\simIKK w$) whenever $v$ and $w$ are words 
such that $w$ is obtained from $v$ by swapping its first two letters,
 or that satisfy $v \simCK w$ (respectively, $v \simKK w$).
 The relation $\simKK$ is called 
 \emph{weak $K$-Knuth equivalence} in \cite{BuchSamuel,HMP4,HKPWZZ}.

If $v$ and $w$ are words,
then $v\simIKK w$ 
if and only if 
$2[v] \simFKK 2[w]$,
and $v\simICK w$ 
if and only if 
$2[v] \simFCK 2[w]$.
The next three results are immediate from Theorem~\ref{farrow-thm}
and Corollaries~\ref{fkk-cor} and \ref{reduced-cor}.

\begin{corollary}
Let $T$ be an increasing shifted tableau and $a \in \PP$. 
\ben
\item[(a)] The tableau $T\iarrow a$ is increasing and shifted and $\row(T\iarrow a) \simIKK \row(T)a$.
\item[(b)] If $\row(T)a$ is an involution word then $\row(T\iarrow) \simICK \row(T)a$.
\een
\end{corollary}

\begin{corollary}[{\cite[Corollary~2.18]{HKPWZZ}}]
\label{ikk-cor}
If $w$ is any word then $w \simIKK \row(\PI(w))$.
\end{corollary}


\begin{corollary}\label{ick-cor}
If $w$ is an involution word then $w \simICK \row(\PI(w))$.
\end{corollary}

Thus, if $v$ and $w$ are any words (respectively, involution words)
with $\PI(v) = \PI(w)$,
then $v \simIKK w$ (respectively, $v\simICK w$).
The converse of this property does not hold in general (see \cite[Remark 2.19]{HKPWZZ}),
but computations support the following, which also appears as  \cite[Conjecture 5.24]{HMP4}:

\begin{conjecture}
If $v$ and $w$ are involution words with $v \simICK w$ then $\PI(v) = \PI(w)$.
\end{conjecture}

\begin{remark*}[Note added in proof]
A proof of this result now appears in \cite{Marberg2019}.
\end{remark*}

Since $\Des(w) = \Des(2[w])$ for any word $w$,
the following is clear from Theorem~\ref{fpf-des-thm}.

\begin{corollary}[{\cite[Proposition~2.24]{HKPWZZ}}]
\label{inv-des-cor}
If $w$ is any word then $\Des(w) = \Des(\QI(w))$.
\end{corollary}

It follows from Theorem~\ref{fpf-hecke-char-thm} and Corollary~\ref{fkk-cor}
that if $w$ is a symplectic Hecke word,
then the insertion tableau $\PF(w)$ has all even entries 
if and only if $w$ has all even letters.
The next result therefore follows from Theorem~\ref{fpf-bij-thm} and Corollary~\ref{ikk-cor}.

\begin{corollary}[{\cite[Theorem~5.18]{PatPyl2014}}]
\label{inv-bij-cor}
Let $z \in \II$.
Then $w \mapsto (\PI(w), \QI(w))$ 
is a bijection from 
the set 
of orthogonal Hecke words (respectively, involution words) for $z$ of length $n\in\NN$
to the set of
pairs $(P,Q)$
 in which $P$ is an increasing shifted tableau with $\row(P) \in \iH(z)$ 
 (respectively, $\row(P) \in \iR(z)$)
 and $Q$ is a standard shifted set-valued (respectively, marked) tableau of length $n$ with the same shape as $P$.
\end{corollary}

Finally, there is a semistandard version of the preceding corollary.
If $i=(i_1\leq i_2\leq \dots \leq i_m)$ is a weakly increasing factorization of a word $w=w_1w_2\cdots w_m$,
then $i$ is also a weakly increasing factorization of $2[w]$
and we define
$\QI(w,i) = \QF(2[w],i)$.
For example, if $w=451132$ as in Example~\ref{pqi-ex} and $i=113335$, 
so that $(w,i)$ corresponds to $(45)()(113)()(2)$, then  
 \[
\ytableausetup{boxsize = .7cm,aligntableaux=center}
  \QI(w,i) = 
\begin{ytableau}
\none &  3  \\ 
1 &  1 & 3'3' & 5' 
 \end{ytableau}.
 \]
Given Corollaries~\ref{inv-des-cor} and \ref{inv-bij-cor},
the following result has the same proof as Theorem~\ref{ssyt-thm}.
\begin{corollary}\label{ssyt-cor}
Let $z \in \II$.
The correspondence $(w,i) \mapsto (\PI(w),\QI(w,i))$
is a bijection
from weakly increasing factorizations of orthogonal Hecke words
(respectively, involution words)
for $z$ to pairs
$(P,Q)$ where $P$ is an increasing shifted tableau with $\row(P) \in \iH(z)$
(respectively, $\row(P) \in \iR(z)$)
and $Q$ is a semistandard shifted weak set-valued (respectively, marked) tableau
with the same shape as $P$.
Moreover, $(w,i)\mapsto \QI(w,i)$ is a weight-preserving map.
\end{corollary}

\subsection{Involution Coxeter-Knuth insertion}\label{ick-sect}

Restricted to (FPF-)involution words,
symplectic and orthogonal Hecke insertion
reduce to less complicated algorithms, which refer to as \emph{(FPF-)involution Coxeter-Knuth insertion}.
Propositions~\ref{fpfins-prop} and \ref{invins-prop}
describe these bumping procedures, which 
are shifted analogues of \emph{Edelman-Greene insertion} \cite{EG}
and ``reduced word'' generalizations of \emph{Sagan-Worley insertion} \cite{Sagan1987,Worley}.

\begin{proposition}[FPF-involution Coxeter-Knuth insertion]
\label{fpfins-prop}
Let $a$ be a positive integer.
Suppose $L$ is a row or column of an increasing tableau.
One inserts $a$ into $L$ as follows:
\ben
\item[] Find the first entry $b$ of $L$ with $a \leq b$. 
If no such entry exists then add $a$ to the end of $L$ and say that no entry is bumped,
but refer to the added box as the bumped position.
Otherwise:
\begin{itemize}
\item If $a=b$ then leave $L$ unchanged but say that $a+1$ is bumped from the position \textbf{directly following} the position of $b$.
\item If $L$ is a row (rather than a column) and $b$ is the first entry of $L$ and $a\not \equiv b\modu 2)$, then leave $L$ unchanged but say that $a+2$ is bumped from the position of $b$.
\item In all other cases replace $b$ by $a$ in $L$ and say that $b$ is bumped.
\end{itemize}
\een
Now suppose $T$ is an increasing shifted tableau such that $\row(T) a$ is an FPF-involution word.
%
\ben
\item Start by inserting $a$ into the first row of $T$ according to the rules above. 

\item  If no entry is bumped then the process terminates.
 Otherwise, an entry is bumped from some position of $T$.
 If this position is on the diagonal or if a position bumped in an earlier step was on the diagonal, 
then we continue by inserting the bumped entry into the next column;
otherwise, we continue by inserting the bumped entry into the next row.

 \item Repeat step 2 until we insert into a row or column and no entry is bumped.
 \een
The resulting tableau is $T\farrow a$
and the sequence of bumped positions is the corresponding 
bumping path from Definition~\ref{farrow-def}.
\end{proposition}

Using Theorem~\ref{fpf-hecke-char-thm}
and Lemmas~\ref{ck-lem1}, \ref{ck-lem2}, and \ref{ck-lem3},
it is straightforward but fairly tedious to deduce that the output of this algorithm coincides with Definition~\ref{farrow-def}
when $\row(T)a$ is an FPF-involution word. We leave these details to the reader.

\begin{example} We compute $\PF(w)$ and $\QF(w)$ for the FPF-involution word $w=42312$:
\[ 
\ytableausetup{boxsize = .45cm,aligntableaux=center}
\ba
\begin{ytableau}
  4 
  \end{ytableau}
  &&\longrightarrow\quad
  \begin{ytableau}
  2 & 4
  \end{ytableau}
  &&  \longrightarrow\quad
  \begin{ytableau}
  \none & 4 \\
  2 & 3
  \end{ytableau}
     && \longrightarrow\quad
  \begin{ytableau}
  \none & 4 \\
  2 & 3 & 4
  \end{ytableau}
     && \longrightarrow\quad
  \begin{ytableau}
  \none & 4 & 5 \\
  2 & 3 & 4
  \end{ytableau} &\ \ =\ \PF(42312)
  \\[-10pt]\\
\begin{ytableau}
  1 
  \end{ytableau}
  &&\longrightarrow\quad
  \begin{ytableau}
  1 & 2'
  \end{ytableau}
  &&  \longrightarrow\quad
  \begin{ytableau}
  \none & 3 \\
  1 & 2'
  \end{ytableau}
     && \longrightarrow\quad
  \begin{ytableau}
  \none & 3\\
  1 & 2' & 4'
  \end{ytableau}
     && \longrightarrow\quad
  \begin{ytableau}
  \none & 3 & 5' \\
  1 & 2' & 4'
  \end{ytableau} &\ \ =\ \QF(42312).
\ea
\]
The bumping path of $\PF(4231)\farrow 2$  is $(1,2)$, $(2,2)$, $(2,3)$.
\end{example}

\begin{proposition}[Involution Coxeter-Knuth insertion]\label{invins-prop}
Let $a$ be a positive integer.
Suppose $L$ is a row or column of an increasing tableau.
One inserts $a$ into $L$ as follows:
\ben
\item[] Find the first entry $b$ of $L$ with $a \leq b$. 
 If no such entry exists then add $a$ to the end of $L$ and say that no entry is bumped.
 Otherwise:
 \begin{itemize}
\item If $a=b$ then leave $L$ unchanged but say that $a+1$ is bumped from the position of $b$.
\item If $a\neq b$ then replace $b$ by $a$ in $L$ and say that $b$ is bumped.
\end{itemize}
\een
Now suppose $T$ is an increasing shifted tableau such that $\row(T)a$ is an involution word.
\ben
\item Start by inserting $a$ into the first row of $T$ according to the rules above. 

\item  If no entry is bumped then the process terminates.
 Otherwise, an entry is bumped from some position of $T$.
 If this position is on the diagonal or if a position bumped in an earlier step was on the diagonal, 
then we continue by inserting the bumped entry into the next column;
otherwise, we continue by inserting the bumped entry into the next row.

 \item Repeat step 2 until we insert into a row or column and no entry is bumped.
 \een
The resulting tableau is $T\iarrow a$
from Definition~\ref{iarrow-def}.
\end{proposition}

We omit the proof of the proposition, which is straightforward from
the results in Section~\ref{ortho-sect}.

\begin{remark}
It would be natural to define the \emph{bumping path} of $T\iarrow a$ to be the bumping path of
$2[T]\farrow 2a$. However, this sequence does not coincide
with the sequence of bumped positions in Proposition~\ref{invins-prop}, although the two paths are closely related.
\end{remark}

\begin{example} We compute $\PI(w)$ and $\QI(w)$ for the involution word $42321$:
\[ 
\ytableausetup{boxsize = .45cm,aligntableaux=center}
\ba
\begin{ytableau}
  4 
  \end{ytableau}
  &&\longrightarrow\quad
  \begin{ytableau}
  2 & 4
  \end{ytableau}
  &&  \longrightarrow\quad
  \begin{ytableau}
  \none & 4 \\
  2 & 3
  \end{ytableau}
     && \longrightarrow\quad
  \begin{ytableau}
  \none & 4 \\
   2 & 3 &4
  \end{ytableau}
     && \longrightarrow\quad
  \begin{ytableau}
  \none &  4 \\
  1 & 2 & 3 & 4
  \end{ytableau} &\ \ =\ \PI(42321)
  \\[-10pt]\\
\begin{ytableau}
  1 
  \end{ytableau}
  &&\longrightarrow\quad
  \begin{ytableau}
  1 & 2'
  \end{ytableau}
  &&  \longrightarrow\quad
  \begin{ytableau}
  \none & 3 \\
  1 & 2'
  \end{ytableau}
     && \longrightarrow\quad
  \begin{ytableau}
  \none & 3\\
  1 & 2' & 4'
  \end{ytableau}
     && \longrightarrow\quad
  \begin{ytableau}
  \none & 3 \\
  1 & 2' & 4'& 5' 
  \end{ytableau} &\ \ =\ \QI(42321).
\ea
\]
\end{example}

\section{Stable Grothendieck polynomials}\label{stab-sect}

Recall the definition of the \emph{stable Grothendieck polynomial} $G_\pi$ for $\pi \in \SS$
from \eqref{gpi-eq}.
Let $\whSym$ be the free $\ZZ$-module of arbitrary (formal)
linear combinations of  the Schur functions $s_\lambda$. 
This module is a subring of $\ZZ[\beta][[x_1,x_2,\dots]]$,
and one has $G_\pi \in \whSym$ for all $\pi \in \SS$ \cite[\S2]{BKSTY}.

\begin{definition}
Given a partition $\lambda$ with $k$ parts,
let $G_\lambda := G_{\pi_\lambda}$
where $\pi_\lambda \in \SS$ is the  permutation with
$\pi_\lambda(i) = i+\lambda_{k+1-i}$ for $i \in [k]$
and $\pi_\lambda(i) < \pi_\lambda(i+1)$ for $i \neq  k$.
\end{definition}

For a (weak) set-valued tableau $T$, define $x^T = \prod_{(i,j) \in T} \prod_{e \in T_{ij}} x_{|e|}$
where $|e|=|e'| =e$ for $e \in \PP$.
If $U$ and $V$ are as in \eqref{uv-eq} then $x^U = x_1 x_2^3 x_3 x_5 x_6^3$
and $x^V = x_2^2 x_5x_6^2$.
Given a partition $\lambda$,
define  $\SetSSYT(\lambda)$ 
to be the family of set-valued tableaux $T$ with domain $\Y_\lambda :=\{ (i,j) \in \PP\times \PP :  j \leq \lambda_i\}$,
whose entries are subsets of $\PP$,
 that 
are \emph{semistandard} in the sense
that if $(a,b),(x,y) \in T$ are distinct with $a \leq x$ and $b\leq y$, then $\max(T_{ab}) \leq \min(T_{xy})$
with  equality only if $a=x$.

\begin{theorem}[{Buch \cite[Theorem 3.1]{Buch2002}}]
\label{lam-thm}
If $\lambda$ is a partition then
$
G_\lambda = \sum_{T \in \SetSSYT(\lambda)} \beta^{|T| - |\lambda|} x^T.
$
\end{theorem}

For $\pi \in \SS$, define $\pi^* \in \SS$ by conjugating $\pi$ by 
$n\cdots 321$ where $n \in \NN$ is minimal with $\pi(i) = i$ for all $i>n$;
the map $w_1w_2\cdots w_l \mapsto (n-w_1)(n-w_2)\cdots (n-w_l)$ is then a bijection $\cH(\pi) \to \cH(\pi^*)$.
Define $\omega : \whSym \to \whSym$ to be the $\ZZ[\beta]$-linear involution with $\omega\(\sum_{\lambda} c_\lambda s_\lambda\) = \sum_\lambda c_\lambda s_{\lambda^T}$ for all coefficients $c_\lambda \in \ZZ[\beta]$, where $\lambda^T$ is the usual partition transpose.

\begin{lemma}\label{gom-lem}
If $\pi \in \SS$ then
$\omega(G_\pi) = G_{\pi^{-1}}\(\tfrac{x_1}{1-\beta x_1}, \tfrac{x_2}{1-\beta x_2},\dots\)
=
G_{\pi^{*}}\(\tfrac{x_1}{1-\beta x_1}, \tfrac{x_2}{1-\beta x_2},\dots\)
.$
\end{lemma}

\begin{proof}
For $n \in \PP$ and $S \subset [n-1]$,
the associated \emph{fundamental quasi-symmetric function} is the power series 
$
L_{S,n} := \sum x_{i_1}x_{i_2}\cdots x_{i_n}
$
where the sum is over all weakly increasing sequences of positive integers $i_1\leq i_2 \leq \dots \leq i_n$
with $i_j < i_{j+1}$ whenever $j \in S$.
Let $\QSym$ denote the $\ZZ$-module generated by these functions.
It is well-known that $\omega$ extends the linear map $\QSym \to \QSym$ with $L_{S,n} \mapsto L_{[n-1] \setminus S, n}$ \cite[\S3.6]{LMW}.
Since $G_\pi$ is a symmetric linear combination of fundamental quasi-symmetric functions,  
$\omega(G_\pi) =  \sum_{(w,i)} \beta^{\ell(w)-\ell(\pi)}  x^i$
where the sum is over pairs of words in which $w=w_1w_2\cdots w_l \in \cH(\pi^{-1})$ and $i=(i_1\leq i_2\leq \dots \leq i_l)$
is such that $i_j < i_{j+1}$ whenever $w_j < w_{j+1}$.
This is equal to $G_{\pi^{-1}}\(\tfrac{x_1}{1-\beta x_1}, \tfrac{x_2}{1-\beta x_2},\dots\)$
and, by similar reasoning, also to
$G_{\pi^{*}}\(\tfrac{x_1}{1-\beta x_1}, \tfrac{x_2}{1-\beta x_2},\dots\)$
\end{proof}

\begin{lemma}\label{glam-eq} If $\lambda$ is a partition then
$
\omega(G_\lambda) = G_{\lambda^T}\(\tfrac{x_1}{1-\beta x_1},\tfrac{x_2}{1-\beta x_2},\dots\).
$
\end{lemma}

\begin{proof}
This  
is equivalent to \cite[Proposition 9.22]{LamPyl}
after observing that the functions $\tilde K_\lambda$ and $J_\lambda$ in \cite{LamPyl} satisfy
$\tilde K_\lambda( \beta x_1, \beta x_2,\dots) = \beta^{|\lambda|} G_\lambda$
and $J_\lambda(\beta x_1,\beta x_2,\dots)  = \beta^{|\lambda|} G_{\lambda^T} \( \frac{ x_1}{1-\beta x_1}, \frac{ x_2}{1-\beta x_2},\dots\)$.
\end{proof}

Recall that  $\ellhat(y)$ and $\ellfpf(z)$ denote the common lengths of all words in $\iR(y)$ and $\cRfpf(z)$.
Formulas for these numbers appear in \cite[\S2.3]{HMP4} and \cite[\S2.3]{HMP5}.
For each $z \in \FF$,
there is a minimal $n \in 2\NN$
such that $z(i) = i - (-1)^i$ for all $i>n$;
define $\overline z \in \II$ to be the involution with  $i \mapsto z(i)$ for $i \in [n]$
that fixes all $i>n$.
If $y=\overline z$ has $\kappa$ cycles of length two, then
\be\label{ellhat-eq}
\ellhat(y) = \tfrac{\ell(y) + \kappa}{2}
\qquand
\ellfpf(z) = \tfrac{\ell(y) -\kappa}{2}.
\ee
We turn to
the \emph{shifted stable Grothendieck polynomials}
$\iG_y$ and $\Gfpf_z$
defined 
by \eqref{igyz-eq}.
These functions can be expressed 
in terms of the sets $\cB(y)$ and $\cBfpf(z)$ 
from 
\eqref{atoms2-eq} and \eqref{atoms1-eq}:

\begin{proposition}
If $y \in \II$ and $z \in \FF$
then
\[\iG_y = \sum_{\pi \in \cB(y)} \beta^{\ell(\pi)- \ellhat(y)} G_{\pi}
\qquand
\Gfpf_z = \sum_{\pi \in \cBfpf(z)} \beta^{\ell(\pi)- \ellfpf(z)}  G_{\pi}.
\]
\end{proposition}

This formulation shows that these power series belong to $\whSym$.
Let $\lambda$ be a strict partition of $n\in\NN$.
Write $\SetSMT(\lambda)$   
for the
set of semistandard shifted set-valued tableaux of shape $\lambda$, i.e.,
with domain $\SY_\lambda:=\{ (i,i+j-1) \in \PP\times \PP : 1\leq j \leq \lambda_i\}$.
Ikeda and Naruse introduce the following ``$K$-theoretic Schur $P$-functions''
in \cite{IkedaNaruse}:

\begin{definition}[{See \cite[Theorem 9.1]{IkedaNaruse}}]\label{GP-def}
The \emph{$K$-theoretic Schur $P$-function} indexed by a strict partition $\lambda$ is
$
\GP_\lambda =\sum_{T \in \SetSSMT(\lambda)} \beta^{|T| - |\lambda|} x^T \in \whSym.
$
\end{definition}

\begin{lemma}\label{om-gp-lem}
 If $\lambda$ is a strict partition then 
$\omega(\GP_\lambda) = \GP_\lambda\(\tfrac{x_1}{1-\beta x_1},\tfrac{x_2}{1-\beta x_2},\dots\).$
\end{lemma}

\begin{proof} It is an easy exercise from \cite[Definition 10.1 and Corollary 10.10]{HIMN} 
to show that the lemma holds whenever $\lambda = (n)$ is a partition with a single part.
The general identity follows since $\omega$ is a (continuous) ring homomorphism and \cite[Theorem 5.4]{NN2018}
implies $\GP_\lambda \in  \ZZ[\beta][[\GP_{(n)} : n \in \PP]]$.
\end{proof}

We can now prove Theorem~\ref{lastmain-thm} from the introduction.


\begin{proof}[Proof of Theorem~\ref{lastmain-thm}]
For a strict partition $\lambda$, 
let $\WkSetSSMT(\lambda)$ 
be the
set of semistandard weak set-valued tableaux 
of shape $\lambda$,
as in Section~\ref{des-sect}.
We have
$\beta^{|\lambda|} \GP_\lambda\( \tfrac{x_1}{1-\beta x_1}, \tfrac{x_2}{1-\beta x_2},\dots\)
=
\sum_{T \in \WkSetSSMT(\lambda)} \beta^{|T|} x^T
$ by  \cite[Proposition 3.5]{HKPWZZ} (after making the substitution $x_i \mapsto -\beta x_i$).
On the other hand,
$\beta^{\ellfpf(z)} \omega(\Gfpf_{z})=\sum_{(w,i)} \beta^{\ell(w)}x^i$
where the sum is over pairs $(w,i)$ such that $w=w_1w_2\cdots w_l \in \cHfpf(z)$
and $i=(i_1\leq i_2\leq \dots \leq i_l)$ is a weakly increasing sequence of positive integers
with $i_j < i_{j+1}$ whenever $w_j > w_{j+1}$.
By Theorem~\ref{ssyt-thm},
this sum is 
exactly
$\sum_\lambda c_{z\lambda}  \sum_{T \in \WkSetSSMT(\lambda)}  \beta^{|T|} x^T
$.
Combining these observations with Lemma~\ref{om-gp-lem}, we have
\[\omega(\Gfpf_{z})= \sum_\lambda \beta^{|\lambda|-\ellfpf(z)} c_{z\lambda} \GP_\lambda\( \tfrac{x_1}{1-\beta x_1}, \tfrac{x_2}{1-\beta x_2},\dots\)
=\sum_\lambda c_{z\lambda} \beta^{|\lambda|-\ellfpf(z)}  \omega(\GP_\lambda).\]
Now simply reapply $\omega$.
The formula for $\iG_y$ follows in the same way via Corollary~\ref{ssyt-cor}.
\end{proof}

The operation $\SS\to \SS$ given by $\pi \mapsto \pi^*$ preserves $\II$.
For $z \in \FF$, define $z^* \in \FF$ by conjugating $z$ by $n\cdots 321$
where $n \in 2\NN$ is minimal such that $z(i) = i-(-1)^i$ for all integers $i>n$;
the map $w_1w_2\cdots w_l \mapsto (n-w_1)(n-w_2)\cdots (n-w_l)$ is then a bijection $\cHfpf(z) \to \cHfpf(z^*)$.

\begin{corollary}\label{g*-cor}
If $y \in \II$ and $z \in \FF$
and $\iG_y = \iG_{y^*}$ and $\Gfpf_z = \Gfpf_{z^*}$.
\end{corollary}

\begin{proof}
Since $\iG_y$ and $\Gfpf_z$ are linear combinations of $\GP_\lambda$'s and since
$\cB(y^*) =\{ \pi^* : \pi \in \cB(y)\}$ and $\cBfpf(z^*) = \{ \pi^* : \pi \in \cBfpf(z)\}$,
this follows by combining Lemmas~\ref{gom-lem} and \ref{om-gp-lem}.
\end{proof}

The homogeneous symmetric functions $\iF_y$ and $\Ffpf_z$ obtained by
setting $\beta=0$ in $\iG_y$ and $\Gfpf_z$ are the \emph{(FPF)-involution Stanley symmetric functions}
studied in \cite{HMP1,HMP3,HMP4,HMP5}.
Setting $\beta =0$ in $\GP_\lambda$, alternatively, yields the well-known \emph{Schur $P$-function} $P_\lambda$.
Theorem~\ref{lastmain-thm} with $\beta=0$ implies that $\iF_y$ and $\Ffpf_z$
are \emph{Schur-$P$-positive} in the following sense:

\begin{corollary}[{See \cite[Corollary 1.12]{HMP4} and \cite[Theorem 1.1]{HMP5}}] 
Let $y \in \II$ and $z \in \FF$.
Then 
$\iF_y = \sum_\lambda   b_{y\lambda} P_\lambda$
and
$\Ffpf_z = \sum_\mu   c_{z\mu} P_\mu$
where the sums are over strict partitions $\lambda$ of $\ellhat(y)$
and $\mu$ of $\ellfpf(z)$,
and the positive integers $b_{y\lambda}$ and $c_{z\mu}$ are defined as in Theorem~\ref{lastmain-thm}. 
\end{corollary}

Our interpretation
of the coefficients in the Schur $P$-expansion of $\Ffpf_z$ in this result
is new.
Finally, setting $\beta=0$ in Lemma~\ref{om-gp-lem} and Corollary~\ref{g*-cor}
gives the following:

\begin{corollary}
If $y \in \II$ and $z \in \FF$
and $\iF_y =\omega(\iF_y)= \iF_{y^*}$ and $\Ffpf_z =\omega(\Ffpf_z)= \Ffpf_{z^*}$.
\end{corollary}


\end{document}

%% file: preamble.tex
\usepackage{latexsym}
\usepackage{amssymb}
\usepackage{amsthm}
\usepackage{amscd}
\usepackage{amsmath}
\usepackage{mathrsfs}
\usepackage{graphicx}
\usepackage{hyperref}
\usepackage{shuffle}
\usepackage{mathtools}
\usepackage[bbgreekl]{mathbbol}
\usepackage{mathdots}
\usepackage{ytableau}

\usepackage[colorinlistoftodos]{todonotes}
\usetikzlibrary{chains,scopes}
\usetikzlibrary{shapes.geometric,positioning}

\DeclareSymbolFontAlphabet{\mathbb}{AMSb}
\DeclareSymbolFontAlphabet{\mathbbl}{bbold}

\newcommand{\SetSSYT}{\ensuremath\mathrm{SetSSYT}}

\newcommand{\SetSSMT}{\ensuremath\mathrm{SetSSMT}}
\newcommand{\SetSMT}{\ensuremath\mathrm{SetSMT}}

\newcommand{\WkSetSSMT}{\ensuremath\mathrm{WeakSet}\mathrm{SSMT}}

\def\QSym{\textsf{QSym}}

\numberwithin{equation}{section}
\allowdisplaybreaks[1]

\theoremstyle{definition}
\newtheorem* {theorem*}{Theorem}
\newtheorem* {conjecture*}{Conjecture}
\newtheorem{theorem}{Theorem}[section]

\theoremstyle{definition}

\newtheorem* {example*}{Example}

\newtheorem{lemma}[theorem]{Lemma}
\theoremstyle{definition}
\newtheorem{definition}[theorem]{Definition}
\theoremstyle{definition}

\newtheorem{conjecture}[theorem]{Conjecture}
\newtheorem{proposition}[theorem]{Proposition}
\newtheorem{corollary}[theorem]{Corollary}

\newtheorem* {remark*}{Remark}
\newtheorem{remark}[theorem]{Remark}
\theoremstyle{definition}
\newtheorem {example}[theorem]{Example}
\theoremstyle{definition}

\theoremstyle{definition}

\theoremstyle{definition}

\def\modu{\ (\mathrm{mod}\ }
\def\wh{\widehat}
\def\({\left(}
\def\){\right)}

\newcommand{\CC}{\mathbb{C}}

\newcommand{\cR}{\mathcal{R}}

\newcommand{\cD}{\mathcal{D}}

\def\cX{\mathcal{X}}

\def\NN{\mathbb{N}}

\def\CC{\mathbb{C}}

\def\ZZ{\mathbb{Z}}

\def\GL{\textsf{GL}}

\def\spanning{\textnormal{-span}}

\def\fk{\mathfrak}

\def\barr{\begin{array}}
\def\earr{\end{array}}
\def\ba{\begin{aligned}}
\def\ea{\end{aligned}}
\def\be{\begin{equation}}
\def\ee{\end{equation}}

\def\qquand{\qquad\text{and}\qquad}

\def\cH{\mathcal H}

\def\PP{\mathbb{P}}
\def\MM{\mathbb{M}}

\def\ben{\begin{enumerate}}
\def\een{\end{enumerate}}

\def\fpf{{\textsf {FPF}}}

\def\Des{\mathrm{Des}}

\def\FF{\mathfrak{F}_\infty}
\def\II{\mathfrak{I}_\infty}
\def\SS{\mathfrak{S}_\infty}

\renewcommand{\O}{\textsf{O}}
\newcommand{\Sp}{\textsf{Sp}}

\renewcommand{\r}[1]{\textcolor{red}{#1}}

\newcommand{\cA}{\mathcal{A}}
\newcommand{\cB}{\mathcal{B}}
\newcommand{\cBfpf}{\mathcal{B}_\fpf}

\def\cAfpf{\cA_\fpf}
\def\cRfpf{\hat\cR_\fpf}
\def\iR{\hat\cR}
\def\iF{\hat F}

\def\GP{G\hspace{-0.2mm}P}

\def\ellhat{\hat\ell}
\def\ellfpf{\hat\ell_\fpf}

\def\Ffpf{\hat F^\fpf}

\def\r{\mathbf{r}}

\def\cT{\mathcal{T}}

\def\st{\mathfrak{st}}

\newcommand{\row}{\operatorname{row}}

\def\fsim{\mathbin{=_\Sp}}
\def\fequiv{\mathbin{\equiv_\Sp}}
\def\isim{\mathbin{=_\O}}
\def\iequiv{\mathbin{\equiv_\O}}
\def\simA{=_{\textsf{Br}}}
\def\equivA{\equiv_{\textsf{Br}}}

\def\simKK{\overset{\textsf{K}}\approx}
\def\simCK{\overset{\textsf{K}}\sim}
\def\simICK{\overset{\O}\sim}
\def\simIKK{\overset{\O}\approx}
\def\simFCK{\overset{\Sp}\sim}
\def\simFKK{\overset{\Sp}\approx}

\def\whSym{\wh{\textsf{S}}\textsf{ym}}

\newcommand{\weight}{\operatorname{wt}}

\def\cW{\mathcal{W}}

\def\row{\mathfrak{row}}
\def\col{\mathfrak{col}}
\def\swdiag{\mathfrak{d}_{\textsf{SW}}}
\def\nediag{\mathfrak{d}_{\textsf{NE}}}

\def\iarrow{\xleftarrow{\O}}
\def\farrow{\xleftarrow{\Sp}}

\def\SY{\mathrm{SD}}
\def\Y{\mathrm{D}}

\def\cHfpf{\cH_{\Sp}}
\def\iH{\cH_{\O}}

\def\iG{\GP^{\O}}
\def\Gfpf{\GP^{\Sp}}

\def\wSp{w_{\Sp}}
\def\PI{P_{\O}}
\def\QI{Q_{\O}}
\def\PF{P_{\Sp}}
\def\QF{Q_{\Sp}}

\def\unbump{\textsf{uninsert}}
\def\bump{\textsf{insert}}

\newcommand*\circled[1]{\tikz[baseline=(char.base)]{
\node[shape=circle,draw,inner sep=1](char) {#1};}}

\def\word{\mathfrak{word}}

\def\rowbt{\overset{(i,j)}{\underset{\row}{\leadsto}}}
\def\colbt{\overset{(i,j)}{\underset{\col}{\leadsto}}}

\newcommand{\cN}{\mathcal{N}_\infty}
\newcommand{\cM}{\mathcal{M}_\infty}
\def\Nil{\mathcal{U}_\infty}
\def\*{\mathbin{\hat\circ}}

\def\qm{{*}}